\documentclass{amsart}

\usepackage{amsmath,amsthm,amssymb}
\usepackage{enumitem}
\usepackage{graphicx}
\usepackage{float}
\usepackage[margin=1.5in]{geometry}
\numberwithin{equation}{section}

\theoremstyle{plain}
\newtheorem{thm}{Theorem}[section]
\newtheorem{cor}[thm]{Corollary}
\newtheorem{lem}[thm]{Lemma}
\newtheorem{prop}[thm]{Proposition}

\theoremstyle{definition}
\newtheorem{defn}[thm]{Definition}

\theoremstyle{remark}
\newtheorem{rem}[thm]{Remark}

\newtheorem*{notation}{Notation}

\usepackage{tikz}
\usetikzlibrary{arrows,shapes,trees}
\usepackage{graphicx}
\usepackage{hyperref}
\usepackage{amssymb}
\usepackage{amsfonts}
\usepackage{amsthm}
\usepackage{amsmath}
\usepackage{epstopdf}
\usepackage{mathrsfs}

\newcommand{\op}{\operatorname}
\newcommand{\N}{\mathbb N}

\newcommand{\Z}{\mathbb Z}

\newcommand{\R}{\mathbb R}

\newcommand{\W}{\mathcal W}

\newcommand{\e}{\varepsilon}

\newcommand{\wt}{\widetilde}

\newcommand{\diam}{\op{diam}}
\newcommand{\ls}{\lesssim}
\newcommand{\gs}{\gtrsim}

\newcommand{\loc}{\mathrm{loc}}
\newcommand{\supp}{\op{supp}}
\newcommand{\ch}{\check}

\newcommand{\X}{\mathcal{X}}
\newcommand{\CC}{\mathcal{C}}

\renewcommand{\t}{\tilde}
\newcommand{\la}{\lambda}

\newcommand{\p}{\partial}

\def\Xint#1{\mathchoice
{\XXint\displaystyle\textstyle{#1}}%
{\XXint\textstyle\scriptstyle{#1}}%
{\XXint\scriptstyle\scriptscriptstyle{#1}}%
{\XXint\scriptscriptstyle\scriptscriptstyle{#1}}%
\!\int}
\def\XXint#1#2#3{{\setbox0=\hbox{$#1{#2#3}{\int}$ }
\vcenter{\hbox{$#2#3$ }}\kern-.6\wd0}}

\def\dashint{\Xint-}

\title[Patterson-Sullivan construction and global leaf geometry for Anosov flows]{The Patterson-Sullivan construction and global leaf geometry for Anosov flows}
\author[C.~Butler]{Clark Butler}\address{\noindent e-mail: \rm
  \texttt{butlerphi@gmail.com}}

\begin{document}
\begin{abstract}
We give a new construction of the measure of maximal entropy for transitive Anosov flows through a method analogous to the construction of Patterson-Sullivan measures in negative curvature. In order to carry out our procedure we prove several new results concerning the global geometry of the leaves of the center-unstable foliation of an Anosov flow. We show that the universal covers of the center-unstable leaves are Gromov hyperbolic in the induced Riemannian metric and their relative Gromov boundaries canonically identify with the unstable leaves within in such a way that the Hamenst\"adt metrics on these leaves correspond to visual metrics on the relative Gromov boundary. These center-unstable leaves are then uniformized according to a technique inspired by methods of Bonk-Heinonen-Koskela which, in addition to its utility in the construction itself, also leads to rich analytic properties for these uniformized leaves such as supporting a Poincar\'e inequality. As a corollary we obtain that the fundamental group of a closed Riemannian manifold with Anosov geodesic flow must be Gromov hyperbolic.  
\end{abstract}

\maketitle

\section{Introduction}\label{sec:intro}

The Patterson-Sullivan theory of measures on the boundary at infinity associated to group actions on hyperbolic spaces has many rich interactions with the dynamics of the geodesic flow on the associated quotient space. We refer to the introduction of \cite{PPS15} for a comprehensive account of recent developments in this area. Specializing to the realm of cocompact group actions and geodesic flows on closed negatively curved Riemannian manifolds, this interplay leads to a trio of related characterizations of the measure of maximal entropy in the negatively curved setting: as a product of leaf measures equivalent to the Patterson-Sullivan measures \cite{Rob03}, as a limit in average of measures concentrated on periodic orbits \cite{Bow73}, and as a product of conditional measures that arise from pushing forward and renormalizing Lebesgue measure on unstable leaves \cite{Mar70}.  

When we transition to the general setting of Anosov flows the second and third characterizations described above still go through but the interpretation of the conditionals of the measure of maximal entropy on unstable leaves as arising from Patterson-Sullivan measures is lost. Indeed due to the poor overall understanding of the generic structure of Anosov flows in higher dimensions there may not be any group action on a universal cover for us to exploit in the first place. Nevertheless it is still desirable to have some version of the Patterson-Sullivan construction available for these flows, both to gain a new perspective on existing methods and to better handle certain issues that are delicate to deal with using current methods (such as leaf measures).  

Our starting point is the critical observation of Hamenst\"adt \cite{Ham89} which was later expanded upon by Hasselblatt \cite{Has4} that the conditionals of the Bowen-Margulis measure of maximal entropy on unstable manifolds are the Hausdorff measures of a family of metrics on these leaves, known now as \emph{Hamenst\"adt metrics}. This suggests that these conditionals could be constructed by establishing certain Ahlfors regularity properties for the Hamenst\"adt metrics first. This will be the path that we take in our construction. In a closely related construction Climenhaga-Pesin-Zelerowicz harnessed the family of dynamical metrics used to define Bowen balls in common constructions related to topological entropy to prove that the measure of maximal entropy can be built using the \emph{Caratheodory dimension structure} defined by the family as a whole \cite{CPZ19}, \cite{CPZ20}. 

Despite our dynamical motivations, the dynamical content of the proofs in this paper is quite limited. Instead we draw inspiration from the fields of coarse geometry and analysis on metric spaces for our methods. The bulk of this work is devoted to establishing that the dynamically invariant foliations associated to an Anosov flow are an ideal setting to apply these techniques. We will show that the center-unstable manifolds of an Anosov flow are Gromov hyperbolic spaces with their induced Riemannian metrics and, through a uniformization technique inspired by the work of Bonk-Heinonen-Koskela \cite{BHK}, we will show that the Riemannian metrics on these manifolds can be conformally rescaled to produce a uniform metric space satisfying a doubling condition and supporting a Poincar\'e inequality with respect to a renormalization of the Riemannian volume. These efforts were done in preparation for the development of a proper function theory for the Hamenst\"adt metrics which we unfortunately do not have the space to elaborate on in this paper. See \cite{Bu21} for some analysis in a more abstract setting that will be applied to the realm of Hamenst\"adt metrics associated to Anosov flows in a future work, and \cite{BBS21} for the work that inspired \cite{Bu21}. 

We move now to stating our results in detail. For this we will need to recall the basic structural properties of an Anosov flow and to fix some notation. These properties can be found in any basic reference, e.g. \cite{FH19}.  For $r \geq 3$ we let $M$ be a $C^{r+1}$ closed Riemannian manifold and let $f^{t}:M \rightarrow M$ be a $C^r$ transitive Anosov flow on a $C^{r+1}$ Riemannian manifold $M$. To be precise we mean a $C^{r+1}$ closed manifold $M$ equipped with a $C^r$ Riemannian metric, and we mean that $f^{t}$ is $C^r$ in the sense that its generator $\dot{f}:M \rightarrow TM$ is $C^r$ vector field on $M$. The Anosov condition means that there is a $Df^{t}$-invariant splitting $TM = E^{u} \oplus E^{c} \oplus E^{s}$ of the tangent bundle of $M$ such that $E^{c} = \op{span}(\dot{f})$ and there are constants $C_{u},C_{s} \geq 1$, $0 < a_{u} \leq A_{u}$ and $0 < a_{s} \leq A_{s}$ such that for any $v \in E^{u}$ and $t \geq 0$,
\begin{equation}\label{unstable expansion}
C_{u}^{-1}e^{a_{u}t}\|v\| \leq \|Df^{t}v\| \leq C_{u}e^{A_{u}t}\|v\|,
\end{equation}
and for any $v \in E^{s}$ and $t \geq 0$,
\begin{equation}\label{stable expansion}
C_{s}^{-1}e^{-A_{s}t}\|v\| \leq \|Df^{t}v\| \leq C_{s}e^{-a_{s}t}\|v\|.
\end{equation}
We write $E^{cu} = E^{u} \oplus E^{c}$ and $E^{cs} = E^{s} \oplus E^{c}$. In addition to these inequalities we also always assume that $\dot{f}$ is a unit vector field and the splitting $TM = E^{u} \oplus E^{c} \oplus E^{s}$ is orthogonal. Neither of these assumptions are restrictive as one can simply modify the Riemannian metric if needed so that they hold.

There are $f^{t}$-invariant foliations $\W^{*}$, $* \in \{u,c,s,cu,cs\}$ tangent to each of the subbundles $E^{*}$ with uniformly $C^r$ leaves. The foliations $\W^{*}$ are in general only H\"older continuous, however the foliation $\W^{c}$ tangent to $\dot{f}$ is $C^r$. An in-depth discussion on the topic of foliation regularity in this context (in particular a precise description of the notion of having uniformly $C^r$ leaves) can be found in \cite[Section 2]{Bu17}. 

Common examples of Anosov flows include suspension flows of hyperbolic toral automorphisms and the geodesic flows of negatively curved Riemannian manifolds. In dimension 3 the picture is much  more diverse with various surgery constructions available to produce Anosov flows on new manifolds from existing ones.

Our first theorem establishes the Gromov hyperbolicity of the center-unstable manifolds of an Anosov flow and identifies their Gromov boundary relative to a distinguished point $\ast$ at infinity with an unstable manifold (in fact any unstable manifold) within that same leaf. We refer to Section \ref{sec:expanding cone} for more details on the definitions of Gromov hyperbolicity, the relative Gromov boundary, and the Hamenst\"adt metric. 

As this theorem does not require the full structure of an Anosov flow detailed above, it is convenient to isolate the structure that is needed for applications. With $r \geq 3$ as above we let $f^{t}:M \rightarrow M$ be a $C^r$ flow on a closed $C^{r+1}$ Riemannian manifold $M$. We assume that we have a foliation $\W^{cu}$ of $M$ with uniformly $C^r$ leaves such that the tangent bundle $E^{cu}:=T\W^{cu}$ to the leaves of the foliation splits orthogonally as $E^{cu} = E^{u} \oplus E^{c}$ with $E^{c} = \op{span}(\dot{f})$ and \eqref{unstable expansion} holding on $E^{u}$. In this case we say $\W^{cu}$ is a \emph{weak expanding foliation} for $f^{t}$. See \cite[Section 2.7]{Bu17} for further details. While it's not stated as part of the definition there is also a foliation $\W^{u}$ that is tangent to $E^{u}$ and $C^r$ along $\W^{cu}$ \cite[Proposition 2.23]{Bu17}. Finally, below for $y \in \W^{cu}(x)$ and $t \geq 0$ we write $\gamma_{y}(t) = f^{t}y$ for the geodesic ray in $\W^{cu}(x)$ starting from $y$ that follows the flow of $f^{t}$.

\begin{thm}\label{thm:expanding hyperbolic}
Let $f^{t}:M \rightarrow M$ be a $C^r$ flow on a closed $C^{r+1}$ Riemannian manifold $M$ ($r \geq 3$) with weak expanding foliation $\W^{cu}$. Then the universal cover $\wt{\W}^{cu}(x)$ of each leaf $\W^{cu}(x)$ is Gromov hyperbolic in the induced metric from $M$ with hyperbolicity constant $\delta = \delta(C_{u},a_{u},A_{u})$. The map $y \rightarrow \gamma_{y}$ for $y \in \W^{u}(x)$ canonically identifies $\W^{u}(x)$ with the relative Gromov boundary $\p_{*}\wt{\W}^{cu}(x)$.
\end{thm}

Passage to the universal cover is required when the leaf $\W^{cu}(x)$ contains a periodic orbit for $f^{t}$; this is further described in the proof near the end of Section \ref{sec:expanding cone}.

Theorem \ref{thm:expanding hyperbolic} of course applies to the center-unstable foliation $\W^{cu}$ of an Anosov flow in particular, but it also applies in a number of other cases, such as when the unstable bundle $E^{u}$ of an Anosov flow $f^{t}$ has a dominated splitting into two subbundles $E^{u} = E^{uu} \oplus E^{wu}$ in which case there is always a weak expanding foliation for $f^{t}$ tangent to $E^{c} \oplus E^{uu}$ (here $E^{uu}$ is the subbundle along which the strongest expansion occurs). This situation occurs naturally in the context of suspension flows for hyperbolic toral automorphisms with multiple eigenvalues of distinct absolute value as well as for the geodesic flows of nonconstant negative curvature locally symmetric spaces. Furthermore it is stable under $C^1$ small perturbations, leading to a broad class of interesting weak expanding foliations for flows that Theorem \ref{thm:expanding hyperbolic} applies to. 

It is important to note for future work that orbit equivalences between flows $f^{t}$ and $g^{t}$ with weak expanding foliations $\W^{cu,f}$ and $\W^{cu,g}$ induce leafwise quasi-isometries between the associated Gromov hyperbolic metrics on their leaves (Proposition \ref{quasi orbit}). In other words, quasi-isometry invariants of the leaves define orbit equivalence invariants for the flows.

By applying Theorem \ref{thm:expanding hyperbolic} to the center-unstable foliation of a closed Riemannian manifold with Anosov geodesic flow we derive the following corollary.

\begin{cor}\label{Anosov metric hyperbolic}
Suppose that a smooth closed Riemannian manifold $M$ has Anosov geodesic flow. Then $\pi_{1}(M)$ is Gromov hyperbolic. 
\end{cor}

Since the Anosov condition is $C^1$ open among flows \cite[Chapter 5.4]{FH19}, there is no loss of generality in only considering smooth flows in the corollary. 

From here it is important to assume that $C_{u} = C_{s} = 1$ in \eqref{unstable expansion} and \eqref{stable expansion}, which can always be done through a modification of the Riemannian metric (Proposition \ref{Lyapunov metric}). Theorem \ref{thm:expanding hyperbolic} is proved by establishing the corresponding results in an abstract setting that we refer to as \emph{expanding cones}. These objects should be thought of as like the individual center-unstable manifolds of an Anosov flow ripped out of their original context. Theorem \ref{thm:expanding hyperbolic} leads to our construction of the measure of maximal entropy through the \emph{uniformization} procedure for expanding cones that we carry out in Section \ref{sec:uniform cone}. One should think of this procedure as mimicking the process that produces the Euclidean metric on the upper half plane model of hyperbolic space from the constant curvature $K \equiv-1$ hyperbolic metric.  

We remain in the setting of Theorem \ref{thm:expanding hyperbolic} specialized to the center-unstable foliation of an Anosov flow. For $x \in M$ we write $\X_{x} = \wt{\W}^{cu}(x)$ for the universal cover of the center-unstable leaf through $x$ and write $d_{x}$ for the Riemannian metric on $\X_{x}$ induced from $M$, and $\mu_{x}$ the Riemannian volume on $\X_{x}$ (these should be understood as lifted from $\W^{cu}(x)$ when $\W^{cu}(x)$ contains a periodic orbit).  We write $\X_{b,x}$ for the uniformization of $\X_{x}$ with respect to the height function $b$ vanishing at $x$ (with metric $d_{b,x}$) and for a parameter $\sigma > 0$ we write $\mu_{\sigma,x}$ for the uniformization of the Riemannian volume $\mu_{x}$ using $b$ (see \eqref{sigma def}). The height function $b$ has level sets the unstable manifolds inside of $\X_{x}$ and the metric $d_{b,x}$ is built by conformally rescaling the metric on $\X_{x}$ to transform it into an incomplete metric space with metric boundary $\p \X_{b,x}$ identifiable with the relative Gromov boundary $\p_{*}\X_{x}$. Lastly we write $h = h_{\mathrm{top}}(f)$ for the topological entropy of $f^{t}$ and $a = a_{u}$ for the minimal expansion rate on $E^{u}$ from \eqref{unstable expansion} used to define the Hamenst\"adt metrics and uniformization.

The next theorem shows that this construction can be used to build the conditionals of the Bowen-Margulis measure of maximal entropy on unstable leaves using a Patterson-Sullivan type construction. The noncompactness of $\X_{b,x}$ means that the ``renormalization" of the measures  $\mu_{\sigma,x}$ as $\sigma$ decreases to $h$ is a bit delicate to interpret. The identification $\p \X_{b,x} \cong \W^{u}(x)$ in the theorem statement is informal and corresponds to the identification in the statement of Theorem \ref{thm:expanding hyperbolic}. For both of these claims Section \ref{sec:max entropy} contains the relevant details. In the statement $\bar{\X}_{b,x} = \X_{b,x} \cup \p \X_{b,x}$ is the completion of the incomplete metric space $\X_{b,x}$.

\begin{thm}\label{thm:Patterson-Sullivan}
For each $x \in M$ the family of measures $\{\mu_{\sigma,x}\}_{\sigma > h}$ on $\bar{\X}_{b,x}$ can be renormalized  as $\sigma \rightarrow h$ in such a way that they converge to an $h/a$-dimensional Hausdorff measure for the Hamenst\"adt metric concentrated on $\p \X_{b,x} \cong \W^{u}(x)$. 
\end{thm}

Since the proof of Theorem \ref{thm:Patterson-Sullivan} does not use the existence of a measure of maximal entropy for $f^{t}$, as an immediate consequence of Theorem \ref{thm:Patterson-Sullivan} we obtain a new means of building the measure of maximal entropy for a topologically transitive Anosov flow. 

\begin{rem}\label{lit correction}
There are numerous misleadingly phrased citations of the papers \cite{Ham89}, \cite{Has4}of Hamenst\"adt and Hasselblatt that claim that these papers used the Hamenst\"adt metric to \emph{build} the measure of maximal entropy when these papers ``only" identified the unstable conditionals of the Bowen-Margulis measure as a constant multiple of the $h/a$-dimensional Hausdorff measure for the Hamenst\"adt metric given that this measure has already been built using Margulis' construction (this is still a significant contribution to be clear and both papers are very explicit that this is what they are doing). This issue seems to be well-known to specialists. To our knowledge Theorem \ref{thm:Patterson-Sullivan} is the first construction of the measure of maximal entropy that proceeds by directly establishing the Ahlfors regularity property of the Hamenst\"adt metrics. A similar approach to building the measure of maximal entropy can be found in the work of Climenhaga-Pesin-Zelerowicz \cite{CPZ19}, \cite{CPZ20} which uses Bowen balls instead of Hamenst\"adt balls; this requires the use of the more complex \emph{Caratheodory dimension structure} associated to the Bowen balls, but has the advantage that Bowen balls are not as ``rough" as Hamenst\"adt balls. 
\end{rem}

The technical nature of our final result and likely foreign terminology to an audience unfamiliar with analysis on metric spaces techniques makes a concise explanation trickier. We will state the result first and then explain its importance afterwards. All notation is the same as in Theorem \ref{thm:Patterson-Sullivan}.

\begin{thm}\label{thm:Anosov Poincare}
For each $x \in M$ and $\sigma > h$ the metric measure spaces $(\X_{b,x},d_{b,x},\mu_{\sigma,x})$ are doubling and support a Poincar\'e inequality with constants uniform in $x$. Furthermore the same is true of $(\bar{\X}_{b,x},d_{b,x},\mu_{\sigma,x})$.
\end{thm}

The notions of doubling and supporting a Poincar\'e inequality for a metric measure space are explained in Section \ref{sec:uniform cone}. Since $\mu_{\sigma,x}$ is extended to $\bar{\X}_{b,x}$ by $\mu_{\sigma,x}(\p \X_{b,x}) = 0$ the implication that the doubling property and Poincar\'e inequality hold on $(\bar{\X}_{b,x},d_{b,x},\mu_{\sigma,x})$ is immediate from the corresponding properties on $(\X_{b,x},d_{b,x},\mu_{\sigma,x})$ by the general theory \cite[Lemma 8.2.3]{HKST}. The conclusions cannot be extended to the critical exponent $\sigma = h$, as the discussion at the end of Section \ref{sec:max entropy} shows.

A succinct explanation of Theorem \ref{thm:Anosov Poincare} is that it implies that the potential theory associated to these uniformized center-unstable manifolds matches much of what one would expect generalizing from the potential theory of the upper half space in Euclidean space. However, whereas in the Euclidean case the metric on the boundary would simply be the Euclidean metric on a codimension one hyperplane, here the metric on $(\p \X_{b,x},d_{b,x})$ is biLipschitz equivalent to the Hamenst\"adt metric on $\W^{u}(x)$. This enables the study of function spaces like Besov spaces for the Hamenst\"adt metric through the analysis of boundary values of Sobolev-type functions on $\X_{b,x}$. We point to \cite{Bu21} for a general feel of the types of results that are possible here.

The structure of the paper is as follows: first in Section \ref{sec:expanding cone} we introduce the framework of expanding cones and prove Theorem \ref{thm:expanding hyperbolic}. In section \ref{sec:uniform cone} we detail the uniformization procedure described prior to the statement of Theorem \ref{sec:uniform cone}. Section \ref{sec:doubling} contains the necessary estimates on associated uniformized measures to obtain the doubling and Poincar\'e inequality properties in Theorem \ref{thm:Anosov Poincare}. Finally in Section \ref{sec:max entropy} we explain how these elements combine together to give our new construction of the measure of maximal entropy in this setting. For the reader only interested in Theorem \ref{thm:Patterson-Sullivan} the material of Section \ref{sec:doubling} can be skipped aside from the opening definitions. 

We thank Amie Wilkinson for commentary on an earlier draft of the paper. We also thank Zhenghe Zhang for providing the impetus to the author to pick up mathematics again.

\section{Expanding Cones}\label{sec:expanding cone}

\begin{notation} 
Throughout the paper the symbols $\asymp$, $\lesssim$, and $\gtrsim$ will be used to indicate the comparability of two quantities $V$ and $W$ up to a multiplicative constant $C \geq 1$ whose dependencies we will indicate towards the beginning of each section, with changes in the dependencies being further indicated as need be. Thus $V \ls W$ means $V \leq CW$, $V \gs W$ means $V \geq C^{-1}W$, and $V \asymp W$ means $V \ls W$ and $V \gs W$. When the constant $C$ needs to be made explicit we write $\asymp_{C}$, etc. 

For additive constants $c \geq 0$ we write $V \doteq W$ if $W-c \leq V \leq W+c$. As above we write $\doteq_{c}$ to explicitly indicate the constant. While the constants $C$ and $c$ described here may change from line to line, they should always be understood as being uniformly bounded above when the parameters they depend on are confined to a compact region of the parameter space. 

For a curve $\gamma:I \rightarrow \X$ in a Riemannian manifold $\X$ the notation $I$ for the domain of the  parametrization can denote any interval $I = (s,t)$ in $\R$, with $-\infty \leq s \leq t \leq \infty$ and each endpoint either included or not included. For $s \in I$ the notation $I_{\geq s} = I \cap [s,\infty)$ and $I_{\leq s} = I \cap (-\infty,s]$ will indicate restriction to times $t \geq s$ and $t \leq s$ respectively. We use the shorthand $\gamma_{\geq s} = \gamma|_{I_{\geq s}}$ and $\gamma_{\leq s} = \gamma|_{I_{\leq s}}$ for the restrictions of the geodesic $\gamma$ to these intervals. If $\gamma$ is rectifiable then we write $\ell(\gamma)$ for the length of $\gamma$. When $\gamma$ is a geodesic the parametrization will always be unit speed. 
\end{notation}

For a $C^{r+1}$ Riemannian manifold $\X$ we write $g$ for the $C^r$ Riemannian metric on $TM$ and write $d$ for the associated distance function. Throughout this section the regularity parameter $r$ will always satisfy $r \geq 2$, which enables us to apply the standard results of Riemannian geometry without issue.  Following the usual conventions we use the notation $\langle\,,\rangle$ and $\|\cdot\|$ for inner products and norms measured using $g$. We let $\nabla$ be the Levi-Civita connection for $g$. For a $C^1$ function $b: \X \rightarrow \R$ we write $\nabla b$ for the gradient of $b$, and if $b$ is $C^2$ then we write $\nabla^{2}b = \nabla(\nabla b)$ for the Hessian of $b$, defined to be the covariant derivative of the gradient of $b$. The Hessian of $b$ defines a symmetric bilinear form on $T\X$ by
\[
\nabla^{2}b(v,w) = \langle \nabla_{v}\nabla h,w\rangle,
\]
for tangent vectors $v$, $w \in T\X_{x}$ in the same tangent space. 

For the definition that follows we will be considering a $C^r$ unit vector field $\dot{f}:\X \rightarrow T\X$ on a complete $C^{r+1}$ Riemannian manifold $\X$ with $\dim \X \geq 2$. For such a vector field we write $f^{t}:\X \rightarrow \X$ for the $C^r$ flow it generates. We call a $C^{r+1}$ function $b:\X \rightarrow \R$ a \emph{height function} if the flow $f^{t}:\X \rightarrow \X$ generated by the gradient $\dot{f}=\nabla b$ satisfies 
\begin{equation}\label{height flow}
b(f^{t}x) = b(x) + t,
\end{equation}
for all $x \in \X$ and $t \in \R$. We will use the notations $\dot{f}$ and $\nabla b$ interchangeably for the gradient of $b$ depending on what needs to be emphasized. Observe that \eqref{height flow} implies that $b$ is surjective, that $\dot{f}$ is a unit vector field on $\X$, and that if we write $E^{c} = \op{span}(\dot{f})$ and $E^{u} = (E^{c})^{\perp}$ then there is a $Df^{t}$-invariant orthogonal decomposition $T\X = E^{u} \oplus E^{c}$ of the tangent bundle. Note that if $b$ is a height function then $b_{s} = b + s$ is also a height function for any $s \in \R$ with $\nabla b = \nabla b_{s}$.  Since $\nabla b = \dot{f}$ and $\dot{f}$ is a unit vector we conclude that $b$ is $1$-Lipschitz as a map $b:\X \rightarrow \R$; we will use this fact frequently in what follows. 

The level sets of $b$ form a codimension one $C^r$ foliation $\W^{u}$ of $\X$ tangent to $E^{u}$ with the leaf through $x \in \X$ being given by $\W^{u}(x) = b^{-1}(\{b(x)\})$.   The equation \eqref{height flow} shows that this foliation is $f^{t}$-invariant in the sense that $f^{t}(\W^{u}(x)) = \W^{u}(f^{t}x)$ for all $x \in \X$ and $t \in \R$. The flowlines of $f^{t}$ are unit speed geodesics that form a $1$-dimensional foliation $\W^{c}$ of $\X$ tangent to $E^{c}$. The flowline through $x$ will be denoted $\W^{c}(x) = \{f^{t}x\}_{t \in \R}$, with the resulting foliation $\W^{c}$ of $\X$ being $C^r$. See the discussion at the beginning of Section \ref{sec:max entropy} for more details on foliation regularity. 

\begin{defn}\label{defn: expanding cone}
An \emph{expanding cone} $(\X,b)$ is a complete $C^{r+1}$ Riemannian manifold $\X$ equipped with a $C^{r+1}$ height function $b: \X \rightarrow \R$ whose associated flow $f^{t}:\X \rightarrow \X$ has the property that there are constants $0 < a \leq A$ such that for all $v \in E^{u}$ and $t \geq 0$,
\begin{equation}\label{expansion}
e^{at}\|v\| \leq \|Df^{t}v\| \leq e^{At}\|v\|.
\end{equation}
\end{defn}

Thus in an expanding cone the gradient flow generated by the height function uniformly expands the leaves of the foliation $\W^{u}$. We will usually refer to $\X$ itself as the expanding cone whenever the height function $b$ is understood. We also point out that with $b_{s} = b+s$ as above the expanding cones $(\X,b_{s})$ for $s \in \R$ are considered to be distinct as they have different height functions.

Sometimes the inequality \eqref{expansion} is only satisfied up to a multiplicative constant. This will happen in the proof of Corollary \ref{Anosov metric hyperbolic} for instance. In this case it is useful to observe that this constant can be removed by a standard argument with a change of metric. This incurs a loss of a derivative on the metric $g$ which is typically irrelevant in practice. The proof given here is adapted from \cite[Lemma 9.1.11]{Lef25}.

\begin{prop}\label{Lyapunov metric}
Suppose that all of the conditions of Definition \ref{defn: expanding cone} are satisfied with $r \geq 3$ except that instead of \eqref{expansion} we have the weaker bound for $t \geq 0$ and $v \in E^{u}$,
\begin{equation}\label{expansion with constants}
C^{-1}e^{at}\|v\| \leq \|Df^{t}v\| \leq Ce^{At}\|v\|,
\end{equation}
for some constant $C \geq 1$. Then we can find a new $C^{r-1}$ Riemannian metric $g_{*}$ on $T\X$ such that $g_{*}$ satisfies Definition \ref{defn: expanding cone} with constants $0 < a_{*} \leq A_{*}$ for which we have for any $t \geq 0$ and $v \in E^{u}$,
\begin{equation}\label{expansion without constants}
e^{a_{*}t}\|v\|_{*} \leq \|Df^{t}v\|_{*} \leq e^{A_{*}t}\|v\|_{*},
\end{equation}
The constants $a_{*}$ and $A_{*}$ as well as the uniform comparability constants of $g_{*}$ to $g$ are determined only by $a$, $A$, and $C$.
\end{prop}

\begin{proof}
Since $C \geq 1$ we can choose $T > 0$ such that $C^{-1}e^{aT} = 2$ and define a new metric $g_{*}$ on $E^{u}$ by
\[
g_{*}(v,w) = \int_{0}^{T}g(Df^{t}v,Df^{t}w) \,dt,
\]
and then extending to $T\X$ by requiring the splitting $T\X = E^{u} \oplus E^{c}$ to be orthogonal and setting $g_{*}(\dot{f},\dot{f}) \equiv 1$. The metric $g_{*}$ is $C^{r-1}$ since $g$ and $f^{t}$ are $C^r$ and it will satisfy the requirements of Definition \ref{defn: expanding cone} provided we can verify \eqref{expansion without constants}. Then 
\[
\nabla_{\dot{f}}g_{*}|_{E^{u}} = (g \circ Df^{T} - g)|_{E^{u}}.
\]
Evaluating on a vector $v \in E^{u}$ and using \eqref{expansion with constants} gives
\[
(C^{-1}e^{aT}-1)g(v,v) \leq \nabla_{\dot{f}}g_{*}(v,v) \leq (Ce^{AT}-1)g(v,v),
\]
which implies by our choice of $T$,
\[
g(v,v) \leq \nabla_{\dot{f}}g_{*}(v,v) \leq Ce^{AT}g(v,v),
\]
and therefore by \eqref{expansion with constants},
\[
C^{-1}e^{aT}g_{*}(v,v) \leq \nabla_{\dot{f}}g_{*}(v,v) \leq C^{2}e^{2AT}g_{*}(v,v).
\]
The conclusion \eqref{expansion without constants} follows with $a_{*} = C^{-1}e^{aT}$ and $A_{*} = C^{2}e^{2AT}$.
\end{proof}

It is easy to see from the lower bound in \eqref{expansion} that $\X$ and the foliation $\W^{u}$ share the defining characteristic of center-unstable and unstable leaves for an Anosov flow,
\[
\W^{u}(x) = \{y \in \X: \lim_{t \rightarrow \infty} d(f^{-t}x,f^{-t}y) = 0\},
\]
and for $t \geq 0$ and any $x,y \in \X$
\[
d(f^{-t}x,f^{-t}y) \leq C,
\]
for some constant $C$; in fact we may take any $C > |b(x)-b(y)|$ provided $t$ is restricted to be sufficiently large. By combining \eqref{height flow} and \eqref{expansion} it is also not hard to show that each leaf $\W^{u}(x)$ of $\W^{u}$ is diffeomorphic to $\R^{n}$ and $\X$ is diffeomorphic to $\R^{n+1}$ ($n+1 = \dim \X$) via the same argument used to show that the unstable leaves of an Anosov flow are diffeomorphic to $\R^{n}$. In particular $\X$ and the leaves of $\W^{u}$ are contractible and therefore simply connected. 

In the bulk of this section (up until the discussion prior to Lemma \ref{metric comparison}) we only need to assume that the lower bound in \eqref{expansion} holds, i.e., the upper bound $A$ on the growth rate is unnecessary. Until then all implied constants in $\asymp$, $\doteq$, etc. will only depend only on $a$.

We will say that a geodesic $\gamma$ is \emph{vertical} if it is tangent to $E^{c}$, i.e., if it is part of a flowline of $f^{t}$. Since all vertical geodesics are contained within the flowlines of $f^{t}$, by uniqueness of geodesics a geodesic $\gamma$ that is not vertical will be nowhere tangent to $E^{c}$. A vertical geodesic is \emph{ascending} if $b(\gamma(t))$ increases with $t$ and \emph{descending} if $b(\gamma(t))$ decreases with $t$.  For a geodesic $\gamma: I \rightarrow \X$ we write $b_{\gamma}:I \rightarrow \R$ for the composition $b_{\gamma} = b \circ \gamma$.

For two symmetric $2$-tensors $g_{1}$ and $g_{2}$ we use the notation $g_{1} \geq g_{2}$ for the inequality $g_{1}(v,v) \geq g_{2}(v,v)$ for all $v \in T\X$, with $>$, $\leq$, $<$ defined analogously. Thus for instance the notation $\nabla^{2}h \geq cg$ for a $C^2$ function $h:\X \rightarrow \R$ and constant $c \in \R$ indicates the bound
\[
\nabla^{2}h(v,v) \geq c\|v\|^{2},
\] 
for all $v \in T\X$. The function $h$ is convex (in the Riemannian sense along geodesics) if and only if $\nabla^{2}h \geq 0$, and strictly convex if and only if $\nabla^{2}h > 0$. We also recall the notation $\iota_{v}g(w) = g(v,w)$ for the contraction of a tensor against a vector $v$.

\begin{lem}\label{convex}
We have $\iota_{\dot{f}}\nabla^{2}b = 0$, $\nabla^{2}b \geq 0$, and $\nabla^{2}b|_{E^{u}} \geq ag$. Consequently $b$ is convex. Furthermore if $\gamma:I \rightarrow \X$ is not a vertical geodesic then the function $b_{\gamma}$ is strictly convex on $I$. 
\end{lem}

\begin{proof}
Since the flowlines of $f^{t}$ are unit speed geodesics we must have $\nabla_{\dot{f}}\dot{f} = 0$, and therefore $\nabla^{2}b(\dot{f},v) = 0$ for any vector $v \in T\X$. This gives $\iota_{\dot{f}}\nabla^{2}b = 0$. Furthermore, for a vertical geodesic $\gamma: I \rightarrow Y$ and a unit vector $v \in E_{\gamma(t)}^{u}$ for some $t \in I$, the unique extension of $v$ to a $C^r$ parallel vector field $V:I \rightarrow T\X$ along $\gamma$ will consist of unit vectors and take values in the orthogonal complement $E^{u}$ of $E^{c} = \op{span}(\dot{f})$ along $\gamma$. These are standard Riemannian geometry facts immediately obtained from the observation that the expressions $\langle V,V\rangle$ and $\langle V,\dot{f}\rangle$ along $\gamma$ are taken to $0$ by applying $\nabla_{\dot{f}}$. 

Now let $x \in \X$ be a given point and $v \in E_{x}^{u}$ a given unit vector. Let $\gamma: \R \rightarrow \X$ be the vertical geodesic with $\gamma(0) = x$. Let $V: \R \rightarrow E$ be the extension of $v$ along $\gamma$ by parallel transport, so that $V(0) = v$. Since the Levi-Civita connection is torsion free we have
\[
\nabla_{\dot{f}}V - \nabla_{V}\dot{f} = \mathcal{L}_{\dot{f}}V,
\] 
where $\mathcal{L}_{\dot{f}}V$ is the Lie derivative of $V$ along $\dot{f}$. Formally we extend $V$ to a $C^r$ vector field in a neighborhood of $\gamma$ for the purpose of taking the Lie derivative, but since this derivative at $x$ only depends on the values of $V$ along $\gamma$ the extension we choose doesn't matter. We have $\nabla_{\dot{f}}V = 0$ by the parallel condition on $V$ and $\nabla_{V}\dot{f}$ only depends on the value of $V$ at the point of evaluation.  Since $f^{t}$ describes the flow by $\dot{f}$, we can thus compute $\nabla_{v}\dot{f}$ as
\[
\nabla_{v}\dot{f} = \lim_{t \rightarrow \infty} \frac{v-Df^{-t}(V(\gamma(t)))}{t}. 
\]
By taking the inner product with $v$ we get
\[
\langle\nabla_{v}\dot{f},v\rangle = \lim_{t \rightarrow \infty} \frac{1-\langle Df^{-t}(V(\gamma(t))),v\rangle}{t}.
\]  
But by \eqref{expansion} and the fact that $V(\gamma(t))$ is a unit vector,
\[
\langle Df^{-t}(V(\gamma(t))),v\rangle \leq \|Df^{-t}(V(\gamma(t)))\| \leq e^{-at}.
\]
Thus
\[
\langle\nabla_{v}\dot{f},v\rangle \geq \lim_{t \rightarrow \infty} \frac{1-e^{-at}}{t} = a,
\]
since this last limit is just the derivative of the function $t \rightarrow -e^{-at}$ at $t = 0$. This concludes the proof of the main assertion. The claim that $b$ is convex follows immediately from $\nabla^{2}b \geq 0$. For the final claim, as noted prior to the statement of the lemma if a geodesic $\gamma: I \rightarrow Y$ is not a vertical geodesic then it is nowhere tangent to $E^{c}$. Thus $\gamma'(t)$ has a nontrivial component in $E^{u}$ for all $t \in I$, hence $\nabla^{2}b(\gamma'(t),\gamma'(t)) > 0$, and therefore  $b_{\gamma}$ has positive second derivative on $I$ which implies that $b_{\gamma}$ is strictly convex.
\end{proof}

\begin{rem}\label{upper convex bound}
Assuming the upper bound in \eqref{expansion} instead and swapping $t$ with $-t$ in the proof of Lemma \ref{convex} gives
\[
\nabla^{2}b(v,v) \leq A\|v\|^{2},
\] 
for all $v \in E^{u}$. Since the Hessian of $b$ is computed locally, it's also clear that these estimates on $\nabla^{2}b|_{E^{u}}$ hold at a point $x \in \X$ if we only assume the expansion condition \eqref{expansion} in a neighborhood of $x$. 
\end{rem}

We have the following straightforward consequence of Lemma \ref{convex}. 

\begin{lem}\label{convex consequences}
Let $\gamma:I \rightarrow \X$ be a geodesic in $\X$. Then $b_{\gamma}$ achieves a minimum at at most one time $s \in I$. If such a time $s$ exists and belongs to the interior of $I$ then $b_{\gamma}$ is strictly decreasing on $I_{\leq s}$ and strictly increasing on $I_{\geq s}$. Otherwise $b_{\gamma}$ is either strictly increasing or strictly decreasing. 

If $\gamma$ is a vertical geodesic then $b_{\gamma}' \equiv 1$. If $\gamma$ is not vertical then $b_{\gamma}'$ is strictly increasing on $I$.
\end{lem} 

\begin{proof}
If $\gamma$ is a vertical geodesic then the claims of the lemma are obvious, so we can assume that $\gamma$ is not vertical. Then $b_{\gamma}$ is strictly convex by Lemma \ref{convex}. The claims then follow from the corresponding basic facts for a strictly convex function defined on an interval $I \subseteq \R$.
\end{proof}

We can refine the consequences of convexity described in Lemma \ref{convex consequences} into the following differential inequality. 

\begin{lem}\label{convex inequality}
For any geodesic $\gamma$ in $\X$ the height function restricted to $\gamma$ satisfies the differential inequality
\begin{equation}\label{diff inequality}
b_{\gamma}'' \geq a\sqrt{1-(b_{\gamma}')^{2}}.
\end{equation}
\end{lem}

\begin{proof}
Since
\[
b_{\gamma}'(t) = \langle \gamma'(t),\dot{f}(\gamma(t))\rangle,
\]
we have
\begin{align*}
b_{\gamma}''(\gamma(t)) &= \langle \gamma'(t),\nabla_{\gamma'(t)}\dot{f}(\gamma(t))\rangle \\
&= \nabla^{2}b(\gamma'(t),\gamma'(t)) \\
&\geq  a \sqrt{1-b_{\gamma}'(t)^{2}},
\end{align*}
where we've written 
\begin{equation}\label{orth decomp}
\gamma'(t) = b_{\gamma}'(t)\dot{f}(\gamma(t)) + \left(\sqrt{1-b_{\gamma}'(t)^{2}}\right)v_{t},
\end{equation}
for a unit vector $v_{t} \in E^{u}$ and used $\nabla^{2}b(\dot{f}(\gamma(t)),\dot{f}(\gamma(t))) = 0$ as well as $\nabla^{2}b(\dot{f}(\gamma(t)),v_{t}) = 0$ and $\nabla^{2}b|_{E^{u}} \geq ag$. 
\end{proof}

For $x \in \X$ we let $P_{x}:\X \rightarrow \W^{u}(x)$ be the projection onto $\W^{u}(x)$ along the flowlines of $f^{t}$, i.e., $P_{x}(y) = f^{b(x)-b(y)}y$. Note $P_{x} = P_{y}$ for $y \in \W^{u}(x)$ and $d(y,P_{x}(y)) = b(y)-b(x)$ for any pair of points $x, y \in \X$. We will use the inequality \eqref{diff inequality} to estimate the length of the images of geodesics under $P_{x}$ in a manner similar to \cite[Proposition 4.7]{HH}.

\begin{lem}\label{horosphere projection}
Let $x$ and $y$ be points in $\X$ that do not lie on the same vertical geodesic. Let $\gamma:[0,T] \rightarrow \X$ be a geodesic oriented from $x$ to $y$ and suppose that $b_{\gamma}'(0) \geq 0$. Then if we set $\eta = P_{x} \circ \gamma$ and $\theta = \sqrt{1-b_{\gamma}'(0)^{2}}$ then $\ell(\eta) \ls \theta$ and therefore $d(x,P_{x}(y)) \ls \theta$ and $d(x,y) \doteq b(y)-b(x)$.  
\end{lem}

\begin{proof}
Since $b_{\gamma}'(0) \geq 0$ we have $b_{\gamma}'(t) > 0$ for $t > 0$ by Lemma \ref{convex consequences}. Since the derivative $DP_{x}$ of $P_{x}$ at a point $y \in \X$ can be decomposed as $DP_{x} = Df^{b(x)-b(y)} \circ \pi_{y}$, where $\pi_{y}: T\X_{y} \rightarrow E_{y}^{u}$ is the orthogonal projection, we can estimate $\ell(\eta)$ by
\[
\ell(\eta) \leq \int_{0}^{T} e^{-a(b_{\gamma}(t)-b_{\gamma}(0))}\sqrt{1-b_{\gamma}'(t)^{2}}\,dt.
\]
Here we've applied $DP_{x}$ to the same decomposition of $\gamma'(t)$ as in \eqref{orth decomp}. Since $b_{\gamma}'$ is increasing we can estimate $\sqrt{1-b_{\gamma}'(t)^{2}}$ by its value at $t = 0$ to get
\[
\ell(\eta) \leq \theta\int_{0}^{T} e^{-a(b_{\gamma}(t)-b_{\gamma}(0))}\,dt.
\]
We first assume that $b_{\gamma}'(0) \geq \frac{a}{2}$. Then $b_{\gamma}'(t) \geq \frac{a}{2}$ for $0 \leq t \leq T$ which implies that 
\[
b_{\gamma}(t)-b_{\gamma}(0) \geq \frac{a}{2}t,
\]
on $[0,T]$. Thus
\begin{align*}
\ell(\eta) &\leq \theta\int_{0}^{T} e^{-\frac{a}{2}t}\,dt\\
&\leq \theta\int_{0}^{\infty} e^{-\frac{a}{2}t}\,dt \\
&\lesssim \theta,
\end{align*}
where we recall that implied constants depend only on $a$. 

In the second case $b_{\gamma}'(0) < \frac{a}{2}$ we consider the inequality \eqref{diff inequality} for $b_{\gamma}'(t) \leq \frac{1}{2}$. Then we have $b_{\gamma}''(t) \geq \frac{a}{2}$. Thus on any interval $[0,s]$ on which the inequality $b_{\gamma}'(t) \leq \frac{1}{2}$ holds we can conclude that 
\[
b_{\gamma}'(t) \geq \frac{a}{2} t + b_{\gamma}'(0) \geq \frac{a}{2} t,
\]
since $b_{\gamma}'(0) \geq 0$. This implies that $s \leq a^{-1}$ since the inequality $b_{\gamma}'(t) \leq \frac{1}{2}$ would be violated for $t > a^{-1}$ above. We conclude that $b_{\gamma}'(t) \geq \frac{a}{2}$ for $t \geq a^{-1}$. Setting $h = \min\{a^{-1},T\}$, we compute using the fact that $b_{\gamma}(t) \geq b_{\gamma}(0)$ for $t \geq 0$,
\begin{align*}
\ell(\eta) &\leq \theta\left(\int_{0}^{h}e^{-a(b_{\gamma}(t)-b_{\gamma}(0))}\,dt + \int_{h}^{T} e^{-\frac{a}{2}(t-h)-a(b_{\gamma}(h)-b_{\gamma}(0))}\,dt\right) \\
&\leq \theta\left(\int_{0}^{h}\,dt + \int_{0}^{T-h} e^{-\frac{a}{2}t}\,dt\right) \\
&\leq \theta\left(a^{-1} + \int_{0}^{\infty} e^{-\frac{a}{2}t}\,dt\right) \ls \theta.
\end{align*}

We thus conclude in both cases that $\ell(\eta) \ls \theta$ and therefore $d(x,P_{x}(y)) \ls \theta$. The conclusion $d(x,y) \doteq b(y)-b(x)$ then follows from $d(y,P_{x}(y)) = b(y)-b(x)$ and the trivial bound $\theta \leq 1$.
\end{proof}

Note that, as remarked at the end of the proof, we have $\theta \leq 1$ and therefore we always obtain $\ell(\eta) \ls 1$ and $d(x,P_{x}(y)) \ls 1$ as conclusions regardless of the value of $b_{\gamma}'(0) \geq 0$. 

We define the formal \emph{Gromov product} based at $b$ of two points $x, y \in \X$ by
\begin{equation}\label{Gromov product}
(x|y)_{b} = \frac{1}{2}(b(x)+b(y)-d(x,y)). 
\end{equation}
Since $b$ is $1$-Lipschitz we have
\begin{equation}\label{lip height}
(x|y)_{b} \leq \min\{b(x),b(y)\}.
\end{equation}

The unique minimum obtained from Lemma \ref{convex consequences} can be connected to the Gromov product $(x|y)_{b}$.

\begin{lem}\label{inf busemann} 
For $x$, $y \in \X$ let $\gamma:I \rightarrow \X$ be a geodesic joining $x$ to $y$. Let $z$ be the unique point on $\gamma$ at which the minimum for $b$ is achieved. Then
\[
b(z) = \inf_{t \in I}b(\gamma(t)) \doteq (x|y)_{b}.
\]
Letting $\gamma_{xz}$ denote the arc of $\gamma$ from $x$ to $z$ and $\gamma_{zy}$ the arc of $\gamma$ from $z$ to $y$, if $w \in \gamma_{xz}$ then
\[
d(P_{w}(x),w) \doteq 0, \; \; d(x,w) \doteq b(x)-b(w), 
\]
and if $w \in \gamma_{zy}$ then 
\[
d(P_{w}(y),w) \doteq 0, \; \; d(y,w) \doteq b(y)-b(w). 
\]
\end{lem}

\begin{proof}
We can assume without loss of generality that $b(x) \leq b(y)$. If $x$ and $y$ lie on the same vertical geodesic then $z = x$, $(x|y)_{b} = b(x)$, and for any point $w$ on $\gamma$ we have $P_{w}(x) = P_{w}(y) = w$ which gives the conclusion. So we can assume that $x$ is not on the same vertical geodesic as $y$. Let $\gamma:[0,T] \rightarrow \X$ be a geodesic joining $x$ to $y$. If $b_{\gamma}'(0) \geq 0$ then $z = x$ by Lemma \ref{convex consequences}. By Lemma \ref{horosphere projection} we have  $b(y)-b(x) \doteq d(x,y)$ which gives $(x|y)_{b} \doteq b(x)$. For the second conclusion the arc $\gamma_{xz}$ is trivial and $\gamma_{zy} = \gamma$, so we need to show for each $t \in [0,T]$ that $d(P_{\gamma(t)}(y),\gamma(t)) \doteq 0$ and $d(y,\gamma(t)) \doteq b(y)-b(\gamma(t))$. But by Lemma \ref{convex consequences} we must have $b_{\gamma}'(t) \geq 0$ for $0 \leq t \leq T$ since $b_{\gamma}'$ is increasing on $\gamma$, so the conclusion follows immediately from Lemma \ref{horosphere projection}.

Lastly suppose that $b_{\gamma}'(0) < 0$. Since $b(x) \leq b(y)$ and $x \neq y$ this implies that $b$ attains its minimum on the interior of $[0,T]$; let $t_{0}$ be the time at which the minimum is attained and $z = \gamma(t_{0})$. Let $\check{\gamma}(t) = \gamma(T-t)$ denote $\gamma$ with the reversed orientation. We then have two geodesics $\gamma_{1} = \check{\gamma}|_{[T-t_{0},T]}$ and  $\gamma_{2} = \gamma|_{[t_{0},T]}$ from $z$ to $x$ and $z$ to $y$ respectively. Since $b_{\gamma}'(t_{0}) = 0$ as a consequence of $b$ achieving its minimum on $\gamma$ at $z$, we can apply the previous case to $\gamma_{1}$ and $\gamma_{2}$ separately. The second conclusion follows immediately, while for the first since we have $d(x,z) \doteq b(x)-b(z)$ and $d(y,z) \doteq b(y)-b(z)$ we conclude that $(x|y)_{b} \doteq b(z)$ as desired.
\end{proof}

We can now derive an inequality known as a \emph{$\delta$-inequality} for Gromov products based at $b$, and consequently the Gromov hyperbolicity of $\X$. We need a quick preparatory lemma.

\begin{lem}\label{monotonicity}
Let $\gamma_{i}$ be descending geodesic segments connecting $x_{i}$ down to $y_{i}$ for $i = 1,2$. Then $(y_{1}|y_{2})_{b} \leq (x_{1}|x_{2})_{b}$.
\end{lem}

\begin{proof}
We have
\begin{align*}
2(y_{1}|y_{2})_{b} &= b(y_{1}) + b(y_{2})-d(y_{1},y_{2}) \\
&= b(x_{1}) + b(x_{2}) - d(x_{1},y_{1}) - d(x_{2},y_{2}) - d(y_{1},y_{2}) \\
&= 2(x_{1}|x_{2})_{b} + d(x_{1},x_{2}) - d(x_{1},y_{1}) - d(x_{2},y_{2}) - d(y_{1},y_{2}),
\end{align*}
Since
\[
 d(x_{1},x_{2}) \leq d(x_{1},y_{1}) + d(x_{2},y_{2}) + d(y_{1},y_{2}),
\]
by the triangle inequality, the conclusion follows.
\end{proof}

\begin{prop}\label{hyperbolicity cone}
There is a $\delta \geq 0$ depending only on $a$ such that for any $x$, $y$, $z \in \X$, 
\[
(x|z)_{b} \geq \min\{(x|y)_{b},(y|z)_{b}\} - \delta.
\]
\end{prop}

\begin{proof}
For $p\in \{x,y,z\}$ let $\gamma_{p}: \R \rightarrow \X$ be the ascending vertical geodesic parametrized as $\gamma_{p}(b(p)) = p$. By applying Lemma \ref{inf busemann} to geodesics $\gamma_{xy}$ and $\gamma_{yz}$ joining $x$ to $y$ and $y$ to $z$ respectively we obtain points $x' \in \gamma_{x}$, $y',y'' \in \gamma_{y}$, and $z' \in \gamma_{z}$ such that
\[
d(x',y') \doteq 0 \doteq d(y'',z'), 
\]
and $b(y') \doteq (x|y)_{b}$, $b(y'')\doteq (y|z)_{b}$, and finally $x'$ and $z'$ lie below $x$ and $z$ on $\gamma_{x}$ and $\gamma_{z}$ respectively. Then $(y'|y'')_{b} = \min\{b(y'),b(y'')\}$ since these two points belong to the same vertical geodesic $\gamma_{y}$. Thus
\[
(x'|z')_{b} \doteq (y'|y'')_{b} \doteq \min\{(x|y)_{b},(y|z)_{b}\},
\]
and the desired conclusion then follows from Lemma \ref{monotonicity}.
\end{proof} 

In order to use Proposition \ref{hyperbolicity cone} to show that $\X$ is Gromov hyperbolic we will need a lemma and some terminology from \cite[Chapter 2]{BS07}: a \emph{$\delta$-triple} for $\delta \geq 0$ is a triple $(a,b,c)$ of real numbers $a,b,c$ such that the two smallest numbers differ by at most $\delta$. Observe that $(a,b,c)$ is a $\delta$-triple if and only if the inequality 
\begin{equation}\label{basic delta}
c \geq \min\{a,b\} - \delta,
\end{equation}
holds for all permutations of the roles of $a$, $b$, and $c$. We will need the following claim known as the \emph{Tetrahedron Lemma}. 

\begin{lem}\cite[Lemma 2.1.4]{BS07}\label{tetrahedron}
Let $d_{12}$, $ d_{13}$, $d_{14}$, $d_{23}$, $d_{24}$, $d_{34}$ be six numbers such that the four triples $(d_{23},d_{24},d_{34})$, $(d_{13},d_{14},d_{34})$, $(d_{12},d_{14},d_{24})$, and $(d_{12},d_{13},d_{23})$ are $\delta$-triples. Then 
\[
(d_{12}+d_{34},d_{13}+d_{24},d_{14}+d_{23})
\]
is a $2\delta$-triple.
\end{lem}

Finally we state a definition of Gromov hyperbolicity in terms of Gromov products: a metric space $\X$ is \emph{$\delta$-hyperbolic} if for every $x,y,z,p \in \X$ we have
\begin{equation}\label{delta inequality}
(x|z)_{p} \geq \min \{(x|y)_{p},(y|z)_{p}\} - \delta. 
\end{equation}
Here $(x|y)_{p}$ is the ordinary Gromov product based at $p$ defined by 
\[
(x|y)_{p} = \frac{1}{2}(d(x,p)+d(y,p)-d(x,y)). 
\]
Note this definition is equivalent to saying that the triple $((x|y)_{p},(x|z)_{p},(y|z)_{p})$ is a $\delta$-triple for any choice of points $x,y,z,p \in \X$. This definition is equivalent to the Rips definition in terms of $\delta$-thinnness of geodesic triangles up to a multiplicative factor $4$ on $\delta$ \cite[Chapitre 2, Proposition 21]{GdH90}. We also remark that the text \cite{BS07} that we frequently refer to uses a stricter definition of $\delta$-hyperbolicity that is implied by \eqref{delta inequality} with $\delta/4$ in place of $\delta$ \cite[Proposition 2.1.3]{BS07}. This results in some constants being multiplied by $4$ in the claims we cite from there.

\begin{prop}\label{hyperbolicity filling}
The space $\X$ is $\delta$-hyperbolic with $\delta = \delta(a)$. 
\end{prop}

\begin{proof}
We will use the \emph{cross-difference triple} defined in \cite[Chapter 2.4]{BS07}. For a quadruple of points $Q=(x,y,z,u)$ of points in $\X$ and a fixed basepoint $o \in \X$ this triple is defined by
\[
A_{o}(Q)= ((x|y)_{o} + (z|u)_{o},(x|z)_{o}+(y|u)_{o},(x|u)_{o}+(y|z)_{o}). 
\]
The triple $A_{o}(Q)$ has the same differences among its members as the triple 
\[
A_{b}(Q) = ((x|y)_{b} + (z|u)_{b}, (x|z)_{b}+(y|u)_{b}, (x|u)_{b} + (y|z)_{b}),
\]
as a routine calculation shows for instance that
\[
(x|y)_{o} + (z|u)_{o} - (x|z)_{o}-(y|u)_{o} = (x|y)_{b} + (z|u)_{b} - (x|z)_{b}-(y|u)_{b} ,
\]
with both expressions being equal to
\[
\frac{1}{2}(-d(x,y)-d(z,u)+d(x,z)+d(y,u)). 
\]
Similar calculations give equality for the other differences. Thus $A_{o}(Q)$ is a $\delta$-triple for a given $\delta \geq 0$ if and only if $A_{b}(Q)$ is a $\delta$-triple. 

Using Lemma \ref{hyperbolicity cone} we conclude that the six numbers $(x|y)_{b}$, $(z|u)_{b}$, $(x|z)_{b}$, $(y|u)_{b}$, $(x|u)_{b}$, $(y|z)_{b}$ together satisfy the hypotheses of Lemma \ref{tetrahedron} with parameter $\delta = \delta(a)$. This implies that $A_{b}(Q)$ is a $2\delta$-triple and therefore that $A_{o}(Q)$ is a $2\delta$-triple. By \cite[Proposition 2.4.1]{BS07} this implies that inequality \eqref{delta inequality} holds for Gromov products based at $o$ in $\X$ with $2\delta$ replacing $\delta$. We conclude that $\X$ is $2\delta$-hyperbolic. 
\end{proof}

We conclude from Proposition \ref{hyperbolicity filling} that $\X$ is a proper geodesic $\delta$-hyperbolic space. We recall that the \emph{Gromov boundary} $\p \X$ of $\X$  can then be defined (for a fixed choice of basepoint $o \in \X$) as the collection of equivalence classes $\{x_{n}\}$ of sequences in $\X$ that \emph{converge to infinity} in the sense that
\[
(x_{n}|x_{m})_{o} \rightarrow \infty,
\]
as $m,n \rightarrow \infty$, with $\{x_{n}\} \sim \{y_{n}\}$ if 
\[
\lim_{n \rightarrow \infty} (x_{n}|y_{n})_{o} = \infty.
\]
Note that these notions do not depend on the choice of basepoint $o$. Since $\X$ is proper and geodesic the Gromov boundary can also be characterized as the \emph{geodesic boundary} of $\X$ whose elements are equivalence classes $[\gamma]$ of geodesic rays $\gamma:[0,\infty) \rightarrow \X$ with $\gamma \sim \eta$ if there is a constant $C$ such that $d(\gamma(t),\eta(t)) \leq C$ for all $t \geq 0$. We will freely switch between these two formulations as needed. For details we refer to \cite[Chapters 1-2]{BS07}.

By \eqref{expansion} every descending geodesic ray in $\X$ determines the same point in $\p \X$. In particular if $b(x) = b(y)$ and $\gamma_{x}$, $\gamma_{y}:[0,\infty) \rightarrow \X$ are two descending geodesic rays starting at $x$ and $y$ respectively then for all $t \geq 0$,
\[
d(\gamma_{x}(t),\gamma_{y}(t)) \leq e^{-at}d(x,y),
\]
and consequently $d(\gamma_{x}(t),\gamma_{y}(t)) \rightarrow 0$ as $t \rightarrow \infty$. We write $\ast \in \p \X$ for this common point and set $\p_{*}\X = \p X \backslash \{\ast\}$ to be the complement of $\ast$ in $\p \X$.

For a descending geodesic ray $\gamma:[0,\infty) \rightarrow Y$ and $x \in \X$ we write
\[
B_{\gamma}(x) = \lim_{t \rightarrow \infty} d(\gamma(t),x)-t,
\]
for the standard definition of the Busemann function associated to the geodesic ray $\gamma$.

\begin{lem}\label{cone busemann}
Let $\gamma:[0,\infty) \rightarrow \X$ be a descending geodesic ray. Then 
\[
B_{\gamma}(x) \doteq b(x)-b(\gamma(0)).
\]
\end{lem}

\begin{proof}
Set $y_{n} = \gamma(n)$ for each $n \in \N$, let $\eta$ be a geodesic joining $y_{0}$ to $x$, and let $z \in \eta$ be the point at which $b$ is minimized on $\eta$. Then take $n$ large enough that 
\[
b(y_{n}) = b(y_{0})-n < b(z).
\] 
Then by Lemma \ref{inf busemann} we can compute 
\[
d(x,y_{n}) \doteq d(x,z) + d(z,y_{n}) \doteq b(x)-b(z) + b(z)-b(y_{n}) = b(x)-b(y_{n}).
\]
Then
\[
d(x,y_{n}) - n = d(x,y_{n}) - b(y_{n})+b(y_{0}) \doteq b(x)-b(y_{0}).
\]
Letting $n \rightarrow \infty$ and noting $y_{0} = \gamma(0)$ completes the proof. 
\end{proof}

Thus we can treat $b$ as a Busemann function based at $\ast$ in the extended sense of Buyalo-Schroeder \cite[Chapter 3]{BS07}, as $b$ differs from any given ``true" Busemann function $B_{\gamma}$ based at $\ast$ by an additive error depending only on $a$. The only difference in our treatment is that now the artificial choice of error bound used to define a Busemann function depends directly on $a$ instead of the hyperbolicity constant $\delta = \delta(a)$, however this discrepancy can be reconciled by increasing $\delta$ to be larger than the additive constant $c = c(a)$ in Lemma \ref{cone busemann} if necessary, which we will do going forward.  

For $\xi,\zeta \in \p_{*} \X$ we define
\[
(\xi|\zeta)_{b} = \inf \liminf_{n \rightarrow \infty} (x_{n}|y_{n})_{b},
\]
with the infimum taken over all sequences $\{x_{n}\} \in \xi$, $\{y_{n}\} \in \zeta$ (here we are considering points of the Gromov boundary as equivalence classes of sequences converging to infinity). This extends the Gromov product based at $b$ to $\p_{*} \X$. We then have the following. 

\begin{lem}\cite[Lemma 3.2.4]{BS07}\label{busemann inequality}
There is a universal constant $C \geq 1$ such that for the Gromov product based at $b$, 
\begin{enumerate}
\item For any $\xi$, $\zeta \in \p_{*} \X$ and any $\{x_{n}\} \in \xi$, $\{y_{n}\} \in \zeta$ we have 
\[
(\xi |\zeta)_{b} \leq \liminf_{n \rightarrow \infty}(x_{n}|y_{n})_{b} \leq \limsup_{n \rightarrow \infty}(x_{n}|y_{n})_{b} \leq (\xi |\zeta)_{b} + C\delta,
\]
and the same holds if we replace $\zeta$ with $x \in \X$.
\item For any $\xi,\zeta,\la \in \X \cup \p_{*} \X$ we have
\[
(\xi |\la)_{b} \geq \min \{(\xi | \zeta)_{b},(\zeta | \la)_{b}\} - C\delta. 
\]
\end{enumerate}
\end{lem}  

Comparing our definition of $\delta$-hyperbolicity to that of the reference we see that we may take $C = 176$.

The Gromov boundary $\p_{*}\X$ relative to $\ast$ is defined in analogy to the standard Gromov boundary using Gromov products based at $b$ instead: we take $\p_{*}\X$ to be the collection of equivalence classes $\{x_{n}\}$ of sequences converging to infinity with respect to $b$ in the sense that
\[
(x_{n}|x_{m})_{b} \rightarrow \infty,
\]
as $m,n \rightarrow \infty$, with $\{x_{n}\} \sim \{y_{n}\}$ if 
\[
\lim_{n \rightarrow \infty} (x_{n}|y_{n})_{b} = \infty.
\]
As our previous use of the notation $\p_{*}\X = \p \X \backslash \{\ast\}$ suggests, we have the following characterization of points in $\p_{*}X$.

\begin{prop}\cite[Proposition 3.4.1]{BS07}\label{convergence Busemann}
A sequence  $\{x_{n}\}$ converges to infinity with respect to $b$ if and only if it converges to a point $\xi \in \p \X \backslash \{\ast\}$. This correspondence defines a canonical identification of $\p_{*} \X$ and $\p \X \backslash \{\ast\}$.
\end{prop}

By combining this proposition with the characterization of $\p \X$ as the geodesic boundary of $\X$ we conclude that $\p_{*} \X$ can be identified with the space of equivalence classes of geodesic rays $[\gamma]$ in $\X$ that are not within bounded distance of any descending geodesic ray. Lemma \ref{busemann inequality} implies that the Gromov product based at $b$ admits a natural extension to the boundary $\p_{*} \X$ which continues to satisfy the $\delta$-inequality \eqref{delta inequality} with a larger choice of $\delta$. 

Since $b$ is $1$-Lipschitz, no ascending geodesic ray can be at a bounded distance from a descending geodesic ray. Below we push this further to show that $\p_{*}\X$ can be identified with the space of all ascending geodesic rays in $\X$ starting from $\W^{u}(x)$ for any choice of $x \in \X$. For this we will need the following easy consequence of \eqref{diff inequality}.

\begin{lem}\label{angle converge}
Let $\gamma: I \rightarrow \X$ be a geodesic such that the right endpoint of the interval $I$ is $\infty$. Then either $b_{\gamma}'(t) = -1$ for all $t \in I$ or $b_{\gamma}'(t) \rightarrow 1$ as $t \rightarrow \infty$. 
\end{lem}

\begin{proof}
If $\gamma$ is a vertical geodesic then either $b_{\gamma}' \equiv -1$ or $b_{\gamma}' \equiv 1$ and the claim is clear. If $\gamma$ is not vertical then by Lemma \ref{convex consequences} $b_{\gamma}'$ is strictly increasing. Set $c = \lim_{t \rightarrow \infty} b_{\gamma}'(t)$. Since $b$ is $1$-Lipschitz we have $|b_{\gamma}'| \leq 1$ and therefore $|c| \leq 1$. If $|c| < 1$ then by \eqref{diff inequality} there is a $t_{0} \in I$ such that we have
\[
b_{\gamma}''(t) \geq \sqrt{1-c^{2}},
\]
for $t \geq t_{0}$. This implies that
\[
b_{\gamma}'(t) \geq b_{\gamma}'(t_{0}) + (t-t_{0})\sqrt{1-c^{2}},
\]
for large enough $t$, which contradicts $b_{\gamma}'(t) \leq 1$ as $t \rightarrow \infty$. We conclude that $|c| = 1$, and since $b_{\gamma}'$ is increasing and $|b_{\gamma}'| \leq 1$ we further deduce that $c = 1$. 
\end{proof}
 
\begin{prop}\label{vert identify}
For any $x \in \X$ the map $\gamma \rightarrow [\gamma]$ sending an ascending geodesic ray starting on $\W^{u}(x)$ to its representative in $\p_{*}\X$ is a bijection. 
\end{prop}

\begin{proof}
For injectivity let $\gamma$ and $\eta$ be two ascending geodesic rays starting on $\W^{u}(x)$. If there were some constant $C$ such that $d(\gamma(t),\eta(t)) \leq C$ for all $t \geq 0$ then by \eqref{expansion} we would have $d(\gamma(0),\eta(0)) \leq Ce^{-at}$ for all $t \geq 0$. By letting $t \rightarrow \infty$ we conclude that $\gamma(0) = \eta(0)$ and therefore $\gamma = \eta$. This confirms injectivity.

For surjectivity consider a geodesic ray $\gamma:[0,\infty) \rightarrow \X$ representing a point $\xi$ of $\p_{*}\X$. By Lemma \ref{angle converge} we must then have $b_{\gamma}'(t) \rightarrow 1$ as $t \rightarrow \infty$, as otherwise $b_{\gamma}' \equiv -1$ which implies that $\gamma$ is a descending geodesic ray and this is impossible by the discussion prior to the proposition statement. Integrating the limit relation $b_{\gamma}'(t) \rightarrow 1$ thus gives $b(\gamma(t)) \rightarrow \infty$ as $t \rightarrow \infty$. If $\gamma$ is an ascending geodesic ray then the extension of $\gamma$ to a full vertical geodesic line $\t{\gamma}:\R \rightarrow \X$ intersects $\W^{u}(x)$ in a unique point $y$ and the desired ascending geodesic ray is then the sub-ray $\eta$ of $\t{\gamma}$ starting from $y$. 

Thus we can assume $\gamma$ is not a vertical geodesic. Let $\{x_{n}\} = \{\gamma(n)\}$ for $n \in \N$, and then take $N$ large enough that $b_{\gamma}|_{[N,\infty)} \geq b(x)$ and that $b_{\gamma}$ is increasing on $[N,\infty)$. We then restrict to $n \geq N$. For $m \geq n$ the segment of $\gamma$ joining $x_{n}$ to $x_{m}$ satisfies the hypotheses of Lemma \ref{horosphere projection} and thus $d(x_{n},P_{x_{n}}(x_{m})) \doteq 0$. Set $p_{n} = P_{x}(x_{n})$ for each $n$ and observe that we then have for $m \geq n \geq N$,
\[
d(p_{m},p_{n}) \lesssim e^{-a(b(x_{n})-b(x))}.
\]
Thus $\{p_{n}\}$ defines a Cauchy sequence in $\X$ which converges to a point $p \in \W^{u}(x)$. Let $\eta$ be the ascending geodesic ray starting at $p$. Let $y_{n}$ be the corresponding point on $\eta$ such that $b(y_{n}) = b(x_{n})$. As shown above, there is a constant $C$ such that for any $m \geq n \geq N$ we have $d(x_{n},P_{x_{n}}(x_{m})) \leq C$. Since $P_{x}(P_{x_{n}}(x_{m})) = p_{m}$, if we fix $n$, let $m \rightarrow \infty$, and note that $f^{b(x_{n})-b(x)}$ provides an inverse of $P_{x}|_{\W^{u}(x_{n})}$, we conclude that $P_{x_{n}}(x_{m}) \rightarrow y_{n}$ as $m \rightarrow \infty$ and therefore $d(x_{n},y_{n}) \leq C$ for $n \geq N$. Since the sequences $\{x_{n}\}$ and $\{y_{n}\}$ have unit length spacing on $\gamma$ and $\eta$ respectively it follows that these geodesics are a bounded distance from one another and therefore $[\gamma] = [\eta]$. This establishes the desired surjectivity.
\end{proof}

By using the sharper estimate with $\theta$ obtained in Lemma \ref{horosphere projection} we can improve the conclusion of Proposition \ref{vert identify} by showing that the geodesics $\gamma$ and $\eta$ above actually converge to one another at infinity.

\begin{lem}\label{converge zero}
Suppose that $\gamma$, $\eta:[0,\infty) \rightarrow \X$ are geodesic rays determining the same point $\xi \in \p_{*} \X$. Then $d(\gamma(t),\eta(s_{t})) \rightarrow 0$ as $t \rightarrow \infty$, where $s_{t} \geq 0$ is chosen such that $b(\gamma(t)) = b(\eta(s_{t}))$ once $t$ is sufficiently large. 
\end{lem}

\begin{proof}
Since every point of $\p_{*} \X$ is represented by an ascending geodesic ray by Proposition \ref{vert identify} we can assume that $\eta$ is an ascending geodesic ray since then the general case follows from the triangle inequality. We can also assume that $\gamma$ is not vertical since otherwise either $\gamma$ is a subsegment of $\eta$ or the opposite; in either case the conclusion of the lemma follows trivially. By omitting an initial segment of $\gamma$ and reparametrizing we can assume that $b_{\gamma}'(0) \geq 0$. By a further initial segment omission and reparametrization (for both $\gamma$ and $\eta$ if needed) we can then assume that $\gamma$ and $\eta$ start from the same height, i.e., $b_{\gamma}(0) = b_{\eta}(0)$. Then for each $t \geq 0$ there is an $s_{t} \geq 0$ such that $b_{\gamma}(t) = b_{\eta}(s_{t})$. 

By Lemma \ref{angle converge} we have $b_{\gamma}'(t) \rightarrow 1$ as $t \rightarrow \infty$. Thus if we define $\theta(t) = \sqrt{1-(b_{\gamma}'(t))^{2}}$ as in Lemma \ref{horosphere projection} then $\theta(t) \rightarrow 0$. For a fixed $T \geq 0$ if we consider $t \geq T$ then the projections $x_{t,T} = P_{\gamma(T)}(\gamma(t))$ satisfy $d(\gamma(T),x_{t,T}) \ls \theta(T)$ and as $t \rightarrow \infty$ we have $x_{t,T} \rightarrow \eta(s_{T})$ by the arguments used in the proof of Proposition \ref{vert identify}. Thus $d(\gamma(T),\eta(s_{T})) \ls \theta(T)$ which implies the desired claim since $\theta(T) \rightarrow 0$ as $T \rightarrow \infty$.
\end{proof}

We record the following extension of Lemma \ref{inf busemann} to geodesics that potentially have endpoints in $\p_{*} \X$. The generalization proceeds via a routine limiting argument. For $x \in \X$ we extend the projection $P_{x}:\X \rightarrow \W^{u}(x)$ to $\p_{*} \X$ by setting $P_{x}(\zeta)$ to be the unique intersection point with $\W^{u}(x)$ of the vertical geodesic containing the ascending geodesic ray defining $\zeta \in \p_{*}\X$ from Proposition \ref{vert identify}. The proof of Proposition \ref{vert identify} shows that this extension is continuous in the sense that if $\gamma:[0,\infty) \rightarrow \X$ is a geodesic ray with endpoint $\zeta$ then $P_{x}(\gamma(t)) \rightarrow P_{x}(\zeta)$ as $t \rightarrow \infty$.

\begin{lem}\label{infinite inf busemann}
For $x$, $y \in \X \cup \p_{*} \X$ let $\gamma:I \rightarrow \X$ be a geodesic joining $x$ to $y$. Let $z$ be the unique point on $\gamma$ at which the minimum for $b$ is achieved. Then
\[
b(z) = \inf_{t \in I}b(\gamma(t)) \doteq (x|y)_{b}.
\]
Letting $\gamma_{xz}$ denote the arc of $\gamma$ from $x$ to $z$ and $\gamma_{zy}$ the arc of $\gamma$ from $z$ to $y$, if $w \in \gamma_{xz}$ then
\[
d(P_{w}(x),w) \doteq 0, \; \; d(x,w) \doteq b(x)-b(w), 
\]
and if $w \in \gamma_{zy}$ then 
\[
d(P_{w}(y),w) \doteq 0, \; \; d(y,w) \doteq b(y)-b(w). 
\]
\end{lem}

\begin{proof}
We prove the extension when $x$, $y \in \p_{*} \X$; the other cases are analogous with minor adjustments. In this case $\gamma: \R \rightarrow \X$ is defined for all times. For $n \in \N$ we set $x_{n} = \gamma(-n)$, $y_{n} = \gamma(n)$, $\gamma_{n}=\gamma|_{[-n,n]}$. Letting $z$ be the unique minimum of $b$ on $\gamma$, for $n$ large enough we will have $z \in \gamma_{n}$ and $z$ will be the unique minimum of $b$ on $\gamma_{n}$. The claims of Lemma \ref{inf busemann} thus hold for $\gamma_{n}$ with reference minimum point $z$. Since $(x|y)_{b} \doteq (x_{n}|y_{n})_{b}$ for $n$ sufficiently large by Lemma \ref{busemann inequality}, we conclude that $b(z) \doteq (x|y)_{b}$ which gives the first statement. 

For the second pair of statements it's easy to see that these follow from the corresponding statements in Lemma \ref{inf busemann} provided one knows that the projections $P_{w}(x_{n})$ and $P_{w}(y_{n})$ converge to $P_{w}(x)$ and $P_{w}(y)$ as $n \rightarrow \infty$. But this is precisely what we established in the course of proving surjectivity in Proposition \ref{vert identify}. 
\end{proof}

%\begin{rem}\label{bound inf}
%Lemma \ref{infinite inf busemann} admits the following immediate corollary: if $w$ is a point on $\gamma$ given as in the lemma statement and $b(w) \leq (x|y)_{b} + c$ for some $c \geq 0$ then $d(w,z) \leq c'$ for a constant $c' = c'(c,a)$ depending only on $c$ and $a$.  
%\end{rem}

A map $F: \X_{1}\rightarrow \X_{2}$ between metric spaces $\X_{i}$ is $(K,C)$-quasi-isometric if for some $K \geq 1$ and $C \geq 0$ we have for any $x$, $y \in \X_{1}$,
\begin{equation}\label{quasi inequality}
K^{-1}d_{2}(x,y) - C \leq d_{1}(F(x),F(y)) \leq Kd_{2}(x,y) + C.
\end{equation}
If in addition for each $y \in \X_{2}$ that there is some $x \in \X_{1}$ such that $d_{2}(F(x),y) \leq K$ then we say that $F$ is a $(K,C)$-quasi-isometry, or simply a quasi-isometry. The case $C = 0$ corresponds to biLipschitz maps. Gromov hyperbolicity is a quasi-isometry invariant of geodesic metric spaces in the sense that if $F: \X_{1} \rightarrow \X_{2}$ is a $(K,C)$-quasi-isometry and $\X_{1}$ is $\delta$-hyperbolic then $\X_{2}$ is $\delta'$-hyperbolic with $\delta'$ depending only on $\delta$, $K$, and $C$. See for instance \cite[Chapter 4]{BS07}. 

Proposition \ref{Lyapunov metric} entails swapping a Riemannian metric $g_{1}$ on the manifold $\X$ for a Riemannian metric $g_{2}$ such that there is a constant $L = L(C,a,A)$ for which $g_{1} \asymp_{L} g_{2}$. Letting $d_{1}$ and $d_{2}$ denote the associated metrics on $\X$, the identity map $(\X,d_{1}) \rightarrow (\X,d_{2})$ is then $L$-biLipschitz. Since we show in Proposition \ref{hyperbolicity filling} that $(\X,d_{2})$ is $\delta$-hyperbolic with $\delta = \delta(a_{*})$ in the notation of Proposition \ref{Lyapunov metric}, we conclude that the original Riemannian manifold $(\X,d_{1})$ is also $\delta'$-hyperbolic with $\delta' = \delta'(a,A,C)$. Since $(\X,d_{1})$ and $(\X,d_{2})$ have the same vertical geodesics the conclusion of Proposition \ref{vert identify} applies in exactly the same way to $(\X,d_{1})$ as it does to $(\X,d_{2})$. Here we are using that quasi-isometries between geodesic hyperbolic spaces $\X_{1}:=(\X,d_{1})$ and $\X_{2}:=(\X,d_{2})$  induce a correspondence $\p_{\ast}\X_{1} \rightarrow \p_{\ast}\X_{2}$ between points of the relative Gromov boundaries provided that the distinguished point at infinity for $\X_{1}$ is taken to the distinguished point for $\X_{2}$ \cite[Chapter 4]{BS07}. In this case the assertion is obvious as a geodesic ray is descending in $\X_{1}$ if and only if it is descending in $\X_{2}$.

\begin{proof}[Proof of Theorem \ref{thm:expanding hyperbolic}] We have essentially completed the proof already; we only need to clarify some points regarding those leaves of the foliation that contain a periodic orbit. We assume we are in the setting of weak expanding foliations as described prior to the statement of Theorem \ref{thm:expanding hyperbolic}. We will use the same notation here that we introduced there. By Proposition \ref{Lyapunov metric} we can spend a derivative on the metric on $M$ to obtain a new biLipschitz equivalent Riemannian metric for which \eqref{expansion without constants} holds for possibly different $0 < a_{*} \leq A_{*}$. This takes us from $r \geq 3$ to $r \geq 2$; we will assume this has been done in what follows. By the discussion prior to the start of this proof it is sufficient to obtain the conclusions of Theorem \ref{thm:expanding hyperbolic} after this modification has been done. We thus assume this modification has already been made in what follows. 

The leaves of the foliation $\W^{u}$ for $f^{t}$ are all diffeomorphic to $\R^{n}$, $n = \dim E^{u}$, since $f^{t}$ uniformly expands this foliation. Each leaf $\W^{cu}(x)$ is either diffeomorphic to $\R^{n+1}$ (if it contains no periodic orbit for $f^{t}$) or diffeomorphic to the \emph{suspension space},
\begin{equation}\label{suspend periodic}
\W^{cu}(x) \cong \W^{u}(x) \times \R /(p,0) \sim (f^{T}p,T),
\end{equation}
where $T$ is the period of the periodic orbit inside of $\W^{cu}(x)$. This exhausts the possibilities as if $f^{-T}:\W^{u}(x) \rightarrow \W^{u}(x)$ is a map from an unstable leaf to itself with $T > 0$ then it has exactly one fixed point by the contraction mapping theorem (corresponding to the periodic orbit) and if we take the minimal $T > 0$ such that $f^{-T}(\W^{u}(x)) = \W^{u}(x)$ then we obtain the description \eqref{suspend periodic}. The positivity of this minimal $T$ follows from the compactness of $M$. 

In the case that $\W^{cu}(x)$ contains no periodic orbit we have $\wt{\W}^{cu}(x) = \W^{cu}(x)$.  We define a height function $b_{x}:\W^{cu}(x) \rightarrow \R$ by setting $b_{x}(y) = t$ if $y \in \W^{u}(f^{t}x)$. Then $b_{x}$ is a $C^{r+1}$ function on $\W^{cu}(x)$ with gradient $\nabla b_{x} = \dot{f}$ and whose level sets are the unstable leaves $\W^{u}(y)$, $y \in \W^{cu}(x)$. With this definition $(\W^{cu}(x),b_{x})$ is an expanding cone as in Definition \ref{defn: expanding cone}.

When $\W^{cu}(x)$ does have a periodic orbit the universal cover $\wt{\W}^{cu}(x) \cong \W^{u}(x) \times \R$ in this case has deck group $\Z$ with the deck transformations generated by the map $(p,t) \rightarrow (p,t+T)$ according to the description \eqref{suspend periodic} above. The induced Riemannian metric on $\W^{cu}(x)$ from $M$ lifts to $\wt{\W}^{cu}(x)$, as does the flow $f^{t}$, the unstable foliations $\W^{u}$ and the orbit foliation $\W^{c}$. The height function $b_{x}: \wt{\W}^{cu}(x)$ is then simply defined to be $b_{x}(p,t) = t$. As in the previous paragraph $(\wt{\W}^{cu}(x),b_{x})$ is then an expanding cone. We unify the notation by writing  $\X_{x} = (\wt{\W}^{cu}(x),b_{x})$ for the expanding cone in both cases. With this the first claim of Theorem \ref{thm:expanding hyperbolic} follows by applying Proposition \ref{hyperbolicity filling} to the expanding cone $\X_{x}$, and the second claim regarding the relative Gromov boundary follows from Proposition \ref{vert identify}.
\end{proof}

\begin{proof}[Proof of Corollary \ref{Anosov metric hyperbolic}]
Let $M$ be a smooth closed Riemannian manifold with Anosov geodesic flow. Write $T^{1}M$ for the unit tangent bundle to $M$ endowed with the Sasaki metric and $f^{t}:T^{1}M \rightarrow T^{1}M$ for the geodesic flow. Then Theorem \ref{thm:expanding hyperbolic} applies to the universal cover $\wt{\W}^{cu}(x)$ of each center-unstable leaf $\W^{cu}(x)$ of $f^{t}$. Let $\pi: T^{1}M \rightarrow M$ be the projection of a unit tangent vector to its basepoint and write $\mathbb{V} = \ker D\pi$ for the ``vertical" subbundle in $TT^{1}M$. By a result of Klingenberg \cite{Kli74} the stable and unstable bundles $E^{s}$ and $E^{u}$ for $f^{t}$ are each transverse to $\mathbb{V}$. This remains true if we add the geodesic spray $\dot{f}$, so $E^{cs}$ and $E^{cu}$ are each transverse to $\mathbb{V}$.

The differential $D\pi: TT^{1}M \rightarrow TM$ of the projection thus restricts to an isomorphism $D\pi_{x}: E^{cu}_{x} \rightarrow TM_{x}$ for each $x \in T^{1}M$. For each $x \in T^{1}M$ the projection $\pi_{x}: \W^{cu}(x) \rightarrow M$ is therefore a local diffeomorphism which is locally biLipschitz in the induced Riemannian metric on $\W^{cu}(x)$ and the given metric on $M$. Since $\W^{cu}(x)$ and $M$ are both connected, we conclude that $\pi_{x}$ is surjective and thus a covering map. We can thus pass to universal covers and obtain a locally (and therefore globally) biLipschitz map $\wt{\W}^{cu}(x) \rightarrow \wt{M}$ when these spaces are given their lifted Riemannian metrics. Since $\wt{\W}^{cu}(x)$ is a Gromov hyperbolic metric space by Theorem \ref{thm:expanding hyperbolic}, $\wt{M}$ is also Gromov hyperbolic. By the Milnor-\v{S}varc lemma $\pi_{1}(M)$ is quasi-isometric to $\wt{M}$ so we obtain that $\pi_{1}(M)$ is Gromov hyperbolic. 
\end{proof}

We remark that the conclusions of Theorem \ref{thm:expanding hyperbolic} also hold with an identical proof in the more general setting considered in \cite{Bu17} of continuous flows $f^{t}:M \rightarrow M$ with weak expanding foliations $\W$ such that $f^{t}$ is uniformly $C^r$ along $\W$ but $f^{t}$ might not even be differentiable on $M$ itself. We have simply chosen to avoid introducing the additional complications involved in that setting here. 

From this point forward we will also assume that the upper bound in \eqref{expansion} holds with the parameter $A$. Consequently we will extend our convention regarding the symbols $\asymp$, $\ls$, $\gs$, and $\doteq$ to allow the implied constants to depend on $A$ as well as $a$. 

We write $d_{x}^{u}$ for the induced Riemannian metric on $\W^{u}(x)$ from $\X$. With this notation we have $d_{x}^{u} = d_{y}^{u}$ for $y \in \W^{u}(x)$. Inequality \eqref{expansion} implies that we have 
\begin{equation}\label{distance growth}
e^{at}d_{x}^{u}(x,y) \leq d_{f^{t}x}^{u}(f^{t}x,f^{t}y) \leq e^{At}d_{x}^{u}(x,y),
\end{equation}
for every $x$, $y \in \X$ and $t \geq 0$, with the inequalities reversed for $t \leq 0$. We write $\rho_{x}$ for the \emph{Hamenst\"adt metric} on $\W^{u}(x)$ with parameter $a$, defined for $x$, $y \in \W^{u}(x)$ by
\begin{equation}\label{metric definition}
\rho_{x}(x,y) = \exp(-a\sup\{t \in \R: d_{f^{t}x}^{u}(f^{t}x, f^{t}y) \leq 1\}).
\end{equation}
The supremum is finite due to \eqref{distance growth} and we interpret $\exp(-\infty) = 0$. The fact that $\rho_{x}$ defines a metric under these conditions can be found in \cite{Has4}; the same proof that works in the context of Anosov flows using \eqref{expansion} also works here. As a consequence of the definition we have the conformal scaling property
\begin{equation}\label{conformal scaling}
\rho_{f^{t}x}(f^{t}x,f^{t}y) = e^{at}\rho_{x}(x,y)
\end{equation}
for every $x$, $y \in \X$ and $t \in \R$.

\begin{lem}\label{metric comparison}
For $x \in \X$, $t \in \R$ and $y,z \in \W^{u}(x)$ we have the comparisons
\begin{equation}\label{leaf control}
d(y,z) \leq d_{x}^{u}(y,z) \leq e^{Ad(y,z)}d(y,z), 
\end{equation}
of $d$ to $d_{x}^{u}$. When $d_{x}^{u}(y,z) \leq 1$ we have the comparison
\begin{equation}\label{Hamenstadt comparison 1}
d_{x}^{u}(y,z) \leq \rho_{x}(y,z) \leq d_{x}^{u}(y,z)^{a/A},
\end{equation}
of $d_{x}^{u}$ to $\rho_{x}$, and when $d_{x}^{u}(y,z) \geq 1$ we have instead
\begin{equation}\label{Hamenstadt comparison 2}
d_{x}^{u}(y,z)^{a/A} \leq \rho_{x}(y,z) \leq d_{x}^{u}(y,z)
\end{equation}
For $x$, $y \in \X$ we have
\begin{equation}\label{projection comparison}
\rho_{x}(x,f^{b(x)-b(y)}y) \leq \max\{e^{Ad(x,y)}d(x,y),e^{ad(x,y)}d(x,y)^{a/A}\}.
\end{equation}
\end{lem}

\begin{proof}
The bound $d \leq d_{x}^{u}$ in \eqref{leaf control} is trivial from the definitions. For the upper bound consider a geodesic $\gamma$ joining $y$ to $z$ in $\X$. Each point $p$ on $\gamma$ must lie within distance $d(y,z)$ of $\W^{u}(x)$ since $b$ is $1$-Lipschitz. Thus the projection $P_{x} \circ \gamma$ of $\gamma$ onto $\W^{u}(x)$ has length at most
\[
\ell(P_{x} \circ \gamma) \leq e^{Ad(y,z)}\ell(\gamma) = e^{Ad(y,z)}d(y,z),
\]
by \eqref{expansion}. 

From the definitions we have that $d_{x}^{u}(y,z) = 1$ if and only if $\rho_{x}(y,z) = 1$. For the general case set $\beta = a^{-1}\log \rho_{x}(y,z)$ so that $e^{-a\beta} = \rho_{x}(y,z)$. Then by \eqref{conformal scaling},
\[
e^{a \beta}\rho_{x}(y,z) = \rho_{x}(f^{\beta}y,f^{\beta}z) = 1,
\]
so that $d_{x}^{u}(f^{\beta}y,f^{\beta}z) = 1$. If $d_{x}^{u}(y,z) \leq 1$ (i.e. $\beta \geq 0$) then by applying \eqref{expansion} to this expression we get
\[
e^{a \beta}d_{x}^{u}(y,z) \leq 1 \leq e^{A\beta} d_{x}^{u}(y,z),
\]
which implies that
\[
d_{x}^{u}(y,z) \leq \rho_{x}(y,z) \leq d_{x}^{u}(y,z)^{a/A}.
\]
If $d_{x}^{u}(y,z) \geq 1$ then $\beta \leq 0$ and we instead have
\[
e^{A \beta}d_{x}^{u}(y,z) \leq 1 \leq e^{a\beta} d_{x}^{u}(y,z),
\]
which implies that
\[
d_{x}^{u}(y,z)^{a/A} \leq \rho_{x}(y,z) \leq d_{x}^{u}(y,z).
\]

Lastly, let $x,y\in \X$ be arbitrary. By considering a geodesic joining $x$ to $y$ and projecting it to $\W^{u}(x)$ as we did when establishing \eqref{leaf control}, we conclude that 
\[
d_{x}^{u}(x,f^{b(x)-b(y)}y) \leq e^{Ad(x,y)}d(x,y). 
\]
The final comparison \eqref{projection comparison} then follows from \eqref{Hamenstadt comparison 1} and \eqref{Hamenstadt comparison 2}. 
\end{proof}

A \emph{visual metric} on $\p_{*}\X$ with parameter $a$ is a metric $\rho$ on this space satisfying
\begin{equation}\label{define visual}
\rho(\zeta,\xi) \asymp Ke^{-a (\zeta|\xi)_{b}},
\end{equation}
uniformly over all $\zeta,\xi \in \p_{*}\X$ for some constant $K > 0$. We impose the constraint that the implied constant is only allowed to depend on $a$ and $A$. 

In the statement of the proposition below $d(y,z)$ should be interpreted as $\infty$ if either $y \in \p_{*}\X$ or $z \in \p_{*}\X$ (or both) and $y \neq z$.

\begin{prop}\label{branch estimate expand}
Suppose that $x \in \X$ and $y,z \in \X \cup \p_{*}\X$ with $d(y,z) \geq 1$. Then
\begin{equation}\label{second branch}
e^{-a(y|z)_{b}} \asymp e^{-ab(x)}\rho_{x}(P_{x}(y),P_{x}(z)). 
\end{equation}
Consequently under the identification of Proposition \ref{vert identify} the Hamenst\"adt metric $\rho_{x}$ defines a visual metric $\rho_{*,x}$ on $\p_{*}\X$ with parameter $a$. These visual metrics satisfy the relation
\begin{equation}\label{visual scaling}
\rho_{*,y} = e^{a(b(y)-b(x))}\rho_{*,x}, 
\end{equation}
for $x,y \in \X$.
\end{prop}

\begin{proof}
It suffices to prove the proposition in the case $y,z \in \X$ as the full claim then follows immediately from a limiting argument using Lemma \ref{busemann inequality}. Let $x,y,z \in \X$ be given. We will prove the estimate \eqref{second branch} first in the case that $b(y) = b(z)$ and then use this to deduce the general case afterward. Set $\beta = a^{-1}\log \rho_{y}(y,z)$ so that $\rho_{y}(y,z) = e^{-a\beta}$. Since $d(y,z) \geq 1$ we have $d^{u}_{y}(y,z) \geq 1$ which implies that $\rho_{y}(y,z) \geq 1$ by \eqref{Hamenstadt comparison 2}. Thus $\beta \leq 0$ by \eqref{metric definition}. 

We do the upper bound first.  Let $\gamma_{y}$ be the vertical geodesic joining $y$ to $f^{\beta}y$ and let $\gamma_{z}$ be the vertical geodesic joining $z$ to $f^{\beta}z$. Since 
\[
d_{f^{\beta}x}^{u}(f^{\beta}y,f^{\beta}z) = 1,
\]
we have by \eqref{leaf control} that $d(f^{\beta}y,f^{\beta}z) \leq 1$. Thus $f^{\beta}y$ and $f^{\beta}z$ can be joined by a geodesic $\eta$ of length $\ell(\eta) \leq 1$. The composite path $\gamma_{y} \cup \eta \cup \gamma_{z}$ then has length $\leq -2\beta + 1$, so we have 
\[
d(y,z) \leq -2\beta + 1.
\]
Dividing both sides by $2$ and then adding $-b(y)= -\frac{1}{2}(b(y)+b(z))$ to each side gives
\[
-(y|z)_{b} \leq -b(y) - \beta + \frac{1}{2}. 
\]
By exponentiating with parameter $a$ we then obtain
\[
e^{-a(y|z)_{b}} \lesssim e^{-ab(y)}\rho_{y}(y,z) = e^{-ab(x)}\rho_{x}(P_{x}(y),P_{x}(z)),
\]
by \eqref{conformal scaling}, which gives the upper bound.

For the lower bound let $\gamma:[0,k] \rightarrow Y$ be a geodesic joining $y$ to $z$, $k = d(y,z)$. Set $n = \lfloor k \rfloor$, then $n \geq 1$ since $k \geq 1$.   For each integer $0 \leq j \leq n$ we set $y_{j} = \gamma(j)$, and we set $y_{n+1} = \gamma(k) = z$ (note we may have $y_{n} = y_{n+1}$ if $k$ is an integer). Let $p_{j}:=P_{y}(y_{j})$ be the vertical projections of these points onto $\W^{u}(y)$. Then by the triangle inequality for $\rho_{y}$ we have
\[
\rho_{y}(y,z) \leq \sum_{j=0}^{n}\rho_{y}(p_{j},p_{j+1}).
\]
We let $1 \leq l \leq n$ be a choice of integer to be tuned later and we set $z_{j} = y_{n+1-j}$ for $0 \leq j \leq n+1$. Then since $b$ is $1$-Lipschitz and $b(y) = b(z)$ we have $b(y_{j}) \geq b(y)-j$ and $b(z_{j}) \geq b(y) - j$ for each $j$. Thus, using \eqref{conformal scaling} and \eqref{projection comparison},
\begin{align*}
\sum_{j=0}^{n}\rho_{y}(p_{j},p_{j+1}) &\leq \sum_{j=0}^{l}e^{a j}\rho_{b(y_{j})}(y_{j},f^{b(y_{j})-b(y_{j+1})}y_{j+1}) \\
&+ \sum_{j=0}^{n-l}e^{a j}\rho_{b(z_{j})}(z_{j},f^{b(z_{j})-b(z_{j+1})}z_{j+1}) \\
&\ls \sum_{j=0}^{l}e^{a j} + \sum_{j=0}^{n-l}e^{a j} \\
&\ls e^{a l} + e^{a(n-l)}.
\end{align*}
Choosing $l = \lceil\frac{n+1}{2} \rceil$ and recalling that $|k-n| \leq 1$ and $k = d(y,z)$, we conclude that
\[
\rho_{y}(y,z) \lesssim e^{a \frac{d(y,z)}{2}},
\]
Multiplying both sides by $e^{-ab(y)}$ then gives 
\[
e^{-ab(y)}\rho_{y}(y,z) \lesssim e^{-a(y|z)_{b}},
\]
which by \eqref{conformal scaling} can be rewritten 
\[
e^{-ab(x)}\rho_{x}(P_{x}(y),P_{x}(z)) \lesssim e^{-a(y|z)_{b}}, 
\]
as desired.

We now consider the general case. Without loss of generality we can assume that $b(y) \geq b(z)$. Then set $p = f^{b(z)-b(y)}y$. Thus $b(p) = b(z)$ and since $P_{x}(y) = P_{x}(p)$ we deduce from the above that 
\[
e^{-a b(x)}\rho_{x}(P_{x}(y),P_{x}(z)) \asymp e^{-a(p|z)_{b}}. 
\]
In order to handle the right side we apply Lemma \ref{inf busemann} to a geodesic $\gamma$ from $y$ to $z$, making the observation that in traveling from $y$ to $z$ along $\gamma$ the function $b$ cannot attain its minimum before the curve reaches the height $b(z)$. Thus if we let $w \in \gamma$ be the first point on $\gamma$ traveling from $y$ such that $b(w) = b(z)$ then $d(w,p) \doteq 0$. Therefore 
\[
(p|z)_{b} \doteq (w|z)_{b} \doteq (y|z)_{b},
\]
with the second approximate equality following from Lemma \ref{inf busemann} and the fact that the segment of $\gamma$ from $w$ to $z$ has the same point of minimal height as the full geodesic $\gamma$. The inequality \eqref{second branch} follows. The final conclusion regarding visual metrics (with $K = e^{-ab(x)}$ in \eqref{define visual}) then follows from the fact that if $\gamma$ and $\eta$ are distinct ascending geodesic rays starting from $\W^{u}(x)$ representing points $\xi$, $\zeta \in \p_{*}\X$ then we eventually have $d(\gamma(t),\eta(t)) \geq 1$ for $t$ large enough and therefore we can take $t\rightarrow \infty$ in \eqref{second branch} with $y = \gamma(t)$ and $z = \eta(t)$. The relation \eqref{visual scaling} is immediate from \eqref{conformal scaling} since for $x,y \in \X$ the chain of identifications $\W^{u}(x) \rightarrow \p_{*}\X \rightarrow \W^{u}(y)$ is simply the flow $f^{b(y)-b(x)}:\W^{u}(x) \rightarrow \W^{u}(y)$.
\end{proof}

We end this section by briefly discussing a notion of equivalence between expanding cones that corresponds to structural stability in the case of Anosov flows. We consider two expanding cones $f^{t}_{i}:\X_{i} \rightarrow \X_{i}$; we use the subscript $i = 1,2$ to denote objects corresponding to each expanding cone. A homeomorphism $\varphi: \X_{1} \rightarrow \X_{2}$ is an \emph{orbit equivalence} if there is a continuous function $\alpha: \R \times \X_{1} \rightarrow \R$ such that for any $x \in \X_{1}$,
\begin{equation}\label{time change cocyle}
\varphi(f^{t}_{1}(x)) = f_{2}^{\alpha(t,x)}(\varphi(x)),
\end{equation}
and such that there is a continuous \emph{generator} $\dot{\alpha}: \X_{1} \rightarrow \R$ such that
\begin{equation}\label{integral cocycle} 
\int_{0}^{t}\dot{\alpha}(f^{s}_{1}x)\,ds = \alpha(t,x),
\end{equation}
for any $t \in \R$ and $x \in \X_{1}$. This function $\alpha$ is an \emph{additive cocycle over $f^{t}_{1}$} in the sense that
\begin{equation}\label{cocycle relation}
\alpha(s+t,x) = \alpha(t,f^{s}_{1}x) + \alpha(s,x).
\end{equation}
Note that \eqref{time change cocyle} implies that $\dot{\alpha} \neq 0$ everywhere, which implies either $\dot{\alpha} > 0$ or $\dot{\alpha} < 0$ everywhere.

The proposition below shows that an orbit equivalence between expanding cones is a quasi-isometry under mild uniformity conditions on its local behavior. These conditions are always satisfied for the orbit equivalences between expanding cones associated to Anosov flows that are induced by structural stability in the setting of Theorem \ref{thm:expanding hyperbolic}. 

\begin{prop}\label{quasi orbit}
Suppose that $\varphi: \X_{1} \rightarrow \X_{2}$ is an orbit equivalence satisfying $\dot{\alpha} > 0$ everywhere. Suppose there are constants $K_{0}$ such that if $d(x,y) \leq 1$ then $d(\varphi(x),\varphi(y)) \leq K_{0}$ and $K_{1}$ such that $\max\{\dot{\alpha},\dot{\alpha}^{-1}\} \leq K_{1}$. Then $\varphi$ is a $(K,C)$-quasi-isometry with $(K,C)$ depending only on the expansion parameters $a_{i}$ in \eqref{expansion}, $K_{0}$, and $K_{1}$. 
\end{prop}

\begin{proof}
To simplify notation, if $x \in \X_{1}$ then we write $x' = \varphi(x)$ for the corresponding point in $\X_{2}$. We use the same notation $d$ and $b$ for $d_{1}$, $d_{2}$, and $b_{1}$, $b_{2}$ since it will be clear from the arguments which one we are referring to. We set $a = \min\{a_{1},a_{2}\}$. The notation $\doteq$ and $\asymp$ in this proof refers to equality up to an additive/multiplicative constant depending only on $a$, $K_{0}$, and $K_{1}$. We can iterate the condition on $\varphi$ with the triangle inequality to obtain for any $t \geq 0$ and $x,y\in \X_{1}$,
\begin{equation}\label{iterated triangle}
d(x,y) \leq t \Rightarrow d(x',y') \leq K_{0}(t+1)
\end{equation}

The equation \eqref{time change cocyle} together with the hypothesis on $\dot{\alpha}$ implies that $\varphi$ is $K_{1}$-biLipschitz when restricted to vertical geodesics. Now let $x$, $y \in \X_{1}$ be arbitrary with $y$ not on the vertical geodesic through $x$. Let $\gamma$ be a geodesic joining $x$ to $y$ and let $z$ be the unique minimum of $b$ on $\gamma$. By Lemma \ref{inf busemann} we can find points $p$ on the vertical geodesic through $x$ below $x$ and $q$ on the vertical geodesic through $y$ below $y$ such that $d(p,z) \doteq 0$ and $d(q,z) \doteq 0$ (so that $d(p,q) \doteq 0$ as well). Then $d(x,y) \doteq d(x,p) + d(q,y)$. Since $\varphi$ is $K_{1}$-biLipschitz on vertical geodesics we have $d(x',p') \asymp d(x,p)$ and $d(y',q') \asymp d(y,q)$. Thus
\[
d(x,p) + d(q,y) \asymp d(x',p') + d(y',q').
\]

Since $\dot{\alpha} > 0$ everywhere, we have that $p'$ and $q'$ lie below $x'$ and $y'$ on their respective vertical geodesics. By \eqref{iterated triangle} we have $d(p',q') \doteq 0$ since $d(p,q) \doteq 0$. By Lemma \ref{monotonicity} and $d(p',q') \doteq 0$ we have 
\[
b(p') \doteq b(q') \doteq (p'|q')_{b} \leq (x'|y')_{b}.
\]
Thus there is an additive constant $c = c(a,K_{0},K_{1})$ such that $b(p') \leq (x'|y')_{b}+c$ and the same is true for $q'$. Then $d(x',p')\doteq b(x')-b(p')$ and $d(y',q')\doteq b(y')-b(q')$, which implies that 
\begin{align*}
d(x',p') + d(y',q') &\doteq b(x')+b(y') - 2b(p') \\
&\geq b(x')+b(y') - 2(x'|y')_{b}-c \\
&= d(x',y')-c. 
\end{align*}
This gives the upper bound in \eqref{quasi inequality}. For the lower bound we simply note that $d(x',q') \doteq d(x',p')$ since $d(p',q') \doteq 0$ so that by the triangle inequality
\[
d(x',p') + d(y',q') \leq d(x',y') + c,
\]
for some additive constant $c = c(a,K_{0},K_{1})$. 
\end{proof}

\begin{rem}\label{quasisymmetry remark}
Since quasi-isometries of Gromov hyperbolic spaces induce quasisymmetries on the Gromov boundary (see \cite[Chapter 5]{BS07} for details and terminology), Proposition \ref{quasi orbit} proves that orbit equivalences induce quasisymmetries between Hamenst\"adt metrics through their identification with visual metrics on $\p_{*}\X$ via Proposition \ref{branch estimate expand}. This claim was previously established through a more direct method by the author in \cite[Proposition 6.7]{Bu17}. 
\end{rem}

\section{Uniformization}\label{sec:uniform cone}
In this section we will give a self-contained exposition of the \emph{uniformization} procedure for an expanding cone $(\X,b)$. This procedure is exposited in greater generality in a previous preprint of the author \cite{Bu20} and the work of Zhou-Saminathan-Antti \cite{ZPR25}. However in the case of expanding cones the construction can be simplified to make the ideas more transparent. 

We start by defining uniform metric spaces. Let $(\Omega,d)$ be a metric space. For a point $x \in \X$ and closed set $K$ we write
\[
d(x,K):= \inf_{y \in K}d(x,y),
\] 
for the distance from $x$ to $y$. We assume further that $\Omega$ is incomplete and write $\p \Omega = \bar{\Omega} \backslash \Omega$ for the complement of $\Omega$ inside of its completion $\bar{\Omega}$. We then write
\[
d_{\Omega}(x) = d(x,\p \Omega),
\]
for $x \in \Omega$. For a curve $\gamma:I \rightarrow \Omega$ we write $\gamma_{-}$ and $\gamma_{+}$ for the limits of $\gamma(t)$ in $\bar{\Omega}$ as $t$ converges to the left and right endpoints of $I$  respectively, assuming that these limits exist. 

\begin{defn}\label{def:uniform}For a constant $L \geq 1$ and a closed interval $I \subset \R$, a rectifiable curve $\gamma: I \rightarrow \Omega$ is \emph{$L$-uniform} if for every $t \in I$ we have
\begin{equation}\label{uniform one}
\min\{\ell(\gamma_{\leq t}),\ell(\gamma_{\geq t})\} \leq L d_{\Omega}(\gamma(t)),
\end{equation}
and we have the inequality
\begin{equation}\label{uniform two}
\ell(\gamma) \leq Ld(\gamma_{-},\gamma_{+}),
\end{equation}
We say that the metric space $\Omega$ is \emph{$L$-uniform} if any two points in $\Omega$ can be joined by an $L$-uniform curve. 
\end{defn}

Note that the first condition \eqref{uniform one} implies the limits in the second condition exist. Also recall the notation $\gamma_{\leq t} = \gamma|_{I \leq t}$ with $I_{\leq t} = I \cap (-\infty,t]$ described at the start of Section \ref{sec:expanding cone} (with $\gamma_{\geq t}$ defined analogously). 

We will transform the expanding cone $\X$ into a uniform metric space through a conformal deformation of the Riemannian metric on $\X$. Specifically, assuming that \eqref{expansion} holds, if we let $g$ denote the Riemannian metric tensor on $\X$ then we let $\X_{b}$ be the Riemannian manifold whose Riemannian metric tensor is $g_{b} = e^{-a b}g$, i.e., we rescale the metric tensor $g_{x}$ at $x \in \X$ by $e^{-a b(x)}$. Then for a curve $\gamma:I \rightarrow \X$ we write
\[
\ell_{b}(\gamma) = \int_{0}^{\ell(\gamma)}e^{-ab(\gamma(t))}\|\gamma'(t)\|\,dt,
\]
for the length of $\gamma$ measured in $\X_{b}$. Thus lengths get exponentially shorter as height increases. Writing $d_{b}$ for the Riemannian distance on $\X_{b}$, we then have by definition
\[
d_{b}(x,y) = \inf_{\gamma}\ell_{b}(\gamma),
\]
with the infimum taken over all curves $\gamma$ joining $x$ to $y$. Since $b$ is assumed to be $C^{r+1}$, the metric tensor remains $C^r$ and thus $\X_{b}$ remains a $C^{r+1}$-Riemannian manifold. It is not hard to see that $\X_{b}$ is now incomplete with this deformation: indeed any ascending geodesic ray in $\X$ will now have finite length in $\X_{b}$. It's also worth pointing out here that the parameter $a$ is not canonically determined (any $a > 0$ satisfying \eqref{expansion} may be used) and the dependence of the uniformization $\X_{b}$ on $a$ is suppressed in the notation. 

It's convenient to denote the density used in the conformal deformation as a function by $\kappa(x) = e^{-ab(x)}$. Then the $1$-Lipschitz property of $b$ implies the following Harnack type inequality,
\begin{equation}\label{Harnack}
e^{-ad(x,y)} \leq \frac{\kappa(x)}{\kappa(y)} \leq e^{ad(x,y)}.
\end{equation}

For this section our conventions regarding $\doteq$, etc. remain the same as in the discussion of the Hamenst\"adt metric near the end of the previous section, i.e., we allow implied constants to depend only on $a$ and $A$. The key result is our first lemma, which estimates $d_{b}$ in terms of $d$ and shows that geodesics in $\X$ are uniformly biLipschitz curves in $\X_{b}$. 

\begin{lem}\label{uniform distance}
We have for any $x,y \in \X$
\begin{equation}\label{uniform estimate}
d_{b}(x,y) \asymp e^{-a(x|y)_{b}}\min\{d(x,y),1\}. 
\end{equation}
Consequently if $\gamma$ is a geodesic in $\X$ joining $x$ to $y$ then
\begin{equation}\label{uniform geodesic}
\ell_{b}(\gamma) \asymp d_{b}(x,y) 
\end{equation}
\end{lem}

\begin{proof}
Since the claim is clear if $x = y$ we can assume that $x \neq y$. For any $x$, $y \in \X$ let $\gamma$ be a geodesic joining $x$ to $y$. Then
\begin{align*}
d_{b}(x,y) &\leq \int_{\gamma} \kappa \, ds \\
&\leq e^{-ab(x)}\int_{0}^{d(x,y)} e^{a t}\, dt \\
&\leq e^{-ab(x)}a^{-1}(e^{a d(x,y)}-1).
\end{align*}
Now let $\gamma$ instead be a geodesic in $\X_{b}$ joining $x$ to $y$, which we parametrize by arclength in $\X$. Then
\begin{align*}
d_{b}(x,y) &= \int_{\gamma} \kappa \, ds \\
&\geq e^{ab(x)}\int_{0}^{\ell(\gamma)} e^{-a t}\, dt \\
&\geq e^{ab(x)}a^{-1}(1-e^{-a \ell(\gamma)}) \\
&\geq e^{ab(x)}a^{-1}(1-e^{-a d(x,y)}),
\end{align*}
where in the last line we used $d(x,y) \leq \ell(\gamma)$. We conclude that
\begin{equation}\label{proto uniform}
e^{-ab(x)}a^{-1}(e^{-a d(x,y)}-1)\leq d_{b}(x,y) \leq e^{-ab(x)}a^{-1}(e^{a d(x,y)}-1).
\end{equation}
For $0 \leq t \leq 1$ we have the inequalities 
\[
1-e^{-a t} \geq a^{-1} e^{-a}t,
\]
and 
\[
e^{a t}-1 \leq a e^{a} t,
\]
as can be verified by noting that equality holds at $t = 0$ and differentiating each side. Thus for $d(x,y) \leq 1$ \eqref{proto uniform} implies that
\[
d_{b}(x,y) \lesssim e^{-ab(x)}d(x,y) \asymp e^{-a(x|y)_{b}}d(x,y),
\]
noting that $(x|y)_{b} \doteq b(x)$. To bound $d_{b}(x,y)$ from below, observe that if we have a curve $\gamma: I \rightarrow Y$ joining $x$ to $y$ then within the ball $B(x,d(x,y))$ in $\X$ there must be an initial segment $\hat{\gamma}$ joining $x$ to the boundary $\p B(x,d(x,y))$ of the ball. We then must have
\[
\ell_{b}(\hat{\gamma}) \gtrsim e^{-a(x|y)_{b}}\ell(\hat{\gamma}) \geq e^{-a(x|y)_{b}}d(x,y). 
\]
Taking $\gamma$ to be a geodesic in $\X_{b}$ joining $x$ to $y$, the lower bound then follows since $\hat{\gamma}$ is a subsegment of $\gamma$. 

Thus we can assume for the rest of the proof that $d(x,y) \geq 1$. We will do the upper bound in \eqref{uniform estimate} first. Let $\gamma$ be a geodesic joining $x$ to $y$. We apply Lemma \ref{inf busemann} and let $\gamma_{1}$ be the arc of $\gamma$ from $x$ to the unique point $z$ at which the minimum of $b$ is achieved on $\gamma$ and let $\gamma_{2}$ be the arc of $\gamma$ from $z$ to $y$ (note $\gamma_{1}$ or $\gamma_{2}$ may be trivial if $\gamma$ is a vertical geodesic). Then using the estimates of Lemma \ref{inf busemann},
\begin{align*}
\ell_{b}(\gamma) &= \ell_{b}(\gamma_{1}) + \ell_{b}(\gamma_{2}) \\
&= \int_{\gamma_{1}}\kappa \, dt + \int_{\gamma_{2}}\kappa \, dt \\
&\asymp e^{-a (x|y)_{b}} \left(\int_{0}^{d(x,z)}e^{-a t}\, dt + \int_{0}^{d(y,z)}e^{-a t}\, dt\right) \\
&= a^{-1}e^{-a (x|y)_{b}}(2-e^{- ad(x,z)} - e^{-ad(y,z)}).
\end{align*}
It follows immediately that
\[
\ell_{b}(\gamma) \lesssim e^{-a (x|y)_{b}},
\]
and therefore the corresponding upper bound for $d_{b}$ holds. 

For the lower bound let $\gamma$ be any rectifiable path joining $x$ to $y$, parametrized by arclength in $\X$. Let $m = \lfloor \ell(\gamma)\rfloor$ and observe that $m \geq 1$ since $\ell(\gamma) \geq d(x,y) \geq 1$. We choose times $0 = t_{0},t_{1},\dots,t_{m},t_{m+1} = \ell(\gamma)$ which divide the interval $[0,\ell(\gamma)]$ into $m+1$ subintervals of equal length $r = \frac{\ell(\gamma)}{m+1}$. We have $\frac{1}{2} \leq r \leq 1$ and therefore $r \asymp 1$. Write $\gamma_{j} = \gamma|_{[t_{j},t_{j+1}]}$ for $0 \leq j \leq m$, then $\ell(\gamma_{j}) \asymp 1$ and therefore $\ell_{b}(\gamma_{j}) \asymp e^{-ab(\gamma(t_{j}))}\ell(\gamma_{j})$. 

Set $x_{j} = \gamma(t_{j})$ and $p_{j} = P_{x}(x_{j})$. Then, using $\ell(\gamma_{j}) \asymp 1$ and \eqref{projection comparison},
\begin{align*}
\rho_{x}(P_{x}(x),P_{x}(y)) &\leq \sum_{j=0}^{m}\rho_{x}(p_{j},p_{j+1}) \\
&= \sum_{j=0}^{m}e^{a (b(x)-b(x_{j}))}\rho_{x_{j}}(x_{j},f^{b(x_{j})-b(x_{j+1})}x_{j+1}) \\
&\lesssim \sum_{j=0}^{m}e^{a (b(x)-b(x_{j}))}\ell(\gamma_{j}) \\
&\lesssim \sum_{j=0}^{m}e^{ab(x)}\ell_{b}(\gamma_{j}) \\
&= e^{ab(x)}\ell_{b}(\gamma). 
\end{align*}
By Proposition \ref{branch estimate expand} we conclude that
\[
\ell_{b}(\gamma) \gtrsim e^{-a(x|y)_{b}}.
\]
Since $\gamma$ was arbitrary we obtain the lower bound in \eqref{uniform estimate}. Finally, for the comparison \eqref{uniform geodesic} the lower bound holds by definition and the upper bound follows from the fact that we used geodesics in $\X$ to obtain the upper bound in \eqref{uniform estimate}.  
\end{proof}

For a given $x \in \X$ let $\gamma_{x}:[0,\infty) \rightarrow \X$ be the ascending geodesic ray starting from $x$. One can verify directly from the estimates of Lemma \ref{uniform distance} that 
\[
\ell_{b}(\gamma_{x}) \asymp e^{-a b(x)} \asymp 1.
\]
Thus $\gamma_{x}$ has finite length in $\X_{b}$ so $\gamma_{x}(t)$ must converge to a point $\zeta$ of $\bar{\X}_{b}$ as $t \rightarrow \infty$. Clearly $\zeta$ cannot belong to $\X_{b}$ since the geodesic $\gamma_{x}$ eventually leaves every compact subset of $\X$ (and therefore of $\X_{b}$) so we have $\zeta \in \p \X_{b}$. We can thus define for each $x \in \X$ a map $\Psi_{x}:\W^{u}(x) \rightarrow \p X_{b}$ by sending a point $y \in \W^{u}(x)$ to the endpoint in $\p \X_{b}$ of the ascending geodesic ray $\gamma_{y}$.

We now show that the map $\Psi_{x}$ induces a biLipschitz identification of the Hamenst\"adt metric $\rho_{x}$ on $\W^{u}(x)$ with $\p \X_{b}$ once $\rho_{x}$ has been suitably rescaled. Consequently by Proposition \ref{vert identify} we have a natural identification of $\p \X_{b}$ with the Gromov boundary $\p_{*}\X$ relative to $\ast$. As a shorthand we write $d_{b}(x) = d_{\X_{b}}(x)$ for the distance to the boundary in $\X_{b}$ and write $\bar{x} = \Psi_{x}(x)$ for the endpoint in $\p X_{b}$ of the ascending geodesic ray $\gamma_{x}$ starting from $x$.

\begin{prop}\label{boundary identify}
There is a constant $K = K(a,A)$ such that for each $x \in \X$ there is a $K$-biLipschitz identification of metric spaces $\Psi_{x}:(\W^{u}(x),e^{-ab(x)}\rho_{x}) \rightarrow (\p \X_{b},d_{b})$. The estimates of Lemma \ref{uniform distance} extend to $\bar{\X}_{b}$ and we have
\begin{equation}\label{distance to boundary}
d_{b}(x) \asymp \ell_{b}(\gamma_{x}) \asymp d_{b}(x,\bar{x}) \asymp e^{-ab(x)},
\end{equation}
for any $x \in \X$.
\end{prop}

\begin{proof}
From the discussion prior to the proposition statement we see that $\Psi_{x}$ is well defined. Since the method of defining $\Psi_{x}$ coincides with the method used to define the identification of $\W^{u}(x)$ with $\p_{*}\X$ in Proposition \ref{vert identify}, for $y \in \W^{u}(x)$ we can consider $\Psi_{x}(y)$ both as a point of $\p \X_{b}$ and as a point in $\p_{*}\X$. Our first goal is to show that $\Psi_{x}:\W^{u}(x) \rightarrow \p \X_{b}$ is a bijection. 

For distinct points $w,z \in \W^{u}(x)$ and $n \in \N$ let $w_{n} = \gamma_{w}(n)$ and $z_{n} = \gamma_{z}(n)$. Then for $n$ large enough Lemma \ref{uniform distance} gives
\[
d_{b}(w_{n},z_{n}) \asymp e^{-a (w_{n}|z_{n})_{b}}.
\]
As $n \rightarrow \infty$ we obtain from Lemma \ref{busemann inequality} and Proposition \ref{branch estimate expand},
\[
d_{b}(\Psi_{x}(w),\Psi_{x}(z)) \asymp e^{-a (\Psi_{x}(w)|\Psi_{x}(z))_{b}} \asymp e^{-ab(x)}\rho_{x}(w,z).
\]
This implies that $\Psi_{x}$ is $K$-biLipschitz as a map from $(\W^{u}(x),e^{-ab(x)}\rho_{x})$ onto its image in $(\p \X_{b},d_{b})$ with $K = K(a,A)$. In particular $\Psi_{x}$ is injective.

For surjectivity we observe that if $\{x_{n}\}$ is a Cauchy sequence in $\X_{b}$ that does not converge to a point of $\X_{b}$ then Lemma \ref{uniform distance} implies that we must have $(x_{m}|x_{n})_{b} \rightarrow \infty$ as $m,n \rightarrow \infty$, as if there is a subsequence that these Gromov products are bounded above along then \eqref{uniform estimate} implies $\{x_{n}\}$ is a Cauchy sequence in $\X$ which converges to some point $x \in \X$, hence $x_{n} \rightarrow x$ in $\X_{b}$ as well. Then by Proposition \ref{convergence Busemann} the sequence $\{x_{n}\}$ defines a point of $\p_{*}\X$ and is thus the endpoint of an ascending geodesic ray starting from a point in $\W^{u}(x)$ by Proposition \ref{vert identify}. We conclude that $\Psi_{x}$ is surjective. The estimates of Lemma \ref{uniform distance} then extend to $\bar{\X}_{b}$ by continuity and Lemma \ref{busemann inequality}.

We have $\ell_{b}(\gamma_{x}) \asymp e^{-a b(x)}$ as noted previously. This implies that $d_{b}(x,\bar{x}) \asymp e^{-a b(x)}$. For any $\zeta \in \p \X_{b}$ we have (using the identification $\p\X_{b} \cong \p_{*}\X$ established above)
\[
d_{b}(x,\zeta) \asymp e^{-a(x|\zeta)_{b}} \gs e^{-a b(x)},
\]
by \eqref{lip height} with the natural interpretation $b(\zeta) = \infty$. Since this holds for all $\zeta \in \p \X_{b}$ the proof is complete. 
\end{proof}

The point $\ast$ is not in the closure of $\X_{b}$: in fact by Lemma \ref{uniform distance} any descending geodesic ray in $\X$ has infinite length in $\X_{b}$. In particular $\X_{b}$ has infinite diameter, in sharp contrast to the standard uniformization procedure of Bonk-Heinonen-Koskela \cite{BHK} that produces a bounded space. It only remains to establish the uniformity property, which can be done easily now that the distance to the boundary \eqref{distance to boundary} has been computed. 

\begin{prop}\label{finish uniform}
The metric space $(\X_{b},d_{b})$ is $L$-uniform with $L = L(a,A)$. More precisely, any geodesic $\gamma$ in $\X$ is an $L$-uniform curve in $\X_{b}$.
\end{prop}

\begin{proof}
Let $\gamma: I \rightarrow \X$ be any geodesic. The second property \eqref{uniform two} is an immediate consequence of \eqref{uniform geodesic}. For the first property observe that for $t \in I$ by \eqref{distance to boundary} and \eqref{uniform distance} the condition \eqref{uniform one} is equivalent to the inequality 
\begin{equation}\label{reduced uniform one}
\min\{\ell_{b}(\gamma_{\leq t}),\ell_{b}(\gamma_{\geq t})\} \leq L e^{-ab(\gamma(t))},
\end{equation}
for $L = L(a,A)$. The computation splits into two cases: when $\gamma$ is a vertical geodesic by direct computation we have $\ell_{b}(\gamma_{\geq t}) \ls e^{-a b(\gamma(t))}$ if $\gamma$ is ascending and  $\ell_{b}(\gamma_{\leq t}) \ls e^{-a b(\gamma(t))}$ if $\gamma$ is descending. In either case \eqref{reduced uniform one} is valid. 

When $\gamma$ is not vertical we use Lemma \ref{infinite inf busemann} and let $t_{0} \in I$ be the time at which the infimum of $b$ is achieved on $\gamma$. Then once again by direct computation using the estimates of that lemma for $t \in I$ we have $\ell_{b}(\gamma_{\leq t}) \ls e^{-a b(\gamma(t))}$ if $t \leq t_{0}$ and  $\ell_{b}(\gamma{\geq t}) \ls e^{-a b(\gamma(t))}$ if $t \geq t_{0}$. Thus in all cases the inequality \eqref{reduced uniform one} is satisfied. 
\end{proof}

\begin{rem}\label{sharper uniform}
Proposition \ref{finish uniform} actually provides the sharper conclusion that, up to a multiplicative constant depending only on $a$ and $A$, the controlled segment of $\gamma$ achieving the minimum in \eqref{uniform one} can be determined solely from the values of the height function $b$ on $\gamma$. 
\end{rem}

The total space $\bar{\X}_{b}$ is rarely a Riemannian manifold with boundary, with the main notable exception being when $\X$ has constant negative curvature $K \equiv -a^{2}$. In particular the Riemannian structure on $\X_{b}$ typically does not extend in a meaningful sense to $\p \X_{b}$. The constant negative curvature case $K \equiv -1$ corresponds to the relationship between the  hyperbolic metric and the Euclidean metric on the upper half space $\R^{n}_{>0}$ of $\R^{n}$ ($n \geq 2$), with a quick computation showing that the Euclidean metric on $\R^{n}_{>0}$ is the uniformization of the standard hyperbolic metric with parameter $a = 1$. 

The metric space $(\bar{\X}_{b},d_{b})$ of the uniformization metrizes the \emph{cone topology} on $\X \cup \p_{*}\X$. This topology uses Gromov products based at $b$ to define a neighborhood basis of each point of $\p_{*}\X$. While we are not aware of a previous appearance of this topology regarding relative Gromov boundaries, the cone topology is a standard tool in the context of ordinary Gromov boundaries which topologizes $\X \cup \p \X$ as a closed ball with boundary $\p \X$. It will be useful to have a more precise description of cones centered on points of $\p \X_{b} \cong \p_{*}\X$ that are adapted to the flow $f^{t}$ and to balls centered at points of the boundary. For $x \in \bar{\X}_{b}$ we write $B_{b}(x,r)$ for the ball of radius $r$ centered at $x$ in the metric $d_{b}$. In general balls with respect to the metric $d_{b}$ will be taken in $\bar{\X}_{b}$ unless otherwise noted. For $x \in \X$ and $r > 0$ define
\begin{equation}\label{cone def}
\mathcal{C}(x,r) = \bigcup_{t > 0}f^{t}(B_{\rho}(x,r)) \cup \Psi_{x}(B_{\rho}(x,r)),
\end{equation}
where $B_{\rho}(x,r)$ is the ball of radius $r$ in the Hamenst\"adt metric $\rho_{x}$ on $\W^{u}(x)$. We additionally set 
\[
\CC_{*}(x,r) = \CC(x,r) \cap \p\X_{b} =  \Psi_{x}(B_{\rho}(x,r)).
\]
The sets $\CC(x,r)$ are open in $\bar{\X}_{b}$ and are easily seen to determine a neighborhood basis of each point of $\p \X_{b}$. Similarly for $x \in \bar{\X}_{b}$ we write
\[
B_{*}(x,r) = B_{b}(x,r) \cap \p \X_{b},
\] 
for the boundary portion of the ball $B_{b}(x,r)$. 

In the following proposition we establish a key relationship between cones and balls in $\bar{\X}_{b}$. The parameter $L$ is introduced in order to allow some flexibility in the comparison.

\begin{prop}\label{ball comparison}
Let $x \in \X_{b}$ and $r > 0$  be such that $r \asymp_{L} d_{b}(x)$, $L \geq 1$. Then there are constants $L_{*} = L_{*}(a,A,L)$ and $t_{*} = t_{*}(a,A,L)$ such that
\begin{equation}\label{ball comparison equation}
\CC(f^{t_{*}}x,L_{*}^{-1}) \subset B_{b}(\bar{x},r) \subset \CC(f^{-t_{*}}x,L_{*}).
\end{equation}
\end{prop}

\begin{proof}
Let $x \in \X_{b}$ and $r \asymp_{L} d_{b}(x)$ be given. Let $\bar{x} \in \p \X_{b}$ be the endpoint of the ascending vertical geodesic $\gamma_{x}$ starting from $x$. By enlarging $L$ by an amount depending only on $a$ and $A$ we can assume that $r \asymp_{L} e^{-ab(x)}$ as well using \eqref{distance to boundary}. By Proposition \ref{boundary identify} there is a $K \geq 1$ with $K = K(a,A)$ such that
\begin{equation}\label{cone pre chain}
\CC_{*}(x,K^{-1}re^{ab(x)}) \subset B_{*}(\bar{x},r) \subset \CC_{*}(x,Kre^{ab(x)}),
\end{equation}
which implies by our condition on $r$ that
\begin{equation}\label{cone chain}
\CC_{*}(x,L_{*}^{-1}) \subset B_{*}(\bar{x},r)  \subset \CC_{*}(x,L_{*}),
\end{equation}
with $L_{*} = KL$. Replacing $L$ by $2L$ gives \eqref{cone chain} with $L$ replaced by $2L$ and $r$ replaced by $r/2$. To simplify notation we assume this replacement has already been done in $L_{*}$ (setting $L_{*} = 2KL$ now) and thus \eqref{cone chain} holds with $r/2$ and $r$. 

In order to establish the left side inclusion in \eqref{ball comparison equation} it thus suffices by the triangle inequality to show that there is a $t_{*} = t_{*}(a,A,L) \geq 0$ such that if $y \in B_{\rho}(x,L_{*}^{-1})$ and $\gamma_{y}:\R \rightarrow \X$ is the vertical geodesic through $y$ with $\gamma_{y}(0) = y$ then $\ell_{b}(\gamma_{y}|_{[t_{*},\infty)}) < r/2$. By \eqref{distance to boundary} and the assumption on $r$ this is implied by
\[
e^{-ab(\gamma_{y}(t_{*}))} = e^{-a(b(x)+t_{*})} \leq \frac{1}{2}C L e^{-ab(x)},
\]
where $C = C(a,A)$ is the comparison constant from \eqref{distance to boundary}. This holds provided that 
\[
t_{*} \geq -a^{-1}\log \frac{1}{2}CL,
\]
and since the right side depends only on $a$, $A$, and $L$, we've established the claim. Similarly to establish the right side inclusion in \eqref{ball comparison equation} it suffices to show that there is a $t_{*} = t_{*}(a,A,L) \leq 0$ such that if $y \in B_{\rho_{x}}(x,L_{*})$ and $\gamma_{y}$ is defined as above then $\ell_{b}(\gamma_{y}|_{[t_{*},\infty)}) > Cr$ for a certain comparison constant $C = C(a,A)$ from \eqref{distance to boundary} (this is because any point $z \in B_{b}(\bar{x},r)$ must satisfy $d_{b}(z) < r$).  This is done by exactly the same argument as was done in the previous case with the inequalities reversed.
\end{proof}

When we consider balls in $\X_{b}$ that are far from the boundary relative to their radius we obtain a natural comparison to conformally rescaled balls in $\X$. The following two lemmas adapt claims of \cite{BBS20}. The first adapts \cite[Theorem 2.10]{BBS20}. 

\begin{lem}\label{sub inclusion}
There is a constant  $C_{*} = C_{*}(a,A) \geq 1$ such that for any $x \in \X$ and any $0 < r \leq \frac{1}{2}d_{b}(x)$ we have the inclusions,
\begin{equation}\label{sub inclusion equation}
B(x,C_{*}^{-1}r e^{a b(x)}) \subset B_{b}(x,r) \subset B(x,C_{*}re^{a b(x)}).
\end{equation}
\end{lem}

\begin{proof}
Let $y \in B(x,C_{*}^{-1}re^{a b(x)})$, for a constant $C_{*} \geq 1$ to be determined. Let $\gamma$ be a geodesic in $\X$ joining $x$ to $y$ and let $z \in \gamma$. Then, since $r \leq \frac{1}{2}d_{b}(x)$, we have by \eqref{distance to boundary},
\[
d(x,z) \leq C^{-1}_{*}e^{a b(x)}d_{b}(x) \leq C_{*}^{-1}C,
\]
with $C = C(a,A) \geq 1$. This then implies by \eqref{Harnack},
\[
e^{-a b(z)}\asymp_{e^{C_{*}^{-1}C a}} e^{-a b(x)}.
\]
Choosing $C_{*}$ large enough that $e^{C_{*}^{-1}C a} < 2$, we then obtain that
\[
e^{-a b(z)} \asymp_{2} e^{-a b(x)},
\]
for $z \in \gamma$. We conclude that
\begin{align*}
d_{b}(x,y) &\leq \int_{\gamma}e^{-a b(\gamma(s))}\,ds \\
&\leq 2e^{-a b(x)}d(x,y) \\
&\leq 2C_{*}^{-1}r \\
&< r,
\end{align*}
provided we take $C_{*} > 2$. This gives the inclusion on the left side of \eqref{sub inclusion equation}.

For the inclusion on the right side of \eqref{sub inclusion equation}, let $y \in B_{b}(x,r)$ and let $\eta$ be a geodesic in $\X_{b}$ connecting $x$ to $y$. For $z \in \eta$ we then have $z \in B_{b}(x,r)$ and therefore $d_{b}(z) \geq \frac{1}{2} d_{b}(x)$  by the triangle inequality since $r \leq \frac{1}{2} d_{b}(x)$. Applying \eqref{distance to boundary}, we then have
\begin{align*}
e^{-a b(z)} &\geq C^{-1}d_{b}(z) \\
&\geq \frac{1}{2}C^{-1}d_{b}(x) \\
&\geq C^{-1}e^{-a b(x)},
\end{align*}
for a constant $C = C(a,A) \geq 1$. Using this we conclude that
\[
r > d_{b}(x,y) = \int_{\eta}e^{-a b(\eta(s))}\,ds \geq C^{-1}e^{-a b(x)}d(x,y),
\]
since $\ell(\eta) \geq d(x,y)$. Choosing $C_{*}$ to be greater than the constant $C$ on the right side of this inequality, we then conclude that
\[
d(x,y) < C_{*}re^{a b(x)},
\]
which gives the right side inclusion in \eqref{sub inclusion equation}. 
\end{proof}

The second claim adapts \cite[Lemma 4.8]{BBS20} to our setting. The additional structure of an expanding cone enables us to obtain a slightly stronger conclusion. 

\begin{lem}\label{choosing center}
There is a constant $c_{0} = c_{0}(a,A) < 1$ such that for every $x \in \bar{\X}_{b}$ and every $r > 0$ we can find a ball $B_{b}(z,c_{0} r) \subset B_{b}(x,r)$ with $d_{b}(z) \geq 2 c_{0} r$. Furthermore $z$ can be chosen to lie on the same vertical geodesic as $x$. 
\end{lem}

\begin{proof}
Let $x \in \bar{\X}_{b}$ and $r > 0$ be given, and let $0 < c_{0} < 1$ be a constant to be determined. Let $\gamma: \R \rightarrow \X$ be the ascending vertical geodesic through $x$ (or with endpoint $x$ in the case $\p \X_{b}$). Since $d_{b}(\gamma(t)) \rightarrow 0$ as $t \rightarrow \infty$ and $d_{b}(\gamma(t)) \rightarrow \infty$ as $t \rightarrow -\infty$, by the intermediate value theorem we can find a point $z_{0}$ on $\gamma$ such that $d_{b}(z_{0}) = r$. 

Now consider the segment $\eta$ of $\gamma$ from $x$ to $z_{0}$, parametrized by $d_{b}$-arclength. We first assume that $\ell_{b}(\eta) \geq \frac{2}{3}r$. In this case we set $z = \eta(\frac{1}{3}r)$. Then since $\eta$ is an $L$-uniform curve (with $L = L(a,A)$) we have $d_{b}(z) \geq \frac{r}{3L}$ and 
\[
B_{b}\left(z,\frac{r}{6L}\right) \subset B_{b}\left(x,\frac{r}{3}+\frac{r}{6L}\right) \subset B_{b}(x,r). 
\]
So in this case we can use any $c_{0} \leq \frac{1}{6L}$. 

Now consider the case in which $\ell_{b}(\eta) < \frac{2}{3}r$. We then set $z = z_{0}$ and observe that 
\[
B_{b}\left(z_{0},\frac{r}{3}\right) \subset B_{b}\left(x,\ell_{b}(\eta)+ \frac{r}{3}\right) \subset B_{b}(x,r). 
\]
By construction we have $d_{b}(z) = r$. Thus in this case any $c_{0} \leq \frac{1}{3}$ will work. By combining these two cases we can then set $c_{0} = \frac{1}{6L}$, noting that $L \geq 1$. 
\end{proof}

Balls of the type $B_{b}(z,c_{0}r)$ with $d_{b}(z) \geq 2 c_{0} r$ will be referred to as \emph{subWhitney balls}. We will use the fixed notation $c_{0} = c_{0}(a,A)$ for the constant of Lemma \ref{choosing center} throughout Section \ref{sec:doubling} below.

\section{Doubling Measures}\label{sec:doubling}

In this section we prove some technical results related to rescalings of measures on the uniformization $\X_{b}$ of an expanding cone $\X$ constructed in the previous section. These results will be used specifically in the proof of Theorem \ref{thm:Anosov Poincare}. The reader only interested in the proof of Theorem \ref{thm:Patterson-Sullivan} may skip this section after  taking note of Definition \ref{defn:local doubling} and the definition \eqref{sigma def} of the rescaled measure. 

We begin with a general definition. Let $(X,d,\mu)$ be a \emph{metric measure space}, meaning that $(X,d)$ is a metric space equipped with a Borel measure $\mu$ such that all balls $B$ in $X$ have finite nonzero measure $0 < \mu(B) < \infty$. We let $B(x,r)$ denote the open ball of radius $r > 0$ centered at $x \in X$. 

\begin{defn}\label{defn:local doubling}
The measure $\mu$ is \emph{doubling} if there is a constant $C_{\mu} \geq 1$ such that for every $x \in X$ and $r > 0$ we have
\begin{equation}\label{define doubling measure}
\mu(B(x,2r)) \leq C_{\mu}\mu(B(x,r)).
\end{equation}
If the inequality \eqref{define doubling measure} only holds for balls of radius at most $R_{0}$ then we will say that $\mu$ is \emph{doubling on balls of radius at most $R_{0}$}.  We will alternatively say that $\mu$ is \emph{uniformly locally doubling} if there is an $R_{0} > 0$ such that $\mu$ is doubling on balls of radius at most $R_{0}$.

It's also useful to have a variant of the doubling condition in the absence of a measure $\mu$. A metric space $(X,d)$ is \emph{doubling} with constant $N \in \N$ if any ball $B(x,2r)$ of radius $2r$ can be covered by at most $N$ balls of radius $r$. The notions of doubling on balls of radius at most $R_{0}$ and the uniformly locally doubling property then extend in the obvious ways.
\end{defn}

It's straightforward to see (using $\supp \mu = X$) that the doubling properties of Definition \ref{defn:local doubling} above for $\mu$ imply the corresponding doubling properties for $X$.

We will frequently make use of the following consequence of the doubling estimate \eqref{define doubling measure}: if $\mu$ is doubling on balls of radius at most $R_{0}$ with constant $C_{\mu}$ and $0 < r \leq R \leq R_{0}$ then 
\begin{equation}\label{define doubling consequence}
\mu(B(x,R)) \asymp_{C} \mu(B(x,r)),
\end{equation}
with constant $C$ depending only on $C_{\mu}$ and the ratio $R/r$. This estimate follows by iterating the estimate \eqref{define doubling measure} and noting that $\mu(B(x,R)) \geq \mu(B(x,r))$ since $B(x,r) \subset B(x,R)$. 

We will require the following proposition from \cite{BBS20}, which is stated there in a more general form. 

\begin{prop}\label{enlarge doubling}\cite[Proposition 3.2]{BBS20} Let $(X,d)$ be a geodesic metric  space and let $\mu$ be a Borel measure on $X$ that is doubling on balls of radius at most $R_{0}$ with doubling constant $C_{\mu}$. Then for any $R_{1} > 0$ the measure $\mu$ is doubling on balls of radius at most $R_{1}$, with doubling constant depending only on $R_{1}/R_{0}$ and $C_{\mu}$. 
\end{prop}

Thus as long as $X$ is geodesic any uniformly locally doubling measure $\mu$ can be taken to be doubling on as large a scale as we desire as long as that scale is ultimately fixed. There is thus no harm in assuming $R_{0} \geq 1$, which we will do going forward. 

Now let $(\X,b)$ be an expanding cone with flow $f^{t}: \X \rightarrow \X$ as in the previous section. We let $\mu$ be a Borel measure on $\X$ that is locally finite, fully supported and uniformly locally doubling so that $(\X,d,\mu)$ is a metric measure space. However $(\bar{\X}_{b},d_{b},\mu)$ (with $\mu$ extended as $\mu(\p \X_{b}) = 0$) need not be a metric measure space unless $\mu$ decays sufficiently quickly with height in $\X$, as balls in $\bar{\X}_{b}$ that intersect $\p \X_{b}$ could potentially have infinite measure.

We can attempt to rectify this issue by defining for $\sigma > 0$,
\begin{equation}\label{sigma def}
d\mu_{\sigma}(x) = e^{-\sigma b(x)}\,d\mu(x),
\end{equation}
i.e., $\mu_{\sigma}$ is the measure equivalent to $\mu$ with density $x \rightarrow e^{-\sigma b(x)}$. This defines a Borel measure on $\X_{b}$ which we extend to a Borel measure on $\bar{\X}_{b}$ by setting $\mu_{\sigma}(\p \X_{b}) = 0$. We will show under some additional hypotheses on $\mu_{\sigma}$ that the triple $(\bar{\X}_{b},d_{b},\mu_{\sigma})$ defines a doubling metric measure space. We remark that no additional hypotheses are needed to show that this is the case if $\sigma$ is sufficiently large (see \cite{Bu22}) however this is unhelpful when one is concerned with specific values of $\sigma$.

Our convention regarding $\asymp$, etc. will be extended to allow implied constants to depend on $\sigma$, the local doubling constant $C_{\mu}$ and radius $R_{0}$ for $\mu$ as well. When we say a quantity depends on ``the data" we mean that it depends only on $a$, $A$, $C_{\mu}$, $R_{0}$, and $\sigma$. Constants that depend on $\sigma$ should be understood to be uniform when the values of $\sigma$ are restricted to a compact subset of $(0,\infty)$. Despite our broadening of the convention, many implied constants will still depend only on $a$ and $A$; for the most part we leave the determination of when this does and does not apply to the interested reader.

We first establish some estimates for $\mu_{\sigma}$ using Lemmas \ref{sub inclusion} and \ref{choosing center} that do not require additional assumptions. We start with the following consequence of Lemma \ref{sub inclusion}.

\begin{lem}\label{sub measure}
Let $x \in \X$ and $0 < r \leq \frac{1}{2}d_{b}(x)$. Then 
\[
\mu_{\sigma}(B_{b}(x,r)) \asymp e^{-\sigma b(x)}\mu(B(x, re^{ab(x)})).
\]
\end{lem}

\begin{proof}
By \eqref{Harnack} and \eqref{distance to boundary} we have for all $y \in B_{b}(x,r)$, 
\begin{equation}\label{comparison chain}
e^{-\sigma b(y)} \asymp d_{b}(y)^{\sigma/a} \asymp d_{b}(x)^{\sigma/a} \asymp e^{-\sigma b(x)} ,
\end{equation}
with the comparison $d_{b}(y) \asymp d_{b}(x)$ following from the condition on $r$. Applying Lemma \ref{sub inclusion} and the chain of comparisons \eqref{comparison chain}, we conclude that
\[
\mu_{\sigma}(B_{b}(x,r)) \asymp e^{-\sigma b(x)}\mu(B_{b}(x,r)) \leq e^{-\sigma b(x)} \mu(B(x,C_{*}r e^{ab(x)})),
\]
with $C_{*}$ being the constant from Lemma \ref{sub inclusion}. A similar argument using the other inclusion from Lemma \ref{sub inclusion} shows that 
\[
\mu_{\sigma}(B_{b}(x,r)) \gtrsim e^{-\sigma b(x)}\mu(B(x,C_{*}^{-1}re^{ab(x)})).
\]
We thus conclude that
\begin{equation}\label{proto sub measure}
e^{-\sigma b(x)}\mu(B(x,C_{*}^{-1}re^{ab(x)}))\lesssim \mu_{\sigma}(B_{b}(x,r)) \lesssim e^{-\sigma b(x)} \mu(B(x,C_{*}re^{ab(x)})).
\end{equation}
The condition on $r$ implies that 
\begin{equation}\label{ratio inequality}
re^{ab(x)} \leq \frac{1}{2}d_{b}(x)e^{ab(x)}\leq C,
\end{equation}
with $C = C(a,A)$ a uniform constant by \eqref{distance to boundary}. By Proposition \ref{enlarge doubling} we can, at the cost of increasing the local doubling constant $C_{\mu}$ of $\mu$ by an amount depending only on the data, assume that $R_{0} > CC_{*}$ for the constant $C$ in inequality \eqref{ratio inequality} and the constant $C_{*}$ in Lemma \ref{sub inclusion}. Then the comparison \eqref{define doubling consequence} allows us to conclude that
\[ 
 \mu(B(x,C_{*}^{-1}re^{ab(x)})) \asymp \mu(B(x,re^{ab(x)} )) \asymp \mu(B(x,C_{*}re^{ab(x)})).
\]
Combining this comparison with inequality \eqref{proto sub measure} proves the lemma. 
\end{proof}  

We can obtain a more precise estimate for $\mu_{\sigma}$ on subWhitney balls. 

\begin{lem}\label{estimate upper}
Let $L \geq 1$ be given and let $z \in Y$ and $r > 0$  be such that $2c_{0} r \leq d_{b}(z) \leq Lr$. Then
\begin{equation}\label{comparison upper}
\mu_{\sigma}(B_{b}(z,c_{0} r)) \asymp r^{\sigma/a}\mu(B(z,1)),
\end{equation}
with implied constant depending only on the data and $L$.
\end{lem}

\begin{proof}
Throughout this proof ``the data" will additionally include the constant $L$, and similarly implied constants from $\asymp$ will additionally depend on $L$.  The assumptions imply that $d_{b}(z) \asymp r$, hence $e^{-\sigma b(z)} \asymp r^{\sigma/a}$ by \eqref{distance to boundary}, with comparison constants depending only on the data since $c_{0}$ depends only on the data. Thus by Lemma \ref{sub measure} we have
\[
\mu_{\sigma}(B_{b}(z,c_{0} r)) \asymp r^{\sigma/a}\mu(B(z,c_{0}re^{-a b(z)})),
\]
with comparison constant depending only on the data. Since $e^{-ab(z)} \asymp d_{b}(z) \asymp r$, we have $c_{0} r e^{-ab(z)} \asymp_{C} 1$ for a constant $C$ depending only on the data. Using Proposition \ref{enlarge doubling} we can assume that $\mu$ is doubling on balls of radius at most $C$, at the cost of increasing the doubling constant by an amount depending only on the data. From this we conclude that the comparison \eqref{comparison upper} holds. 
\end{proof}

We can now show that if one has certain estimates on the $\mu_{\sigma}$-measures of balls $B_{b}(\xi,r)$ centered at points $\xi \in \p \X_{b}$ then $\mu_{\sigma}$ is doubling on $\bar{\X}_{b}$. More precisely, we will assume that the measure inequality for $\mu_{\sigma}$ applied to balls centered on $\p \X_{b}$ by combining Lemmas \ref{choosing center} and \ref{estimate upper} also holds in reverse. We recall the notation from Proposition \ref{boundary identify} for $z \in \X$ that $\bar{z} = \Psi_{z}(z)$.

\begin{prop}\label{crit above}
There is a constant $L = L(a,A)$ such that the following holds: suppose that there is a constant $K$ such that for any $z \in \X$ and $r > 0$ with $c_{0}r \leq d_{b}(z) \leq Lr$ we have
\begin{equation}\label{controlled upper}
\mu_{\sigma}(B_{b}(\bar{z},Lr)) \leq Kr^{\sigma/a}\mu(B(z,1)),
\end{equation} 
Then $\mu_{\sigma}$ is doubling on $(\bar{\X}_{b},d_{b})$ with doubling constant depending only on the data and $K$.
\end{prop} 

\begin{proof}
For the doubling property on $\bar{\X}_{b}$ we split into two cases depending on the distance of the center $x \in \bar{\X}_{b}$ of the ball from the boundary. The first case is that in which $0 < r \leq \frac{1}{4}d_{b}(x)$, which implies in particular that $x \in \X_{b}$. Then we can apply Lemma \ref{sub measure} to both $B_{b}(x,r)$ and $B_{b}(x,2r)$. We conclude that
\begin{equation}\label{case one}
\mu_{\sigma}(B_{b}(x,r)) \asymp e^{-\sigma b(x)}\mu(B(x, re^{ab(x)})) \asymp e^{-\sigma b(x)}\mu(B(x, 2re^{ab(x)}))  \asymp \mu_{\sigma}(B_{b}(x,2r)),
\end{equation}
with comparison constants depending only on the data. To justify the middle comparison in \eqref{case one}, we observe that since $2r \leq \frac{1}{2}d_{b}(x)$ we have by \eqref{distance to boundary} that each of the middle two balls in $\X$ in \eqref{case one} on the right side of this inequality have radius at most $C$ for some constant $C=C(a,A)$. By Proposition \ref{enlarge doubling} we can assume that $\mu$ is doubling on balls of radius at most $C$, at the cost of increasing the doubling constant of $\mu$ by an amount depending only $a$ and $A$. This gives the desired doubling estimate in this case. 

The second case is that in which $d_{b}(x) < 4r$. We then let $\bar{x} \in \p \X_{b}$ be the endpoint of the ascending geodesic ray starting from $x$. Then by \eqref{distance to boundary} we have $B_{b}(x,r) \subset B_{b}(\bar{x},Lr)$ with $L = L(a,A)$. We then use Lemma \ref{choosing center} to choose a point $z \in \X_{b}$ on the same vertical geodesic as $x$ such that $B_{b}(z,c_{0}r) \subset B_{b}(x,r)$ and $d_{b}(z) \geq 2c_{0}r$. Then $\bar{z} = \bar{x}$ and we must have $d_{b}(z) < Lr$ since $z \in B_{b}(\bar{x},Lr)$. Since $B_{b}(z,c_{0}r) \subset B_{b}(\bar{x},Lr)$ and $2c_{0}r \geq c_{0}r$, we conclude from Lemma \ref{estimate upper} and the assumed inequality \eqref{controlled upper} that 
\begin{equation}\label{case two a}
\mu_{\sigma}(B_{b}(x,r)) \asymp \mu_{\sigma}(B_{b}(z,c_{0} r)),
\end{equation}
with comparison constant depending only on the data and $K$. Since we also have $B_{b}(x,2r) \subset B_{b}(\bar{x},2Lr)$ and $2c_{0}r = c_{0} \cdot 2r$, the same combination of Lemma \ref{estimate upper} and \eqref{controlled upper} also shows that
\begin{equation}\label{case two b}
\mu_{\sigma}(B_{b}(x,2r)) \asymp \mu_{\sigma}(B_{b}(z,c_{0} r)),
\end{equation}
with comparison constant depending only on the data and $K$. Combining \eqref{case two a} and \eqref{case two b} gives the desired doubling estimate in this second case.
\end{proof}

By Proposition \ref{ball comparison} it suffices to establish an estimate analogous to \eqref{controlled upper} on cones instead. This is the form in which Proposition \ref{crit above} will be used in the next section.

\begin{prop}\label{cone comparison}
There is a constant $L' = L'(a,A) \geq 1$ such that if there is a constant $K$ for which any $x \in \X$ satisfies
\begin{equation}\label{cone upper}
\mu_{\sigma}(\CC(x,L')) \leq Ke^{-\sigma b(x)}\mu(B(x,1)),
\end{equation}
then there is a constant $K'$ depending only on the data and $K$ such that for any $z \in \X$ and $r > 0$ with $Lr \geq d_{b}(z) \geq c_{0}r$ we have
\begin{equation}\label{cone controlled}
\mu_{\sigma}(B_{b}(\bar{z},Lr)) \leq K'r^{\sigma/a}\mu(B(z,1)),
\end{equation}
where $L = L(a,A)$ is the constant in \eqref{controlled upper}.
\end{prop}

\begin{proof}
This follows immediately by combining Proposition \ref{ball comparison} with the estimate \eqref{distance to boundary} and tracking the dependencies of the parameters, noting that $r^{\sigma/a} \asymp e^{-\sigma b(x)}$ when $r \asymp d_{b}(x)$.
\end{proof}

While the phrasing of Propositions \ref{crit above} and \ref{cone comparison} is a bit awkward, the resulting criteria for doubling are easy to verify. 

The remainder of this section will be devoted to \emph{Poincar\'e inequalities} on metric measure spaces with applications to the uniformized measures $\mu_{\sigma}$ on $\X_{b}$ discussed above.  We first give the general definition. Let $(X,d,\mu)$ be a metric measure space. For a measurable subset $E \subset X$ satisfying $0 < \mu(E) < \infty$ and a function $u$ that is $\mu$-integrable over $E$ we write 
\begin{equation}\label{average notation}
u_{E} = \dashint_{E} u \, d\mu = \frac{1}{\mu(E)}\int_{E} u \, d\mu
\end{equation}
for the mean value of $u$ over $E$. Let $u: X \rightarrow \R$ be given. A Borel function $g: X \rightarrow [0,\infty]$ is an \emph{upper gradient} for $u$ if for each rectifiable curve $\gamma$ joining two points $x,y \in X$ we have
\[
|u(x)-u(y)| \leq \int_{\gamma} g \, ds.
\]
A measurable function $u: X \rightarrow \R$ is \emph{integrable on balls} if for each ball $B \subset X$ we have that $u$ is integrable over $B$. We say that $X$ \emph{supports a Poincar\'e inequality} if there are constants $\la \geq 1$ and $C_{\mathrm{PI}} > 0$ such that for each measurable function $u: X \rightarrow \R$ that is integrable on balls, for each ball $B \subset X$, and each upper gradient $g$ of $u$ we have
\begin{equation}\label{uniformly local Poincare}
\dashint_{B} |u-u_{B}|\,d\mu \leq C_{\mathrm{PI}}\diam(B)\dashint_{\la B} g \, d\mu.
\end{equation}
The constant $\la$ is called the \emph{dilation constant}. Here we are using the notation $\la B = B(x,\la r)$ for the scaling of a ball's radius by $\la$. The dilation constant $\la$ can always be improved to $\la = 1$ if $(X,d)$ is a geodesic metric space \cite[Theorem 4.18]{Hein01} as is the case with the uniformized expanding cone $\bar{\X}_{b}$, however we will need to allow the presence of this dilation when working with local Poincar\'e inequalities below.

\begin{rem}\label{Poincare strength}
The inequality \eqref{uniformly local Poincare} with $\la = 1$ is the strongest form of a Poincar\'e inequality that a metric measure space can satisfy. More generally one can use the $L^p$-norm of the upper gradient on the right side of \eqref{uniformly local Poincare} instead for a fixed $p \geq 1$, giving rise to a \emph{$p$-Poincar\'e inequality}. The inequality \eqref{uniformly local Poincare} with $\la > 1$ is also sometimes referred to as a \emph{weak Poincar\'e inequality}. The reference \cite{HKST} contains a great deal more detail on this particular topic. 
\end{rem}

If there is a constant $R_{\mathrm{PI}} > 0$ such that \eqref{uniformly local Poincare} only holds on balls of radius at most $R_{\mathrm{PI}}$ then we will say that $X$ \emph{supports a Poincar\'e inequality on balls of radius at most $R_{\mathrm{PI}}$}. We will also say that $X$ \emph{supports a uniformly local Poincar\'e inequality}.

We now specialize to the case of an expanding cone $\X$ equipped with a uniformly locally doubling measure $\mu$ as above. We will \emph{assume} for a fixed $\sigma$ that $\mu_{\sigma}$ is doubling on $\X_{b}$ with some doubling constant $C_{\mu_{\sigma}}$, and we will then establish that both of the metric measure spaces $(\X_{b},d_{b},\mu_{\sigma})$ and $(\bar{\X}_{b},d_{b},\mu_{\sigma})$ support a Poincar\'e inequality. This doubling property will be deduced from Proposition \ref{cone comparison} above in the case of Theorem \ref{thm:Anosov Poincare}. We will also assume that $(\X,d,\mu)$ supports a uniformly local Poincar\'e inequality. This condition is not particularly difficult to check when $\mu$ is the Riemannian volume on $\X$. In this case the upper gradient $g$ of a $C^1$ function $u:\X \rightarrow \R$ can be taken to be the norm of the gradient of $u$, $g = \|\nabla u\|$, and the uniformly local Poincar\'e inequality will follow from the ordinary Poincar\'e inequality in Euclidean space (see for instance \cite[Chapter 4]{Hein01}) provided that the unit ball around any given $x \in \X$ can be mapped to the Euclidean unit ball with biLipschitz constant uniformly controlled in $X$. A simple compactness argument shows this is always the case for expanding cones $\X_{x} = (\wt{\W}^{cu}(x),b_{x})$ arising from the center-unstable manifolds of an Anosov flow as in the setting of Theorem \ref{thm:Anosov Poincare}.

We first show that the Poincar\'e inequality \eqref{uniformly local Poincare} holds on sufficiently small subWhitney balls in the metric measure space $(\X_{b},d_{b},\mu_{\sigma})$. The proof is essentially identical to \cite[Lemma 6.1]{BBS20}. In the statement and proof of Lemma \ref{Whitney Poincare} ``the data" refers to the growth rates $a$ and $A$, and the constants $R_{0}$, $C_{\mu}$, $\la$, $R_{\mathrm{PI}}$ and $C_{\mathrm{PI}}$. For this particular lemma we do not actually need to assume that $\mu_{\sigma}$ is doubling.

\begin{lem}\label{Whitney Poincare}
Assume that $(\X,d,\mu)$ supports a uniformly local Poincar\'e inequality. Then there exists $c_{1} > 0$ and $R_{1} > 0$ depending only on the data such that for all $x \in \X_{b}$ and all $0 < r\leq c_{1}d_{b}(x)$ the Poincar\'e inequality \eqref{uniformly local Poincare} for $\mu_{\sigma}$ holds on the ball $B_{b}(x,r)$ with dilation constant $\hat{\la}$ and constant $\hat{C}_{\mathrm{PI}}$ depending only on the data.
\end{lem}

\begin{proof}
Put $B_{b} = B_{b}(x,r)$ with $0 < r \leq c_{1}d_{b}(x)$, where $0 < c_{1} \leq \frac{1}{2}$ is a constant to be determined. Let $C_{*}$ be the constant of Lemma \ref{sub inclusion}. We choose $c_{1} > 0$ small enough that $c_{1}C_{*}^{2} \leq \frac{1}{2}$. We conclude by applying Lemma \ref{sub inclusion} twice that 
\begin{equation}\label{ball chain}
B_{b} \subset B:=B(x,C_{*}r e^{ab(x)})  \subset B_{b}(x,C_{*}^{2}r) = \hat{\la} B_{b},
\end{equation}
with $\hat{\la} = C_{*}^{2}$, since 
\[
C_{*}^{2}r \leq c_{1}C_{*}^{2}d_{b}(x) \leq \frac{1}{2}d_{b}(x).
\]
Moreover by \eqref{comparison chain} we see that for all $y \in \hat{\la} B_{b}$ we have $e^{-\sigma b(y)}\asymp e^{-\sigma b(x)}$ with comparison constant depending only on the data.

Now let $u$ be a function on $\X_{b}$ that is integrable on balls and let $g_{b}$ be an upper gradient of $u$ on $\X_{b}$. By the same basic calculation as in \cite[(6.3)]{BBS20} we have that $g:=g_{b}\kappa$ is an upper gradient of $u$ on $\X$, where we recall that $\kappa(x) = e^{-ab(x)}$ is the density used in the uniformization. For $c_{1}$ sufficiently small (depending only on $a$, $A$, and $R_{\mathrm{PI}}$) we will have by \eqref{distance to boundary} that
\[
C_{*}re^{ab(x)}\leq C_{*}c_{1}d_{b}(x)e^{ab(x)}\leq R_{\mathrm{PI}}.
\]
Thus the Poincar\'e inequality \eqref{uniformly local Poincare} (for $\mu$) holds on $B$. Since $e^{-\sigma b(y)}\asymp e^{-\sigma b(x)}$ on $\hat{\la} B_{b}$ with comparison constant depending only on the uniformization data (by \eqref{comparison chain}) we have that
\begin{equation}\label{measure comparison}
\mu_{\sigma}(B) \asymp e^{-\sigma b(x)}\mu(B),
\end{equation}
with comparison constant depending only on the data, and the same comparison holds with either $B_{b}$ or $\hat{\la}B_{b}$ replacing $B$. Writing $u_{B,\mu} = \dashint_{B} u \, d\mu$, we conclude by using the inclusions of \eqref{ball chain}, the measure comparison \eqref{measure comparison}, and the $p$-Poincar\'e inequality for $\mu$ on $B$,
\begin{align*}
\dashint_{B_{b}}|u-u_{B,\mu}| \, d\mu_{\sigma} &\lesssim \dashint_{B} |u-u_{B,\mu}| d\mu \\
&\leq 2C_{\mathrm{PI}}C_{*}re^{ab(x)}\dashint_{B} g\, d\mu \\
&\asymp re^{ab(x)}\dashint_{B} g_{b}\kappa\, d\mu_{\sigma} \\
&\lesssim r \dashint_{\hat{\la} B_{b}} g_{b}^{p} \, d\mu_{\sigma},
\end{align*}
where all implied constants depend only on the data. Since by the triangle inequality
\[
\dashint_{B_{b}} |u-u_{B_{b},\mu_{\sigma}}|\,d\mu \leq 2\dashint_{B_{b}} |u-u_{B,\mu}|\,d\mu,
\]
we can replace $u_{B,\mu}$ with $u_{B_{b},\mu_{\sigma}} = \dashint_{B_{b}} u \, d\mu_{\sigma}$  to conclude the proof of the lemma. 
\end{proof}

We will also need the following key proposition. We will be using the case $p = 1$.

\begin{prop}\label{uniform upgrade}\cite[Proposition 6.3]{BBS20} 
Let $\Omega$ be an $L$-uniform metric space equipped with a doubling measure $\nu$ such that there is a constant $0 < c_{1} < 1$ for which the Poincar\'e inequality \eqref{uniformly local Poincare} holds for fixed constants $C_{\mathrm{PI}}$ and $\la$ on all subWhitney balls $B$ of the form $B = B_{\Omega}(x,r)$ with $x \in \Omega$ and $0 < r \leq c_{1}d_{\Omega}(x)$. Then the metric measure space $(\Omega,d,\nu)$ supports a $p$-Poincar\'e inequality with dilation constant $L$ and constant $C'_{\mathrm{PI}}$ depending only on $L$, $c_{1}$, $p$, $C_{\mathrm{PI}}$, $\la$, and the doubling constant $C_{\nu}$ for $\nu$. 
\end{prop}

This proposition is stated for bounded $L$-uniform metric spaces in \cite{BBS20} but the proof works without modification for unbounded $L$-uniform metric spaces provided that the doubling property of $\nu$ holds at all scales and the Poincar\'e inequality on subWhitney balls holds at all appropriate scales. 

We can now verify the global Poincar\'e inequality on $\X_{b}$ and $\bar{\X}_{b}$, which proves Theorem \ref{global Poincare}. Below ``the data" includes all the constants from Lemma \ref{Whitney Poincare} as well as the doubling constant $C_{\mu_{\sigma}}$ for $\mu_{\sigma}$. 

\begin{prop}\label{global Poincare}
Suppose that $\mu_{\sigma}$ is doubling with doubling constant $C_{\mu_{\sigma}}$ and suppose that $(\X,d,\mu)$ supports a uniformly local Poincar\'e inequality. Then both $(\X_{b},d_{b},\mu_{\sigma})$ and $(\bar{\X}_{b},d_{b},\mu_{\sigma})$ support a Poincar\'e inequality with dilation constant $\la = 1$ and multiplicative constant $C_{\mathrm{PI}}$ depending on the data and $C_{\mu_{\sigma}}$.
\end{prop}

\begin{proof}
By Lemma \ref{Whitney Poincare} there is a $c_{1} > 0$ determined only by the data such that the Poincar\'e inequality holds on subWhitney balls of the form $B_{b}(x,r)$ with $0 < r \leq c_{1}d_{b}(x)$ for $x \in \X$, with uniform constants $\hat{C}_{\mathrm{PI}}$ and $\hat{\la}$. Since $(\X_{b},d_{b})$ is an $L$-uniform metric space with $L=L(a,A)$  and we assumed $\mu_{\sigma}$ is globally doubling on $\X_{b}$ with constant $\mu_{\sigma}$, it follows from Proposition \ref{uniform upgrade} that the metric measure space $(\X_{b},d_{b},\mu_{\sigma})$ supports a Poincar\'e inequality with constant $C_{\mathrm{PI}}'$ depending only on the data and dilation constant $L$. Since $\X_{b}$ is geodesic it follows that the Poincar\'e inequality \eqref{uniformly local Poincare} in fact holds with dilation constant $1$, with constant  $C_{\mathrm{PI}}^{*}$ depending only on the data \cite[Theorem 4.18]{Hein01}.

By \cite[Lemma 8.2.3]{HKST} we conclude that the completion $(\bar{\X}_{b},d_{b},\mu_{\sigma})$ (with $\mu_{\sigma}(\p \X_{b}) = 0$) also supports a Poincar\'e inequality with constants depending only on the constants for the Poincar\'e inequality on $\X_{b}$ and the doubling constant of $\mu_{\sigma}$. Since $\bar{\X}_{b}$ is also geodesic it follows by the same reasoning as before that we can take the dilation constant to be $1$ in this case as well. 
\end{proof}

\section{Construction of Measure of Maximal Entropy}\label{sec:max entropy}

%For this section we will need the precise notion of uniform smoothness defined in my thesis [cite]. The definition is recalled below, for more details refer to my thesis. We write $B_{k}$ for the unit ball in $\R^{k}$.
%
%\begin{defn}\label{uniform foliation}
%Let $\W$ be a foliation of a $C^s$ manifold $M$ and let $s \geq r \geq 1$. We define $\W$ to have \emph{uniformly $C^r$ leaves}  if there is an atlas of foliation charts $(\psi_{j},U_{j})_{j \in J}$ for $\W$ with $\psi_{j}: U_{j} \rightarrow B_{k} \times B_{m-k}$ such that the compositions for $p \in B_{m-k}$, 
%\[
%\zeta_{j,p} = \psi^{-1}_{j} \circ i_{p} : B_{k} \rightarrow M,
%\] 
%are $C^r$ and depend continuously on $p$ inside of $C^{r}(B_{k},M)$. 
%
%For another $C^r$ manifold $N$, we define $f: M \rightarrow N$ to be \emph{uniformly $C^r$ along $\W$} if for each chart $(\psi_{j},U_{j})$ the compositions $f \circ \zeta_{j,p}$ depend continuously on $p \in B_{m-k}$ inside of $C^{r}(B_{k},N)$. 
%\end{defn} 

We return now to the setting of Anosov flows. We suppose that $M$ is a $C^{r+1}$ closed Riemannian manifold and $f^{t}:M \rightarrow M$ is a $C^r$ transitive Anosov flow on $M$, $r \geq 2$. We use the same notation for subbundles and foliations as was described prior to the statement of Theorem \ref{thm:Patterson-Sullivan}. We write $d^{M}$ for the Riemannian metric on $M$. For each leaf $\W^{*}(x)$ of the foliation $\W^{*}$ we write $d^{*}_{x}$ for the induced Riemannian metric on $\W^{*}(x)$ from $M$. We write $B^{*}(x,r)$ for the ball of radius $r$ in $\W^{*}(x)$ centered at $x$ in the metric $d^{*}_{x}$ (and naturally $B^{M}(x,r)$ for the corresponding ball for $d^{M}$). We assume that the inequalities \eqref{unstable expansion} and \eqref{stable expansion} hold with constants $C_{u} = C_{s} = 1$. If these constants are present in those inequalities we assume $r \geq 3$ instead and remove them by introducing a new Riemannian metric using Proposition \ref{Lyapunov metric}. The restriction $C_{u} = 1$ is required to apply the results of the previous sections, while $C_{s} = 1$ is required to assert that $f^{t}$ does not expand distances along $\W^{cs}$: for any $x \in M$, $y \in \W^{cs}(x)$ we have for all $t \geq 0$,
\begin{equation}\label{no expand}
d^{cs}_{f^{t}x}(f^{t}x,f^{t}y) \leq d^{cs}_{x}(x,y).
\end{equation} 
To simplify notation we write $a = a_{u}$ and $A = A_{u}$ for the lower and upper bounds on growth rates on $E^{u}$. 

As we did in the proof of Theorem \ref{thm:expanding hyperbolic} toward the end of Section \ref{sec:expanding cone}, for $x \in M$ we write $\X_{x} = (\wt{\W}^{cu}(x),b_{x})$ for the expanding cone associated to the point $x$ consisting of the universal cover of the leaf $\W^{cu}(x)$ together with the height function $b_{x}$ satisfying $b_{x}(x) = 0$ whose level sets are the unstable manifolds $\W^{u}(y)$, $y \in \wt{\W}^{cu}(x)$. We recall that $\wt{\W}^{cu}(x) = \W^{cu}(x)$ unless $\W^{cu}(x)$ contains a periodic orbit. 

For $x \in M$ we write $\W^{*}_{\loc}(x)$ for the \emph{local leaf} of $\W^{*}$ through $x$, defined to be any plaque of uniform size of the foliation $\W^{*}$ through $x$ in a foliation chart for $\W^{*}$. We will shortly clarify below what we mean by ``uniform size". We write $h^{*}_{xy}$ for the local \emph{holonomy maps} of the foliation $\W^{*}$ between local transversals (with subscripts indicating two points $x$ and $y$ in each transversal with $y \in \W^{*}_{\loc}(x)$), so that for instance if $x \in M$, $y \in \W^{s}_{\loc}(x)$ then $h^{s}_{xy}: \W^{cu}_{\loc}(x) \rightarrow \W^{cu}_{\loc}(y)$ is the $s$-holonomy map assigning to each $z \in \W^{cu}_{\loc}(x)$ the unique intersection point $h^{s}_{xy}(z)$ of $\W^{s}_{\loc}(z)$ with $\W^{cu}_{\loc}(y)$. The $f^{t}$-invariance of these foliations leads to the \emph{equivariance} property, stated below for $s$-holonomy in particular: for any $y \in \W^{s}_{\loc}(x)$, and any $t \geq 0$, 
\begin{equation}\label{equivariant holonomy}
h^{s}_{f^{t}x f^{t}y} = f^{t} \circ h^{s}_{xy} \circ f^{-t}: f^{t}(\W^{cu}_{\loc}(x)) \rightarrow f^{t}(\W^{cu}_{\loc}(y)).
\end{equation}
An analogous statement holds for $cs$-holonomy replacing $s$ with $cs$ and $cu$ with $u$ above. The corresponding statements for $u$-holonomy and $cu$-holonomy are obtained by replacing $f^{t}$ with $f^{-t}$ and swapping $s$ for $u$. Note also that $c$-holonomy is simply given by flowing by $f^{t}$. The expression in \eqref{equivariant holonomy} should be understood as meaning that $f^{t} \circ h^{s}_{xy} \circ f^{-t}$ corresponds to local stable holonomy in any local leaf of $\W^{cu}$ contained within $f^{t}(\W^{cu}_{\loc}(x))$. 

For $t < 0$ in \eqref{equivariant holonomy} we treat this expression as an extension of the definition of $s$-holonomy to points that may not lie on the same local stable leaf but still lie on the same stable leaf at a global scale (and the same for the other holonomies). Note that at this global scale the basepoint marking in $h^{*}_{xy}$ is crucial to specify which holonomy we are referring to, as the individual leaves of these foliations are usually dense in $M$ (this is always the case for $\W^{cs}$ and $\W^{cu}$, and is also the case for $\W^{s}$ and $\W^{u}$ if $f^{t}$ is topologically mixing).

By the compactness of $M$ and the continuity of the foliations there is an $R > 0$ such that for any $x \in M$ and any $* \in \{u,c,s,cu,cs,M\}$ the ball $B^{*}(x,R)$ fits inside a foliation chart for each of the foliations $\W^{*}$ (and, if $* \neq M$, is contained within a plaque for a foliation chart for its corresponding foliation). By rescaling the Riemannian metric on $M$ by a constant factor (which does not affect \eqref{unstable expansion} or \eqref{stable expansion}) we can assume that $R = 1$, which we will do in the rest of this section. 

For each $x \in M$ we write $\rho_{x}$ for the Hamenst\"adt metric on $\W^{u}(x)$, defined by the formula \eqref{metric definition}. It's straightforward to check that the lift of this metric to $\wt{\W}^{u}(x)$ inside of $\X_{x}$ is a Hamenst\"adt metric in the expanding cone sense defined by the same formula. As in Section \ref{sec:uniform cone} we write $B_{\rho}(x,r)$ for the ball of radius $r$ in the metric $\rho_{x}$ centered at $x$. 

The basepoints in objects like $d^{u}_{x}$, $\rho_{x}$, $h^{cs}_{xy}$, etc. will be omitted when they can be understood from context to reduce clutter in the notation. Typically one can directly infer which basepoint is intended from the arguments being passed to the object, e.g. $d^{u}(x,y) = d^{u}_{x}(x,y)$ for $y \in \W^{u}(x)$.

We start by considering the local center-stable holonomy maps between unstable leaves. The additional conclusion that $K_{R} \rightarrow 1$ as $R \rightarrow 0$ is only needed for Proposition \ref{holonomy invariance}.

\begin{lem}\label{Hamenstadt holonomy}
For each $R > 0$ there is a constant $K = K_{R} \geq 1$ such that if $x \in M$ and $y \in \W^{cs}(x)$ with $d^{cs}(x,y) \leq R$ then the holonomy $h^{cs}:B_{\rho}(x,R) \rightarrow \W^{u}(y)$ is $K$-biLipschitz in the Hamenst\"adt metrics. The constants $K_{R}$ can be chosen such that $K_{R} \rightarrow 1$ as $R \rightarrow 0$.
\end{lem}

\begin{proof}
We first observe that for $x \in M$, $z \in \W^{u}(x)$ with $\rho(x,z) = 1$, and $y \in \W^{cs}(x)$ with $d^{cs}(x,y) \leq R$, by the compactness of $M$, the continuous dependence of the Hamenst\"adt metrics $\rho_{x}$ on the basepoint $x$, and the continuity of the holonomy maps of the foliations there is a uniform constant $K \geq 1$ depending on $R$ such that $\rho(y,h^{cs}(z)) \asymp_{K} 1$. Next observe in this configuration that if instead we have $\rho(x,z) < R$ then we can find $t > 0$ such that 
\[
\rho(f^{t}x,f^{t}z) = e^{at}\rho(x,z) = R. 
\]
Since $f^{t}$ does not expand distances on $\W^{cs}(x)$ by \eqref{no expand} it follows that $d^{cs}(f^{t}x,f^{t}y) \leq R$ as well. Thus $\rho(f^{t}y,h^{cs}(f^{t}z)) \asymp_{K} R$. Applying the equivariance property \eqref{equivariant holonomy} we conclude that
\[
e^{at}\rho(y,h^{cs}(z))= \rho(f^{t}y,f^{t}(h^{cs}(z))) \asymp_{K} 1,
\]
which implies that
\[
\rho(y,h^{cs}(z)) \asymp_{K} \rho(x,z).
\]

Lastly we take a closer look at the dependence of $K$ on $R$ as $R \rightarrow 0$. The constant $K$ is determined by the maximal and minimal values over $x \in M$, $y \in \W^{cs}(x)$, $z \in \W^{u}(z)$ of $\rho(y,h^{cs}(z))$ when $\rho(x,z) = 1$ and $d^{cs}(x,y) \leq R$. The compactness of $M$ implies that the function $(x,y,z) \rightarrow \rho(y,h^{cs}(z))$ is uniformly continuous over all triples $(x,y,z) \in M^{3}$ with $y \in \W^{cs}(x)$ satisfying $d^{cs}(x,y) \leq 1$ and $z \in \W^{u}(x)$ satisfying $\rho(x,z) = 1$. Furthermore if $d^{cs}(x,y) = 0$ then $x = y$, $z = h^{cs}(z)$, and therefore $\rho(y,h^{cs}(z)) = 1$. Thus given $\e > 0$ we can find $\delta > 0$ such that if $d^{cs}(x,y) < \delta$ then 
\[
(1+\e)^{-1} < \rho(y,h^{cs}(z)) < 1+\e. 
\]
Therefore if $R < \delta$ then we may take $K_{R} = 1+\e$. Since $\e > 0$ was arbitrary this means we can choose $K_{R}$ such that $K_{R} \rightarrow 1$ as $R \rightarrow 0$.
\end{proof}

The next claim establishes that any two unit Hamenst\"adt balls have biLipschitz equivalent subballs of uniform size. We emphasize below that that the center $z$ of said subball cannot always be chosen to be the center $x$ of the larger ball, and it may not even be the case that the subball $B_{\rho}(z,R)$ contains $x$. Similarly $y$ may not be in the $cs$-holonomy image of $B_{\rho}(z,R)$.

\begin{lem}\label{subball equivalence}
There is a uniform constant $K \geq 1$ and radius $0 < R \leq 1$ such that for any $x,y \in M$ there is a point $z \in B_{\rho}(x,1)$ with $B_{\rho}(z,R) \subset B_{\rho}(x,1)$ such that there is a center-stable holonomy map $h^{cs}:B_{\rho}(z,R) \rightarrow B_{\rho}(y,1)$ that is a $K$-biLipschitz homeomorphism onto its image in the Hamenst\"adt metrics.
\end{lem}

\begin{proof}
If $f^{t}$ is topologically mixing then the stable leaves of $f^{t}$ are uniformly dense in $M$; this is simply a reformulation of the \emph{specification property} that topologically mixing flows satisfy \cite[Chapter 8.3]{FH19}. Consequently there is an $L > 0$ and $r > 0$ such that for any two points $x,y \in M$ there is $T_{xy} \in \R$ with $|T_{xy}| \leq 1$ and points $p \in f^{T_{xy}}(B_{\rho}(x,1))$ and $q \in B_{\rho}(y,1)$ such that $B_{\rho}(p,r) \subset f^{T_{xy}}(B_{\rho}(x,1))$, $\rho(y,q) <  \frac{1}{4}$, $q \in \W^{s}(p)$ and $d^{s}(p,q) \leq L$. When $f^{t}$ isn't topologically mixing a variant of this claim is also true: in this case $f^{t}$ is a suspension flow with constant roof function by a theorem of Plante \cite[Theorem 9.1.1]{FH19} (in particular the stable and unstable foliations jointly integrate to a foliation $\W^{su}$) and the stable leaves are instead uniformly dense within each leaf $\W^{su}(x)$ by the corresponding specification property for Anosov diffeomorphisms. This uniform density can be taken to be uniform over all leaves $\W^{su}(x)$ since $f^{t}$ is a suspension flow. Furthermore all of the leaves $\W^{su}(x)$ are $c$-holonomy related by the flow $f^{t}$, and the $d^{c}$-length of these $c$-holonomy segments is bounded by some uniform constant $T > 0$ determined by the constant value of the roof function for $f^{t}$. The claim that such an $L > 0$ and $R > 0$ exists with $T_{xy} \in \R$ and points $p,q$ as above follows by combining all of these assertions, now with $|T_{xy}| \leq T$.

Thus in either case we have an $L > 0$, $R > 0$, and $T_{xy} \in \R$ with $p \in f^{T}(B_{\rho}(x,1))$ and $q \in B_{\rho}(y,1)$ such that $B_{\rho}(p,r) \subset f^{T_{xy}}(B_{\rho}(x,1))$, $\rho(y,q) <  \frac{1}{4}$,  $q \in \W^{s}(p)$, $d^{s}(p,q) \leq L$ and $|T_{xy}| \leq T$ for some $T$ independent of $x$ and $y$. There is then a $j = j(L) > 0$ such that $d^{s}(f^{j}p,f^{j}q) \leq 1$ which implies that $d^{cs}(f^{j}p,f^{j}q) \leq 1$. By replacing $r$ with $\min\{r,e^{-aj}\}$ we can assume that $e^{aj}r \leq 1$. We can then apply Lemma \ref{Hamenstadt holonomy} (with $R = 1$) to conclude (using $j = j(L)$) that there is a constant $C = C(L)$ such that the $cs$-holonomy $h^{cs}:B_{\rho}(f^{j}p,e^{aj}r) \rightarrow \W^{u}(f^{j}y)$ is $C$-biLipschitz in the Hamenst\"adt metrics. Applying $f^{-j}$, we conclude that the holonomy  $h^{cs}:B_{\rho}(p,r) \rightarrow B_{\rho}(y,1)$ is $C'$-biLipschitz, $C' = e^{aj}C$. Furthermore the image of $p$ is a point $q$ satisfying $\rho(q,y) < \frac{1}{4}$. Thus by shrinking $r$ further (by an amount that depends only on $C'$ and thus only on $L$) we can assume that the $cs$-holonomy image of $B_{\rho}(p,r)$ in $\W^{u}(y)$ lies entirely inside $B_{\rho}(y,1)$. Lastly, pre-composing the holonomy $h^{cs}$ with $f^{T_{xy}} = h^{c}$ on $f^{-T_{xy}}(B_{\rho}(p,r)) = B_{\rho}(f^{-T_{xy}},e^{-a T_{xy}}r)$ gives us a $cs$-holonomy map $h^{cs}:B_{\rho}(z,R) \rightarrow B_{\rho}(y,1)$ that is $K$-biLipschitz onto its image with $R = e^{-aT}r$ and $K = e^{aT}C'$ (using $|T_{xy}| \leq T$). By construction $R$ and $K$ only depend on the quantities $L$, $r$, and $T$, which we chose at the start of the proof to be independent of $x$ and $y$.  
\end{proof}

Lemma \ref{subball equivalence} will be used in conjunction with the following result on the uniform doubling property of the Hamenst\"adt metrics. Recall Definition \ref{defn:local doubling}.

\begin{lem}\cite[Lemma 6.2]{Bu17}\label{Ham doubling}
The metric spaces $(\W^{u}(x),\rho_{x})$ are doubling for each $x \in M$ with a doubling constant $N$ that can be taken to be independent of $x$.
\end{lem}

For $x \in M$ let $V_{s,l}(x)$ be the cardinality of a maximal $e^{-as}$-separated subset of $B_{\rho}(x,e^{-al})$, $l \leq s$. For $t \in \R$ we then have by \eqref{conformal scaling},
\begin{equation}\label{separated scale}
V_{s+t,l+t}(f^{-t}x) = V_{s,l}(x). 
\end{equation}
We set
\begin{equation}\label{max separated}
V_{s,l} = \sup_{x \in M} V_{s,l}(x), 
\end{equation}
and observe from \eqref{separated scale} that 
\begin{equation}\label{max separated scale}
V_{s+t,l+t} = V_{s,l},
\end{equation} 
for $t \in \R$. In the case $l=0$ we write $V_{s}(x):=V_{s,0}(x)$, $V_{s}:=V_{s,0}$. Note this implies $s \geq 0$.

In what follows our convention for $\asymp$, $\ls$, etc. is that the implied constants are independent of the point $x \in M$ as well as any parameters on which the quantities to be compared depend (such as $s$ and $l$ for $V_{s,l}$).

\begin{lem}\label{independent cardinality}
For any $x,y \in M$ and any $s \geq 0$ we have $V_{s}(x) \asymp V_{s}(y)$ and therefore $V_{s}(x) \asymp V_{s}$.
\end{lem}

\begin{proof}
It suffices to show that $V_{s}(x) \ls V_{s}(y)$ since the roles of $x$ and $y$ are symmetric. By Lemma \ref{subball equivalence} we can find a uniform radius $r > 0$ and constant $K \geq 1$ such that there is a point $p \in B_{\rho}(x,1)$ and $K$-biLipschitz $cs$-holonomy map $h^{cs}: B_{\rho}(p,r) \rightarrow B_{\rho}(y,1)$ where $B_{\rho}(p,r) \subset B_{\rho}(x,1)$. Writing $r = e^{-al}$ for $l = -a^{-1}\log r$ and $k = -a^{-1}\log K$, this implies that $V_{s + k,l}(p) \leq V_{s,0}(y)$. But since the Hamenst\"adt metrics are doubling with a uniform doubling constant by Lemma \ref{Ham doubling}, we have
\begin{equation}\label{Ham doubling use}
V_{s,0}(y) \leq CV_{s,l}(y),
\end{equation} 
for any $y \in M$ and $s \geq 0$, where the constant $C$ depends only on $l$, $a$, and the doubling constant $N$. This is a consequence of iterating the doubling condition of Definition \ref{defn:local doubling} $m = m(a,l)$ times where $m$ is the minimal integer such that $e^{-a(l+1)} < 2^{-m} \leq e^{-al}$. Thus \eqref{Ham doubling use} implies that $V_{s+k,l}(p) \ls V_{s,l}(p)$.  And since $B_{\rho}(p,r) \subset B_{\rho}(x,1)$, by appealing to \eqref{Ham doubling use} again we have $V_{s,0}(x) \ls V_{s,l}(p)$. Since $r$ (and therefore $l$) is independent of $x$, $y$, and $s$, putting all of these claims together gives $V_{s,0}(x) \ls V_{s,0}(y)$. 
\end{proof}

We also have the following submultiplicativity property.

\begin{lem}\label{subadditive entropy}
The sequence $V_{s}$ is submultiplicative in $s$, i.e., for any $s, t > 0$,
\[
V_{s+t} \leq V_{s}V_{t}.
\]
\end{lem}

\begin{proof}
Let $x \in M$ be given and consider a maximal $e^{-a(s+t)}$-separated subset $P$ of $B_{\rho}(x,1)$. We then take a maximal $e^{-as}$-separated subset $Q$ of $P$ (note this is itself an $e^{-as}$-separated subset of $B_{\rho}(x,1)$). Then the balls $B_{\rho}(y,e^{-as})$ for $y \in Q$ cover $P$. Each $z \in P$ can then be assigned to some ball $B_{\rho}(y,e^{-as})$ containing it. The result is that for each $y \in Q$ we have an $e^{-a(s+t)}$-separated subset $Q_{y}$ of $B_{\rho}(y,e^{-as})$ such that each $z \in P$ belongs to some $Q_{y}$. Since $P$ has cardinality $V_{s+t}(x)$,
\begin{align*}
V_{s+t}(x) &\leq \sum_{y \in Q}\sum_{z \in Q_{y}}1 \\
&\leq \sum_{y \in Q}V_{s+t,s}(y) \\
&\leq V_{s}(x)V_{s+t,s} \\
&= V_{s}(x)V_{t},
\end{align*}
where we have used \eqref{max separated scale} in the final line. Taking the supremum over all $x \in M$ first on the right and then on the left gives the conclusion. 
\end{proof}

Define for $T \geq 0$,
\[
d^{T}(x,y) = \sup_{0 \leq t \leq T}d^{M}(f^{t}x,f^{t}y)
\]
Let $N_{T}$ be the cardinality of a maximal $L$-separated subset of $M$ in the metric $d^{T}$, $L \leq 1$; recall that we have rescaled the metric on $M$ so that any two points in $M$ with $d^{M}(x,y) \leq 1$ are contained in a foliation chart for each of the foliations $\W^{*}$, so that in particular any two points in $M$ will eventually separate to a distance $> L$ under the flow $f^{t}$. Then, as long as $L$ is sufficiently small, regardless of the choice of $L$ the topological entropy of the flow $f^{t}$ can be computed as 
\begin{equation}\label{entropy limit}
\lim_{T \rightarrow \infty} \frac{1}{T}\log N_{T} = h_{\mathrm{top}}(f).
\end{equation}
To simplify notation we set $h = h_{\mathrm{top}}(f)$. The equality \eqref{entropy limit} holds as long as $L$ is smaller than a multiple of the expansivity constant for $f^{t}$, see \cite[Theorem 4.2.19]{FH19} (one can also find there the argument for why this limit exists). We will assume a fixed choice of $L$ has been made such that the above holds. We further choose $L$ small enough that if $\rho(x,y) = 1$ for some $x \in M$, $y \in \W^{u}(x)$, then $d^{M}(x,y) \geq L$. For the existence of such an $L$ one can either argue from the compactness of $M$ or appeal to Lemma \ref{metric comparison}. Once this $L$ is selected we then see from a similar argument that there must be some uniform $l \geq 0$ (determined only by $L$) such that if $d^{M}(x,y) \geq L/2$ then $\rho(x,y) \geq e^{-al}$.

\begin{lem}\label{separated count}
We have
\begin{equation}\label{separated limit}
\lim_{s \rightarrow \infty} \frac{1}{s}\log V_{s} = \inf_{s > 0} \frac{1}{s} \log V_{s} = h.
\end{equation}
\end{lem}

\begin{proof}
By Lemma \ref{subadditive entropy} the sequence $\{\log V_{n}\}_{n \in \N}$ is subadditive in $n$. Thus by Fekete's lemma there is an $h_{*} \in \R$ such that
\[
\lim_{n \rightarrow \infty} \frac{1}{n}\log V_{n} = \inf_{n \in \N} \frac{1}{n} \log V_{n} = h_{*}.
\]
An easy application of the doubling property of Lemma \ref{Ham doubling} then shows that we can replace $n \in \N$ above with $s > 0$. It therefore suffices to show that $h_{*} = h$.

For $x \in M$ and any $T > 0$, any $e^{-aT}$-separated subset of $B_{\rho}(x,1)$ (in the Hamenst\"adt metric $\rho$) is an $L$-separated subset of $M$ in the metric $d^{T}$ due to our choice of $L$. This is because for any pair of points $y,z \in B_{\rho}(x,1)$ such that $\rho(y,z) \geq e^{-aT}$ there is some $t$ with $0 \leq t \leq T$ such that $\rho(f^{t}y,f^{t}z) = 1$ (by \eqref{conformal scaling}) and therefore $d^{M}(f^{t}y,f^{t}z) \geq L$. We conclude that $V_{T}(x) \leq N_{T}$. Taking the supremum over $x \in M$ gives $V_{T} \leq N_{T}$ for each $T > 0$, from which it follows that $h_{*} \leq h$.

On the other hand, by the compactness of $M$ we can cover $M$ by a finite number $n$ of foliation boxes $U_{i}$ for the foliation $\W^{u}$ which consist of the center point $x_{i}$ of a Hamenst\"adt metric ball $B_{i} = B_{\rho}(x_{i},1)$, which is then saturated locally by balls in $\W^{cs}$-leaves of radius $L/4$. Expressed in symbols, we have
\[
U_{i} = \bigcup_{y \in B_{i}} B^{cs}(y,L/4).
\]
Since the foliations $\W^{u}$ and $\W^{cs}$ have local product structure with respect to each other \cite[Chapter 6.2]{FH19}, it is not hard to see that $U_{i}$ is actually an open neighborhood of $x_{i}$ in $M$ as claimed. 

Since $f^{t}$ does not expand distances on $\W^{cs}$ leaves by \eqref{no expand}, we see on $U_{i}$ that $d^{T}$ restricted to each ball $B^{cs}(y,L/4)$ is simply $d^{M}$. Thus for each $x \in U_{i}$ there is a point $z \in B_{\rho}(x_{i},1)$ such that $d^{T}(x,z) \leq L/4$, independent of the value of $T$. We conclude that, for a \emph{fixed} value of $T$, if $\{y_{ij}\}$ is a maximal $L/4$-separated collection of points for $d^{T}$ in the ball $B_{\rho}(x_{i},1)$ then the $d^{T}$-balls of radius $L/4$ centered on these $y_{ij}$ cover $U_{i}$. Letting $k_{i}$ denote the cardinality of each collection $\{y_{ij}\}$ (as $j$ varies) and setting $k = \sum_{i=1}^{n}k_{i}$, we conclude that $M$ can be covered by $k$ $d^{T}$-balls of radius $L/2$. This implies that $N_{T} \leq k$, recalling that $N_{T}$ is the maximal cardinality of an $L$-separated subset of $M$ in the metric $d^{T}$. 

On the other hand, for any $x \in M$ and any $y,z \in B_{\rho}(x,1)$, if $d^{T}(y,z) \geq L/2$ then
\[
\sup_{0 \leq t \leq T}d^{M}(f^{t}y,f^{t}z) \geq \frac{L}{2}.
\]
But since $f^{t}$ uniformly expands $\W^{u}(x)$, the quantity $d^{M}(f^{t}y,f^{t}z)$ is an increasing function of $t$. Thus we conclude that
\[
d^{M}(f^{T}y,f^{T}z) \geq \frac{L}{2}
\]
and therefore by our choice of $L$ and $l$ prior to the lemma statement,
\[
\rho(f^{T}y,f^{T}z) \geq e^{-al}.
\]
This implies that $\rho(y,z) \geq e^{-a(l+T)}$. We conclude that $\{y_{ij}\}$ is an $e^{-a(l+T)}$-separated subset of $B_{\rho}(x_{i},1)$ in the metric $\rho$ and therefore has cardinality $k_{i} \leq V_{l+T}$. Summing this over all $i$ gives $N_{T} \leq n V_{l+T}$. Since $l$ and $n$ are both  independent of $T$ it directly follows that $h \leq h_{*}$, which completes the proof.
\end{proof}

It's straightforward to see from what we've done so far that each of the metric spaces $(\W^{u}(x),\rho_{x})$ has Hausdorff dimension $h/a$. We would like to take the conditionals of the measure of maximal entropy to be the $h/a$-dimensional Hausdorff measures for $\rho_{x}$ on each leaf $\W^{u}(x)$. However the limit \eqref{separated limit} does not provide enough information on its own to ensure that these Hausdorff measures have desirable properties such as being finite and positive on balls. To obtain this additional information we can use the constructions of Sections \ref{sec:uniform cone} and \ref{sec:doubling}. 

For $p \in M$ we set $\mu_{p}$ to be the Riemannian volume on the expanding cone $\X_{p} = \wt{\W}^{cu}(p)$ as described at the start of this section. Note $\mu_{p} = \mu_{q}$ if $q \in \W^{cu}(p)$. We write $\mu^{u}_{p}$ for the Riemannian volume on $\W^{u}(p)$.  For $\sigma > 0$ we define $\mu_{\sigma,p}$ as in \eqref{sigma def} using $b_{p}$. 

Lastly define for $\sigma > 0$,
\begin{equation}\label{G def}
G(\sigma) = \int_{0}^{\infty}e^{-\sigma t}V_{t}\,dt.
\end{equation}
Then $G$ is monotone decreasing in $\sigma$, hence monotone increasing as $\sigma \rightarrow h$ from above. We see from \eqref{separated limit} that $G(\sigma) < \infty$ if and only if $\sigma >  h$, and therefore $G(\sigma) \rightarrow \infty$ as $\sigma \rightarrow h$. We emphasize that this assertion requires the description $h = \inf_{t > 0} t^{-1}\log V_{t}$ from Fekete's lemma to obtain the necessary divergence $G(h) = \infty$.

For constants involving $\sigma > h$ we will treat them as implicit (more precisely, as $\asymp h$) if they remain bounded above and below away from zero as $\sigma \rightarrow h$. The quantity $G(\sigma)$ encapsulates the exceptional part of the constants that diverges to infinity as $\sigma \rightarrow h$. We remark that $G$ is the Laplace transform of the function $t \rightarrow V_{t}$.

To simplify notation in the next few claims we omit the subscript $x$ in $\mu_{\sigma} = \mu_{\sigma,p}$ and $\mu = \mu_{p}$ as well as the subscripts on the expanding cone $\X = \X_{p}$ and the height function $b = b_{p}$. The subscripts will be included in the statements of the claims for clarity. We recall the definition \eqref{cone def} of cones $\CC$ inside of the uniformization of an expanding cone. We switch from open balls and cones to closed balls and cones in this lemma and what follows as a matter of convenience for the compactness arguments we will be using. It's straightforward to see from the continuity of the $\Psi_{x}$ map that
\[
\bar{\CC}(x,r) := \overline{\CC(x,r)} =  \bigcup_{t > 0}f^{t}(\bar{B}_{\rho}(x,r)) \cup \Psi_{x}(\bar{B}_{\rho}(x,r)),
\]
as we would expect, with the closure taken in the uniformization $\X_{b}$ and $\bar{B}_{\rho}(x,r)$ being the closed ball of radius $r$ in the Hamenst\"adt metric $\rho_{x}$. For $T > 0$ we also introduce the \emph{truncated cone}
\[
\bar{\CC}_{T}(x,r) = \bigcup_{0 \leq t \leq T} f^{t}(\bar{B}_{\rho}(x,r))
\]

\begin{lem}\label{crit lemma}
For each $x \in \X_{p}$ and $\sigma > 0$ we have
\begin{equation}\label{crit estimate}
\mu_{\sigma,p}(\bar{\CC}(x,1)) \asymp e^{-\sigma b(x)}G(\sigma).
\end{equation}
The implied constants are independent of $\sigma$, and one side is infinite in \eqref{crit estimate} if and only if the other is. 
\end{lem}

\begin{proof}
Since the density used to define $\mu_{\sigma}$ in \eqref{sigma def} is constant on the leaves of the foliation $\W^{u}$ of $\X$, we have
\[
\mu_{\sigma}(\bar{\CC}_{T}(x,1)) = \int_{0}^{T}e^{-\sigma(b(x)+t)}\mu^{u}_{f^{t}x}(B_{\rho}(f^{t}x,e^{at}))\,dt.
\]
Here we are using the fact that the splitting $T\X = E^{u} \oplus E^{c}$ for an expanding cone is orthogonal and thus the Riemannian volume $\mu$ also splits orthogonally into arclength $dt$ along the flowlines of $f^{t}$ and $\mu^{u}$ on the leaves of $\W^{u}$. 

The quantity $\mu_{p}^{u}(\bar{B}_{\rho}(p,1))$ is positive for $p \in M$ and depends continuously on $p$. Thus by the compactness of $M$ we have $\mu_{p}^{u}(\bar{B}_{\rho}(p,1)) \asymp_{C} 1$ for some uniform constant $C \geq 1$ that we will take to be an implicit constant in subsequent expressions. By the same reasoning we have $\mu_{p}^{u}(\bar{B}_{\rho}(p,1/2)) \asymp 1$. 

Fix $t \geq 0$ and let $\{x_{i}\}_{i=1}^{N}$ be a maximal $1$-separated subset of $\bar{B}_{\rho}(f^{t}x,e^{at})$ and $\{y_{i}\}_{i=1}^{N'}$ a maximal $1/2$-separated subset of $\bar{B}_{\rho}(f^{t}x,\frac{1}{2}e^{at})$. Then $N \asymp V_{t} \asymp N'$ by \eqref{separated scale} and Lemma \ref{Ham doubling}  with implied constants independent of $t$, since a maximal $1$-separated subset of $\bar{B}_{\rho}(f^{t}x,e^{at})$ has the same cardinality as a maximal $e^{-at}$-separated subset of $\bar{B}_{\rho}(x,e^{al})$ (and the same with $1/2$ replacing $1$). The closed balls $\bar{B}_{\rho}(x_{i},1)$ cover $\bar{B}_{\rho}(f^{t}x,e^{at})$, while the closed balls $\bar{B}_{\rho}(y_{i},1/2)$ cover $\bar{B}_{\rho}(f^{t}x,\frac{1}{2}e^{at})$ and are contained within $\bar{B}_{\rho}(f^{t}x,e^{at})$ since $e^{at} \geq 1$. We conclude that
\[
\mu^{u}_{f^{t}x}(\bar{B}_{\rho}(f^{t}x,e^{at})) \asymp V_{t}
\]
for each $x \in M$ and $t \geq 0$. Thus
\[
\mu_{\sigma}(\bar{\CC}_{T}(x,1)) \asymp e^{-\sigma b(x)}\int_{0}^{T}e^{-\sigma t}V_{t}\,dt.
\]
As both sides are monotone increasing in $T$, we can let $T \rightarrow \infty$ to conclude that
\[
\mu_{\sigma}(\bar{\CC}(x,1)) \asymp e^{-\sigma b(x)}\int_{0}^{\infty}e^{-\sigma t}V_{t}\,dt,
\]
keeping in mind that both sides may be infinite. This gives \eqref{crit estimate}.
\end{proof}

We extend Lemma \ref{crit lemma} to cones $\bar{\CC}(x,r)$ with $r \neq 1$ as follows. We've split the cases $l \geq 0$ and $l \leq 0$ in the statement for clarity.

\begin{lem}\label{crit extension}
For each $x \in \X_{p}$, $\sigma > 0$, and $l \geq 0$ we have
\begin{equation}\label{crit estimate pos}
\mu_{\sigma,p}(\bar{\CC}(x,e^{al})) + \mu_{\sigma,p}(\bar{\CC}_{l}(f^{-l}x,1)) \asymp e^{-\sigma (b(x)-l)}G(\sigma),
\end{equation}
and for $l \leq 0$,
\begin{equation}\label{crit estimate neg}
\mu_{\sigma,p}(\bar{\CC}(x,e^{al})) - \mu_{\sigma,p}(\bar{\CC}_{-l}(x,e^{al})) \asymp e^{-\sigma (b(x)-l)}G(\sigma).
\end{equation}
The implied constants are independent of $\sigma$, and one side is infinite in each of \eqref{crit estimate pos} and \eqref{crit estimate neg} if and only if the other is. Consequently for $x \in \X_{p}$ and $l \in \R$ we have $\mu_{\sigma,p}(\bar{\CC}(x,e^{al})) < \infty$ if and only if $\sigma > h$, and $\mu_{\sigma,p}(\bar{\CC}(x,e^{al})) \rightarrow \infty$ as $\sigma \rightarrow h$. 
\end{lem}

\begin{proof}
The left sides of \eqref{crit estimate pos} and \eqref{crit estimate neg} are simply $\mu_{\sigma}(\bar{\CC}(f^{-l}x,1))$ since $\mu_{\sigma}$ assigns zero measure to each unstable leaf $\W^{u}(y)$ inside of $\X$. The desired comparisons then follow from  \eqref{crit estimate} and $b(f^{-l}x) = b(x)-l$. The final assertions follow from the properties of $G$ discussed above and the fact that the second term on the left in \eqref{crit estimate pos} and \eqref{crit estimate neg} remains bounded as $\sigma \rightarrow h$ (in fact it converges to $\mu_{h}(\bar{\CC}_{l}(f^{-l}x,1))$ in the first case and $\mu_{h}(\bar{\CC}_{-l}(x,e^{al}))$ in the second). 
\end{proof}

We are now able to prove Theorem \ref{thm:Patterson-Sullivan}. We carefully discuss the renormalization procedure first before applying it to prove Proposition \ref{Hamenstadt regular} below which will complete the proof of Theorem \ref{thm:Patterson-Sullivan}.

For $\sigma > h$, $l \in \R$, and $x \in \X$ we define the normalized measure on $\bar{\X}_{b}$,
\begin{equation}\label{renormalization}
\check{\mu}_{\sigma,l,x} = e^{\sigma l}(\mu_{\sigma}(\bar{\CC}(x,e^{al})))^{-1}\mu_{\sigma}|_{\bar{\CC}(x,e^{al})},
\end{equation}
which is supported on the closed cone $\bar{\CC}(x,e^{al})$ in $\bar{\X}_{b}$ and has total mass $e^{\sigma l}$. Fixing $x \in \X$ and $l \in \R$, by weak* compactness of probability measures with a common compact support (applied here to the probability measures $e^{-\sigma l}\check{\mu}_{\sigma,l,x}$) we can find a sequence $\sigma_{n} \rightarrow h$ such that the measures $\check{\mu}_{\sigma_{n},l,x}$ converge in the weak* topology to a measure $\theta_{l,x}$ supported on $\bar{\CC}(x,e^{al})$. Since $\mu_{\sigma}(\bar{\CC}(x,e^{al})) \rightarrow \infty$ as $\sigma \rightarrow h$ by Lemma \ref{crit extension}, we have
\begin{equation}\label{crush initial}
\check{\mu}_{\sigma,l,x}\left(\bigcup_{0 \leq t \leq T} \bar{B}_{\rho}(f^{t}x,e^{a(t+l)})\right) \rightarrow 0,
\end{equation}
as $\sigma \rightarrow h$ for each $T > 0$. Thus $\theta_{l,x}$ is in fact supported exclusively on $\bar{\CC}_{*}(x,e^{al}) = \bar{\CC}(x,e^{al}) \cap \p \X_{b}$. Furthermore, since $\check{\mu}_{\sigma,l,x}$ has total mass $e^{\sigma l}$, letting $\sigma \rightarrow h$ implies that $\theta_{l,x}$ has total mass $e^{-hl}$.

We specialize now to the case $b(x) = 0$. Observe then by Lemma \ref{crit extension} that for any $k,l \in \R$, $x \in \X$ and $y \in \W^{u}(x)$ we have
\begin{equation}\label{crit normalize}
\limsup_{\sigma \rightarrow h}\frac{e^{-\sigma k}\mu_{\sigma}(\bar{\CC}(x,e^{ak}))}{e^{-\sigma l}\mu_{\sigma}(\bar{\CC}(y,e^{al}))} \ls 1,
\end{equation}
since $G(\sigma) \rightarrow \infty$ as $\sigma \rightarrow h$ and the additional term on the left side of \eqref{crit estimate pos} and \eqref{crit estimate neg} remains bounded as $\sigma \rightarrow h$. The reverse inequality with $\liminf$ is implicit in \eqref{crit normalize} by swapping the pair $(x,k)$ with the pair $(y,l)$.  Referring back to the definition \eqref{renormalization} of $\ch{\mu}_{\sigma,l,x}$, the estimate \eqref{crit normalize} implies by taking reciprocals that if $\varphi:\bar{\X}_{b} \rightarrow \R$ is a continuous function supported on $\bar{\CC}(x,e^{ak}) \cap \bar{\CC}(y,e^{al})$ then
\begin{equation}\label{boundary limit}
\limsup_{\sigma \rightarrow h} \int_{\bar{\X}_{b}} \varphi \, d\check{\mu}_{\sigma,k,x} \ls \int_{\bar{\X}_{b}} \varphi \, d\check{\mu}_{\sigma,l,y},
\end{equation}
since on the support of $\varphi$ the measures $\check{\mu}_{\sigma,k,x}$ and $\check{\mu}_{\sigma,l,y}$ differ only by a constant multiplicative factor controlled by \eqref{crit normalize} as $\sigma \rightarrow h$. 

Now assuming that we have weak* convergence $\check{\mu}_{\sigma_{n},k,x} \rightharpoonup \theta_{k,x}$ and $\check{\mu}_{\sigma_{n'},l,y} \rightharpoonup \theta_{l,y}$ along some subsequences (these need not be the same subsequence) and using the fact the limits must be supported on $\p \X_{b}$, we arrive at the estimate
\begin{equation}\label{boundary measure}
\int_{\p \X_{b}} \varphi \, d\theta_{l,y} \asymp \int_{\p \X_{b}} \varphi \, d\theta_{k,x},
\end{equation}
which is valid whenever $\theta_{k,x}$ and $\theta_{l,y}$ are weak* limits of some subsequences of $\check{\mu}_{\sigma,k,x}$ and $\check{\mu}_{\sigma,k,y}$ as $\sigma \rightarrow h$ and $\varphi$ is supported in $\bar{\CC}(x,e^{ak}) \cap \bar{\CC}(y,e^{al})$. Since any continuous  function $\varphi_{*}$ on $\p \X_{b}$ supported in $\bar{\CC}_{*}(x,e^{ak}) \cap \bar{\CC}_{*}(y,e^{al})$ arises as the restriction of a continuous function $\varphi$ on $\bar{\X}_{b}$ supported in $\bar{\CC}(x,e^{ak}) \cap \bar{\CC}(y,e^{al})$, we conclude that \eqref{boundary measure} holds for any continuous function $\varphi: \p \X_{b} \rightarrow \R$ that is supported on  $\bar{\CC}_{*}(x,e^{ak}) \cap \bar{\CC}_{*}(y,e^{al})$. 

Let's consider \eqref{boundary measure} in the case $y = x$ and $k \leq l \in \N$.   In this case we obtain that 
\begin{equation}\label{special boundary measure}
\int_{\p \X_{b}} \varphi \, d\theta_{l,x} \asymp \int_{\p \X_{b}} \varphi \, d\theta_{k,x},
\end{equation}
for any continuous function $\varphi$ supported on $\bar{\CC}_{*}(x,e^{ak})$. We consider $\{\theta_{l,x}\}_{l \in \N}$ as a sequence of positive linear functionals on $C_{c}(\p \X_{b})$, the space of compactly supported continuous functions on $\p \X_{b}$.  We observe that this sequence is pointwise bounded: for any choice of $\varphi \in  C_{c}(\p \X_{b})$ we will have $\supp \varphi \subset \bar{\CC}_{*}(x,e^{ak})$ once $k$ is large enough and then \eqref{special boundary measure} implies the desired bound for any $k \leq l$. Thus by Tychonoff's theorem the sequence $\{\theta_{l,x}\}_{l \in \N}$ is precompact in the topology of pointwise convergence for linear functionals on $C_{c}(\p \X_{b})$. Let $\theta_{x}$ denote any limit point of this sequence, which will also be a positive linear functional on $C_{c}(\p \X_{b})$ and thus defines a Borel measure on $\p \X_{b}$ by the Riesz representation theorem. The measure $\theta_{x}$ is defined by the property that there is some sequence $n_{k} \rightarrow \infty$ such that
\[
\int_{\p \X_{b}} \varphi \,d\theta_{n_{k},x} \rightarrow \int_{\p \X_{b}} \varphi \,d\theta_{x},
\]
for every $\varphi \in C_{c}(\p \X_{b})$. For any other $y \in \W^{u}(x)$ with associated limit point $\theta_{m_{k},y} \rightarrow \theta_{y}$ we observe that for any given $\varphi \in C_{c}(\X_{b})$ we will have $\supp \varphi \subset \bar{\CC}_{*}(x,e^{an_{k}}) \cap \bar{\CC}_{*}(y,e^{am_{k}})$ once $k$ is sufficiently large. Thus \eqref{boundary measure} holds in the limit, giving us
\begin{equation}\label{limit boundary measure}
\int_{\p \X_{b}} \varphi \, d\theta_{y} \asymp \int_{\p \X_{b}} \varphi \, d\theta_{x},
\end{equation}
for any $\varphi \in C_{c}(\p \X_{b})$. 

We summarize \eqref{limit boundary measure} as follows: given two measures $\mu$ and $\nu$ on the same space we write $\mu \asymp_{K} \nu$ if $\mu$ is equivalent to $\nu$ and the Radon-Nikodym derivative satisfies $d\mu/d\nu \asymp_{K} 1$. Then we conclude from \eqref{limit boundary measure} that $\theta_{x} \asymp \theta_{y}$ for $y \in \W^{u}(x)$.

A metric measure space $(X,d,\mu)$ is \emph{Ahlfors $Q$-regular} for some exponent $Q > 0$ and constant $C \geq 1$ if for any $x \in X$ and $r > 0$ we have
\begin{equation}\label{ahlfors}
\mu(B(x,r)) \asymp_{C} r^{Q}.
\end{equation}
An Ahlfors $Q$-regular metric measure space has Hausdorff dimension $Q$ and the $Q$-Hausdorff measure on $X$ is uniformly equivalent to $\mu$. A metric space $(X,d)$ will be called Ahlfors $Q$-regular if it becomes an Ahlfors $Q$-regular metric measure space when equipped with its $Q$-Hausdorff measure. 
 
\begin{prop}\label{Hamenstadt regular}
For each $x \in M$ the Hamenst\"adt metric $\rho_{x}$ on $\W^{u}(x)$ is Ahlfors $h/a$-regular with constant $K$ independent of $x$. The $h/a$-Hausdorff measure $\nu_{x}$ satisfies
\begin{equation}\label{measure scaling}
f^{t}_{*}\nu_{x} = e^{ht}\nu_{f^{t}x},
\end{equation}
for each $x \in M$ and $t \in \R$. 
\end{prop}

\begin{proof}
For $y \in \W^{u}(x)$ the measure $\theta_{y,l}$ constructed on $\p \X_{b}$ has total mass $e^{hl}$ by construction. Since $\theta_{y,l} \asymp \theta_{x}|_{\bar{\CC}_{*}(y,e^{al})}$ by the limiting case $k \rightarrow \infty$ of \eqref{boundary measure}, we conclude that $\theta_{x}(\bar{\CC}_{*}(y,e^{al})) \asymp e^{hl}$. By the estimate \eqref{cone pre chain} with $r = e^{al}$ (note this does not require the condition on $r$ that the lemma imposes)  there is a constant $K = K(a,A) \geq 1$ such that 
\[
\bar{\CC}_{*}(y,K^{-1}e^{al}) \subset \bar{B}_{*}(\bar{y},e^{al}) \subset \bar{\CC}_{*}(y,Ke^{al}),
\]
which gives, by setting $r = e^{al}$,
\[
\theta_{x}(\bar{B}_{*}(\bar{y},r)) \asymp r^{h/a}.
\]
Since $\bar{y} \in \p \X_{b}$ is arbitrary, we conclude that the metric space $(\p \X_{b},d_{b})$ is Ahlfors $h/a$-regular with constant $C$ independent of $x$ and $h/a$-Hausdorff measure uniformly equivalent to $\theta_{x}$. Thus by using the biLipschitz identification $\Psi_{x}:\W^{u}(x) \rightarrow \p \X_{b}$ of Proposition \ref{boundary identify} we conclude that the metric space $(\W^{u}(x),\rho_{x})$ is Ahlfors $h/a$-regular, again with constant $C$ independent of $x$. The scaling relationship \eqref{measure scaling} for the $h/a$-Hausdorff measure then follows immediately from \eqref{conformal scaling}.
\end{proof}

This completes the core of the proof of Theorem \ref{thm:Patterson-Sullivan}. As a corollary of Proposition \ref{Hamenstadt regular} we obtain the multiplicative asymptotic
\begin{equation}\label{mult estimate}
V_{t} \asymp e^{ht}.
\end{equation}
Thus for $\sigma > h$,
\[
G(\sigma) \asymp \int_{0}^{\infty}e^{(h-\sigma)t}\,dt = \frac{1}{\sigma-h},
\]
which shows for $v > 0$ small that $G(h+v) \asymp v^{-1}$, giving a multiplicative rate of divergence for $G(\sigma)$ to $\infty$ as $\sigma \rightarrow h$ that imposes a corresponding rate estimate on the $\mu_{\sigma}$-measure of the cones $\CC$ in Lemmas \ref{crit lemma} and \ref{crit extension}.

Since the $cs$-holonomy maps are locally biLipschitz in the Hamenst\"adt metrics, it is straightforward to show that the $cs$-holonomy maps are absolutely continuous with respect to the Hausdorff measures $\nu_{x}$.

\begin{prop}\label{holonomy quasi invariance}
If $U \subset \W^{u}(x)$, $V \subset \W^{u}(y)$ are open subsets with compact closure such that there is a $cs$-holonomy homeomorphism $h^{cs}:\bar{U} \rightarrow \bar{V}$ then $h^{cs}_{*}(\nu_{x}|_{U}) \asymp_{K} \nu_{y}|_{V}$ with $K = K_{R}$ depending only on $R = \sup_{x \in U} d^{cs}(x,h^{cs}(x))$ and $K_{R} \rightarrow 1$ as $R \rightarrow 0$.
\end{prop}

\begin{proof}
Set $R = \sup_{x \in U} d^{cs}(x,h^{cs}(x))$. Then by Lemma \ref{Hamenstadt holonomy} there is a constant $K = K_{R}$ (with $K_{R} \rightarrow 1$ as $R \rightarrow 0$) such that the holonomy $h^{cs}:U \rightarrow V$ is locally $K$-biLipschitz in the Hamenst\"adt metrics. Since $\nu_{x}|_{U}$ and $\nu_{y}|_{V}$ are the $h/a$-dimensional Hausdorff measures for $(U,\rho_{x})$ and $(V,\rho_{y})$ respectively, this implies that $h^{cs}_{*}(\nu_{x}|_{U}) \asymp \nu_{y}|_{V}$ with comparison constant determined explicitly by $K$.
\end{proof}

We define the \emph{Margulis measure} $m_{x}$ on $\W^{cu}(x)$ by its disintegration along flowlines,
\begin{equation}\label{Margulis construction}
dm_{x}(y) = d\nu_{f^{t}x}(y)dt, 
\end{equation}
Note that $m_{x} = m_{y}$ for $y \in \W^{cu}(x)$ and for any $t \in \R$ the scaling \eqref{measure scaling} implies the scaling
\begin{equation}\label{Margulis scaling}
f^{t}_{*}m_{x} = e^{ht}m_{f^{t}x} = e^{ht}m_{x}.
\end{equation}
We also note from \eqref{measure scaling} that $m_{x}$ is not a product measure: if we integrate in the other order then we end up with the density
\begin{equation}\label{Margulis flip}
dm_{x}(y) = e^{ht}dt\,d\nu_{x}(y), 
\end{equation}
where $y \in \W^{u}(x)$.

We next consider $s$-holonomy maps $h^{s}:\W^{cu}_{\loc}(x) \rightarrow \W^{cu}_{\loc}(y)$, $y \in \W^{s}_{\loc}(x)$. We observe that the $s$-holonomies $h^{s}:\W^{c}_{\loc}(x) \rightarrow \W^{c}_{\loc}(y)$ between orbits are $C^r$; in fact they are isometric and orientation-preserving. This is straightforward to see via our description of $\W^{cu}(x)$ (and therefore of $\W^{cs}(x)$ by taking the inverse $f^{-t}$) from the fact that $s$-holonomy is height-preserving. 

If we consider the $s$-holonomy image $h^{s}:B_{\rho}(x,r) \rightarrow \W^{cu}_{\loc}(y)$ of a Hamenst\"adt ball for $d^{s}(x,y) \leq 1$ and $0 < r \leq 1$ then the projection $P_{y}(h^{s}(B_{\rho}(x,r))$ onto $\W^{u}(y)$ along the flowlines of $f^{t}$ coincides with the $cs$-holonomy $h^{cs}:B_{\rho}(x,r) \rightarrow \W^{u}_{\loc}(y)$. Lemma \ref{Hamenstadt holonomy} implies that $h^{cs}$ is biLipschitz in the Hamenst\"adt metrics, so in particular it is injective. Thus $P_{y}:h^{s}(B_{\rho}(x,r) \rightarrow \W^{u}(y)$ is injective and so we can consider $h^{s}(B_{\rho}(x,1))$ as a graph over its projection onto $\W^{u}(y)$.

With these properties noted we can prove the following proposition.

\begin{prop}\label{holonomy invariance}
If $U \subset \W^{cu}(x)$, $V \subset \W^{cu}(y)$ are open subsets such that there is a stable holonomy homeomorphism $h^{s}:\bar{U} \rightarrow \bar{V}$ then $h^{s}_{*}(m_{x}|_{U}) = m_{y}|_{V}$. 
\end{prop}

\begin{proof}
We begin with the special case 
\begin{equation}\label{U define}
U = \bigcup_{t \in (-v,v)} f^{t}(B_{\rho}(x,r)),
\end{equation}
with $0 < r,v \leq 1$ , and such that $d^{s}(z,h^{s}(z)) \leq R$ for some $0 < R \leq 1$ and all $z \in U$. By the discussion prior to the proposition statement the $s$-holonomy image $V$ of $U$ has the form
\[
V = \bigcup_{t \in (-v,v)} f^{t}(h^{s}(B_{\rho}(x,r))),
\]
since $s$-holonomy is isometric on flowlines. Letting $Q = P_{y}(h^{s}(B_{\rho}(x,r)))$, we see that there is a continuous function $\psi:Q \rightarrow \W^{c}(y)$ such that we can write $V$ as the region between translates of the graph of $\psi$,
\[
V = \{(z,t) \in \W^{cu}_{\loc}(y): z \in Q, \psi(z) - v \leq t \leq \psi(z) + v\},
\]
where the local coordinates are given by the local product structure of $\W^{u}$ and $\W^{c}$ within $\W^{cu}(y)$, so we take $z \in \W^{u}_{\loc}(y)$, $t \in \W^{c}_{\loc}(y) \cong [-T,T]$. Since the projection $P_{y} = f^{-t}:\W^{u}(f^{t}y) \rightarrow \W^{u}(y)$ sends the measure $\nu_{f^{t}y}$ to $e^{-ht}\nu_{y}$, as in \eqref{Margulis flip} we can compute the $m_{y}$-measure of $V$ by
\begin{align*}
m_{y}(V) &= \int_{Q}\int_{\psi(z) - v }^{\psi(z) + v}e^{ht}\,dt d\nu_{y}(z) \\
&= h^{-1}\int_{Q}e^{h(\psi(z)+v)}-e^{h(\psi(z)-v)}\, d\nu_{y}(z).
\end{align*}
On the other hand the same calculation for $U$ yields
\[
m_{x}(U) = h^{-1}(e^{hv}-e^{-hv})\nu_{x}(B_{\rho}(x,r)).
\]
By a similar argument to the one used at the end of Lemma \ref{Hamenstadt holonomy}, we see that as $R \rightarrow 0$ the function $\psi$ in this configuration must converge uniformly to the zero map $\psi(z) = 0$. Likewise from Proposition \ref{holonomy quasi invariance} we see that as $R \rightarrow 0$ the $\nu_{y}$-measure of the projection $Q$ will converge uniformly to $\nu_{x}(B_{\rho}(x,r))$. Putting these facts together with the above formulas, we see that given any $\e > 0$ there is a $\delta > 0$ such that if $R < \delta$ then
\[
m_{x}(U) \asymp_{1+\e} m_{y}(V),
\]
where $U$ has the form \eqref{U define} for some $0 < r,v \leq 1$ and we have $d^{s}(z,h^{s}(z)) \leq R$ for all $z \in U$. Since sets $U$ of the form \eqref{U define} form a neighborhood basis of each point in any open subset $W \subset \W^{cu}(x)$ we conclude that if $\e > 0$ is given and $d^{s}(z,h^{s}(z)) \leq R$ for some $R < \delta$ and all $z \in W$ then 
\begin{equation}\label{proto invariance}
h^{s}_{*}(m_{x}|_{W}) \asymp_{1+\e} m_{y}|_{h^{s}(W)}. 
\end{equation}

Now let $U$ and $V$ be given as in the proposition statement. It's clear that it suffices to prove the claim when $U$ has compact closure in $\W^{cu}(x)$, as we can extend to the general case by exhausting $U$ by open subsets $U_{i}$ with compact closure in $U$. Then $L = \sup_{z \in U}d^{s}(z,h^{s}(z))$ is finite, so given any $\e > 0$ we can choose $\delta > 0$ such that \eqref{proto invariance} holds for distances $R < \delta$. Then, since $f^{t}$ uniformly contracts the stable foliation $\W^{s}$, we can choose $T > 0$ large enough that 
\[
\sup_{z \in U}d^{s}(f^{T}(z),f^{T}(h^{s}(z))) = \sup_{z \in U}d^{s}(f^{T}(z),h^{s}(f^{T}(z))) < \delta.
\]
Thus we conclude from \eqref{proto invariance} that
\[
h^{s}_{*}(m_{f^{T}x}|_{f^{T}(U)}) \asymp_{1+\e} m_{f^{T}y}|_{f^{T}(V)}.
\]
But by \eqref{Margulis scaling} we then have
\[
h^{s}_{*}(m_{x}|_{U}) \asymp_{1+\e} m_{y}|_{V}.
\]
Since this holds for any $\e > 0$ we can let $\e \rightarrow 0$ to conclude the proof. 
\end{proof}

Having established these properties of the Margulis measures $m_{x}$, $x \in M$, we can label them as $m^{u}_{x}$ and $\nu^{u}_{x}$, then swap $f^{-t}$ for $f^{t}$ to get corresponding measures $m^{s}_{x}$ and $\nu^{s}_{x}$ on $\W^{cs}(x)$ and $\W^{s}(x)$ with scaling laws $f^{t}_{*}m^{s}_{x} = e^{-ht}m_{f^{t}x}$ and $f^{t}_{*}\nu^{s}_{x} = e^{-ht}\nu_{f^{t}x}^{s}$. Here we use the fact that the inverse $f^{-t}$ of a flow $f^{t}$ has the same topological entropy $h_{\mathrm{top}}(f) = h_{\mathrm{top}}(f^{-1})$. Since these measures have the exact same properties as the Margulis measures in Margulis' construction, we can build the measure of maximal entropy locally as a product measure by standard methods; for instance one can follow \cite[Chapter 8.6]{FH19} starting from Lemma 8.6.17. As this construction is well-known, we will not provide further details here. 

Let's return to the setting of Section \ref{sec:doubling} with the measures $\mu_{\sigma,p}$ on the expanding cones $\X_{p}$, $\sigma > h$, now with the additional knowledge that $V_{t} \asymp e^{ht}$. By the compactness of $M$ the balls $(B^{cu}(x,1),d^{cu}_{x})$ for $x \in M$ equipped with the Riemannian leaf metric on $\W^{cu}(x)$ are all $L$-biLipschitz to the standard Euclidean unit ball in $\R^{n}$ for an $L$ that's independent of $x$. Thus by the discussion prior to the statement of Lemma \ref{Whitney Poincare}, we conclude that the metric measure spaces $(\X_{p},d_{p},\mu_{p})$ are uniformly locally doubling and support a uniformly local Poincar\'e inequality with constants and radius independent of $p$. With this we can prove Theorem \ref{thm:Anosov Poincare}.

\begin{proof}[Proof of Theorem \ref{thm:Anosov Poincare}]
For each $x \in \X_{p}$ we apply Lemma \ref{crit extension} with $l \geq 0$ chosen such that $e^{al} = L'$, where $L' = L(a,A)$ is the constant of Proposition \ref{cone comparison}, so that $l = l(a,A)$. Then by \eqref{crit estimate pos} we have
\[
\mu_{\sigma,p}(\bar{\CC}(x,L)) \ls e^{-\sigma b(x)} G(\sigma),
\]
since the second term on the left in \eqref{crit estimate pos} is positive, and $G(\sigma) < \infty$ since $\sigma > h$. Since we have a uniform estimate $\mu_{x}(B^{cu}(x,1)) \asymp 1$ for all $x \in M$ by compactness, we conclude that the estimate \eqref{cone upper} holds for a uniform constant $K$ proportional to $G(\sigma)$. We thus conclude by Proposition \ref{cone comparison} that the metric spaces $(\bar{\X}_{b,p},d_{b,p},\mu_{\sigma,p})$ are doubling with a uniform doubling constant. It then follows from Proposition \ref{global Poincare} that $(\bar{\X}_{b,p},d_{p},\mu_{\sigma,p})$ supports a Poincar\'e inequality with constants depending only on $\sigma$ through $G(\sigma)$. 
\end{proof}

The conclusions of Theorem \ref{thm:Anosov Poincare} cannot be extended to $\sigma = h$. In fact Proposition \ref{ball comparison} shows that $\mu_{h}(B_{b}(\xi,r)) = \infty$ for any $\xi \in \p \X_{b}$ and $r > 0$ since $\mu_{h}(\CC_{\infty}(x,r)) = \infty$ for any $x \in \X_{p}$ and $r > 0$ by the multiplicative asymptotic \eqref{mult estimate}. Thus $\mu_{h}$ assigns infinite measure to any ball centered at a boundary point of $\bar{\X}_{b}$, which prevents $\mu_{h}$ from satisfying any sort of doubling condition on $\bar{\X}_{b}$ or even $\X_{b}$.

\bibliographystyle{plain}
\bibliography{ExtensionTrace}

\end{document}